\documentclass[12pt,a4paper]{article}

\usepackage[english]{babel}
\usepackage{amsmath}
\usepackage{amssymb}
\usepackage{amsthm}
\usepackage{mathdots}
\usepackage{mathabx}
\usepackage{mathtools}
\usepackage{tikz}
\usepackage{tikz-cd}
\usepackage{enumitem}
\usepackage[whole]{bxcjkjatype}
\usepackage{csquotes}
\usepackage{changepage}
\usepackage{color}
\definecolor{linkcolor}{rgb}{0,0,0.6}
\usepackage[colorlinks=true,
pdfstartview=FitV,
linkcolor= linkcolor,
citecolor= linkcolor,
urlcolor= linkcolor,
hyperindex=true,
hyperfigures=false]
{hyperref}

\def\MYTITLE{On the cohomology of the ramified PEL unitary RZ space of signature $(1,n-1)$}
\title{On the cohomology of the ramified PEL unitary Rapoport-Zink space of signature $(1,n-1)$}
\author{J.Muller}

\usepackage{fancyhdr} 

\makeatletter
\renewcommand\tableofcontents{%
  \null\hfill\textbf{\Large\contentsname}\hfill\null\par
  \@mkboth{\MakeUppercase\contentsname}{\MakeUppercase\contentsname}%
  \@starttoc{toc}%
}
\makeatother

\usepackage[
backend=bibtex,
style=alphabetic,
]{biblatex}
\usepackage{chngcntr}
\counterwithin*{paragraph}{section}
\counterwithin*{paragraph}{subsection}

\newcommand{\zarabic}[1]{\ifnum\value{#1}=0 \else.\arabic{#1}\fi}

\usepackage[explicit]{titlesec}
\titlespacing{\paragraph}{0pt}{10pt}{.1em}[]

\usepackage{ytableau}
\usepackage{scalerel}

\begin{document}

\fontsize{12pt}{17pt}\selectfont

\setlength{\parindent}{2em} 
\setlength{\parskip}{0.5em}

\newtheorem*{theo}{Theorem}
\newtheorem*{prop}{Proposition}
\newtheorem*{lem}{Lemma}
\newtheorem*{corol}{Corollary}
\newtheorem*{conj}{Conjecture}

\theoremstyle{remark}
\newtheorem*{rk}{Remark}
\newtheorem*{rks}{Remarks}
\newtheorem*{ex}{Example}

\theoremstyle{definition}
\newtheorem*{defi}{Definition}
\newtheorem*{notation}{Notation}

\newcommand{\lcal}{\mathcal{L}}


\maketitle

\begin{center}

\parbox{15cm}{\small
\textbf{Abstract} : \it In this paper, we study the cohomology of the ramified PEL unitary Rapoport-Zink space of signature $(1,n-1)$ by using the Bruhat-Tits stratification on its special fiber. As such, we apply the same method that we developped for the unramified case in two previous papers. More precisely, we first investigate the cohomology of a given closed Bruhat-Tits stratum. It is isomorphic to a generalized Deligne-Lusztig variety which is in general not smooth, and is associated to a finite group of symplectic similitudes. We determine the weights of the Frobenius and most of the unipotent representations occuring in its cohomology. This computation involves the spectral sequence associated to a stratification by classical Deligne-Lusztig varieties, which are parabolically induced from Coxeter varieties of smaller symplectic groups. In particular, all the unipotent representations contribute to only two cuspidal series. Then, we introduce the analytical tubes of the closed Bruhat-Tits strata, which give an open cover of the generic fiber of the Rapoport-Zink space. Using the associated \v{C}ech spectral sequence, we prove that certain cohomology groups of the Rapoport-Zink space at hyperspecial level fail to be admissible if $n$ is large enough. Eventually, when $n=2$ in the split case, when $n=3$ and when $n=4$ in the non-split case, we give a complete description of the cohomology of the supersingular locus of the associated Shimura variety at hyperspecial level, in terms of automorphic representations. In particular, certain automorphic representations occur with a multiplicity depending on $p$. Thus, our computations in the ramified case recover all the main features of the unramified case, despite new technical difficulties due to the closed Bruhat-Tits strata not being smooth.}

\vspace{0.5cm}
\end{center}

\tableofcontents

\vspace{1.5cm}

\noindent \textbf{\textsc{Introduction:}} Rapoport-Zink spaces are moduli spaces classifying the deformations of some $p$-divisible group $\mathbb X$ equipped with additional structures, called a framing object. It is a formal scheme $\mathcal M = \mathcal M_{\mathbb X}$ to which one can associate a projective system $\mathcal M_{\infty} = (\mathcal M_{K})_{K}$ of analytic spaces called the Rapoport-Zink tower. It is equipped with compatible actions of two $p$-adic groups $G = G(\mathbb Q_p)$ and $J = J(\mathbb Q_p)$, with $J$ being an inner form of some Levi complement of $G$. The cohomology of the tower $\mathcal M_{\infty}$ with coefficients in $\overline{\mathbb Q_{\ell}}$ for $\ell \not = p$ is naturally a representation of $G\times J \times W$ where $W$ is the Weil group of the underlying reflex field $E$, which is a $p$-adic local field. These cohomology groups are believed to play a role in local Langlands correspondence. So far, relatively little is known on the cohomology of $\mathcal M_{\infty}$ in general. It has been computed entirely in the Lubin-Tate case in \cite{boyer2}, whose results have been used in \cite{dat} to deduce the case of the Drinfeld space and compute the action of the monodrony. Both the Lubin-Tate and the Drinfeld cases correspond to Rapoport-Zink spaces of EL type, and their particuliar geometry allowed explicit computations. Aside from this, the Kottwitz conjecture describes the part of the cohomology of $\mathcal M_{\infty}$ which is supercuspidal both for $G$ and $J$. This has first been proved in the Lubin-Tate case in \cite{boyer99} and \cite{harris}. It has been generalized to all unramified Rapoport-Zink spaces of EL type in \cite{fargues} and \cite{shin}, and it has been proved more recently in the case of the unramified PEL unitary Rapoport-Zink space with signature $(r,n-r)$ with $n$ odd in \cite{nguyen} and \cite{nguyenmeli}.\\
One would like to obtain more information on the cohomology of general Rapoport-Zink spaces outside of the supercuspidal part, but this may remain out of reach unless we have a good understanding of the geometry of $\mathcal M$. In \cite{GHNhodge} and \cite{GHNcoxeter}, the authors determined the complete list of all choices of the framing object $\mathbb X$ so that the resulting Rapoport-Zink space $\mathcal M$ exhibits a Bruhat-Tits stratification on its reduced special fiber $\mathcal M_{\mathrm{red}}$. The resulting Bruhat-Tits strata are naturally isomorphic to Deligne-Lusztig varieties which, in the most favorable cases, are of Coxeter type. Classical Deligne-Lusztig theory is a field of mathematics giving a classification of all the irreducible complex representations of finite groups of Lie type. In their foundational paper \cite{dl}, the authors use the cohomology of Deligne-Lusztig varieties to define new induction and restriction functors, allowing them to build all such representations.\\
The present paper is a contribution to a program aiming at making use of Deligne-Lusztig theory in order to access new information on the cohomology of the Rapoport-Zink space. In \cite{mullerBT} and \cite{muller}, we explored this idea with the unramified PEL unitary Rapoport-Zink space of signature $(1,n-1)$. Historically, this Rapoport-Zink space was the first on which the Bruhat-Tits stratification has been built and studied in \cite{vw1} and \cite{vw2}. In this case, the closed Bruhat-Tits strata are projective and smooth, and their cohomology was entirely computable. We used it to study the cohomology of the maximal level of the tower, ie. the cohomology of the Berkovich generic fiber $\mathcal M^{\mathrm{an}}$, as a representation of $J\times W$. We proved that some of these cohomology groups fail to be $J$-admissible for any $n\geq 3$, and we used our results to entirely compute the cohomology of the supersingular locus of the corresponding Shimura variety when $n=3$ or $4$. In the present paper, we now consider the ramified case for which the Bruhat-Tits stratification was described in \cite{RTW}. We reach very similar results despite new technical difficulties arising from the fact that the closed Bruhat-Tits strata are not smooth anymore, so that our method fails to encapsulate a certain part of its cohomology. Let us explain this in more details.\\

\noindent In the ramified case, the reflex field $E$ is a quadratic ramified extension of $\mathbb Q_p$ with $p>2$. When $n$ is odd there is only one choice of framing object $\mathbb X$, but when $n$ is even there are two choices $\mathbb X^+$ and $\mathbb X^-$. The group $J$ is isomorphic to the group of unitary similitudes of a non-degenerate $E/\mathbb Q_p$-hermitian space $C$ of dimension $n$, which is closely related to the Dieudonné isocristal of $\mathbb X$. If $n$ is even, the hermitian space $C$ is split if $\mathbb X = \mathbb X^{+}$ and non-split if $\mathbb X = \mathbb X^+$. The group $J$ is quasi-split if and only if $n$ is odd or $n$ is even with $C$ split. In \cite{RTW}, the authors build the Bruhat-Tits stratification $\mathcal M_{\mathrm{red}} = \bigsqcup_{\Lambda} \mathcal M_{\Lambda}^{\circ}$ where $\Lambda$ runs over the set of vertex lattices $\mathcal L$ in $C$ and where $\mathcal M_{\Lambda}^{\circ}$ is isomorphic to the Coxeter variety associated to the finite group of symplectic similitudes $\mathrm{GSp}(2\theta,\mathbb F_p)$, where $0 \leq t(\Lambda) := 2\theta \leq n$ is the type of the vertex lattice $\Lambda$. The set of vertex lattices $\mathcal L$ forms a polysimplicial complex which is closely related to the Bruhat-Tits building of $J$, giving a good combinatorial description of the incidence relations of the strata. Let $\mathcal M_{\Lambda} := \overline{\mathcal M_{\Lambda}^{\circ}}$ denote the closure of a Bruhat-Tits stratum. Then $\mathcal M_{\Lambda}$ is a projective normal variety over $\mathbb F_p$ which is not smooth as soon as $\theta \geq 2$. It is isomorphic to the projective closure $S_{\theta}$ of a generalized Deligne-Lusztig variety for $\mathrm{GSp}(2\theta,\mathbb F_p)$, the kind of which does not fall under the scope of classical Deligne-Lusztig theory. However, the variety $S_{\theta}$ admits a stratification 
$$S_{\theta} = \bigsqcup_{\theta'=0}^{\theta} X_{I_{\theta'}}(w_{\theta'})$$
(notations of \ref{Stratification}) such that the closure of a stratum $X_{I_{\theta'}}(w_{\theta'})$ is the union of all the smaller strata $X_{I_{t}}(w_{t})$ for $0\leq t \leq \theta'$. It turns out that each stratum $X_{I_{\theta'}}(w_{\theta'})$ is a classical Deligne-Lusztig variety, which is parabolically induced from the Coxeter variety attached to the smaller group of unitary similitudes $\mathrm{GSp}(2\theta',\mathbb F_p)$. In \cite{cox}, Lusztig has computed the cohomology of the Coxeter varieties for all classical groups. Using the combinatorical description of irreducible unipotent representations of symplectic groups in terms of Lusztig's notion of symbols, one may compute through parabolic induction the cohomology of any stratum $X_{I_{\theta'}}(w_{\theta'})$. The stratification then induces a $\mathrm{GSp}(2\theta,\mathbb F_p)\times \langle F \rangle$-equivariant spectral sequence 
$$E_1^{a,b} = \mathrm H_c^{a+b}(X_{I_a}(w_a),\overline{\mathbb Q_{\ell}}) \implies \mathrm H_c^{a+b}(S_{\theta},\overline{\mathbb Q_{\ell}}),$$
where $F$ denotes the geometric Frobenius action. The study of the weights of the Frobenius on the terms $E_1^{a,b}$ shows that the sequence degenerates on the second page, and it offers substantial information on the cohomology of $S_{\theta}$. \\
More precisely, $S_{\theta}$ has dimension $\theta$, and for $0\leq k \leq 2\theta$ the weights of the Frobenius on the cohomology group $\mathrm H_c^k(S_{\theta},\overline{\mathbb Q_{\ell}})$ form a subset of $\{p^i,-p^{j+1}\}$ for $k - \min(k,\theta) \leq i \leq k - \lceil k/2 \rceil$ and for $k - \min(k,\theta) \leq j \leq k - \lceil k/2 \rceil - 1$. Among other things, if $i,j > k-\min(k,\theta)$ then we determine the eigenspaces of the Frobenius $\mathrm H_c^k(S_{\theta},\overline{\mathbb Q_{\ell}})_{p^i}$ and $\mathrm H_c^k(S_{\theta},\overline{\mathbb Q_{\ell}})_{-p^{j+1}}$ explicitely up to at most four irreducible representations of $\mathrm{GSp}(2\theta,\mathbb F_p)$. We refer to \ref{CohomologyS} and \ref{CohomologyT} for the detailed results, and to \ref{ClassificationUnip} Theorem for the notations regarding unipotent representations in terms of symbols, as it would be too long to fit this introduction.\\
In particuliar, we note that the cohomology of $S_{\theta}$ is entirely determined if $\theta \leq 1$ since $S_0$ is a point and $S_1 \simeq \mathbb P^1$. For $\theta \geq 2$ the variety $S_{\theta}$ has singularities and for $\theta \geq 3$ the action of the Frobenius on the cohomology is not pure (for $\theta = 2$ the non-purity is undetermined). All irreducible representations of $\mathrm{GSp}(2\theta,\mathbb F_p)$ occuring in an eigenspace of $F$ for an eigenvalue of the form $p^i$ belong to the unipotent principal series, whereas those corresponding to an eigenvalue of the form $-p^{j+1}$ belong to the cuspidal series determined by the unique cuspidal unipotent representation of $\mathrm{GSp}(4,\mathbb F_p)$.\\

\noindent We then introduce the analytical tube $U_{\Lambda} \subset \mathcal M^{\mathrm{an}}$ of any closed Bruhat-Tits stratum $\mathcal M_{\Lambda}$. As we work at hyperspecial level, the associated integral model of the Shimura variety is smooth so that the nearby cycles are trivial. It allows us to identify the cohomology of $U_{\Lambda}$ and of $\mathcal M_{\Lambda}$. Let $\mathcal L^{\mathrm{max}}$ denote the subset of vertex lattices $\Lambda$ whose type $t(\Lambda)$ is maximal, ie. it is equal to 
$$t_{\mathrm{max}} = \begin{cases}
n-1 & \text{if } n \text{ is odd,}\\
n & \text{if } n \text{ is even and } C \text{ is split,}\\
n-2 & \text{if } n \text{ is even and } C \text{ is non-split.}
\end{cases}$$
We also write $t_{\mathrm{max}} = 2\theta_{\mathrm{max}}$. Let $\{\Lambda_0,\ldots ,\Lambda_{\theta_{\mathrm{max}}}\}$ be a maximal simplex in $\mathcal L$ such that $t(\Lambda_{\theta}) = 2\theta$ for all $0\leq \theta \leq \theta_{\mathrm{max}}$. Let $J_{\theta}$ denote the fixator in $J$ of the vertex lattice $\Lambda_{\theta}$. Then the $J_{\theta}$'s are maximal compact subgroups of $J$, and any maximal compact subgroup of $J$ is conjugate to one of the $J_{\theta}$'s. The collection $\{U_{\Lambda}\}_{\Lambda\in \mathcal L^{\mathrm{max}}}$ forms an open cover of the generic fiber $\mathcal M^{\mathrm{an}}$, from which we deduce the existence of a $(J\times W)$-equivariant spectral sequence (with the notations of \ref{SpectralSequenceGeneric} and \ref{SpectralSequenceGenericBis}) 
$$E_1^{a,b} = \bigoplus_{\theta = 0}^{\theta_{\mathrm{max}}} \left(\mathrm{c-Ind}_{J_{\theta}}^J \, \mathrm H_c^b(U_{\Lambda_{\theta}},\overline{\mathbb Q_{\ell}})\right)^{k_{-a+1,\theta}} \implies \mathrm H_c^{a+b}(\mathcal M^{\mathrm{an}},\overline{\mathbb Q_{\ell}}).$$
Since the non-zero terms $E_1^{a,b}$ are located in a finite strip, this sequence eventually degenerates. By looking carefully at the distribution of the Frobenius weights among the different terms, we are able to determine two terms which remain untouched in the deeper pages of the sequence. It leads to the following statement, with $\tau := (\pi\mathrm{id}, \mathrm{Frob}) \in J\times W$ where $\pi$ is a uniformizer in $E$ and $\mathrm{Frob}$ is the geometric Frobenius.

\begin{theo}[\ref{MinimalEigenvalue1} and \ref{MinimalEigenvalue2}]
There is a $J\times W$-equivariant monomorphism 
$$\mathrm{c-Ind}_{J_{\theta_{\mathrm{max}}}}^{J} \, \mathbf 1 \hookrightarrow \mathrm H_c^{2(n-1-\theta_{\mathrm{max}})}(\mathcal M^{\mathrm{an}}),$$
where on the left-hand side the inertia acts trivially and $\tau$ acts like multiplication by the scalar $p^{n-1-\theta_{\mathrm{max}}}$.\\
Assume that $n\geq 5$ or that $n=4$ with $C$ split. There is a $J\times W$-equivariant monomorphism 
$$\mathrm{c-Ind}_{J_{\theta_{\mathrm{max}}}}^{J} \, \rho_{\theta_{\mathrm{max}}} \hookrightarrow \mathrm H_c^{2(n-\theta_{\mathrm{max}})}(\mathcal M^{\mathrm{an}}),$$
where on the left-hand side the inertia acts trivially and $\tau$ acts like multiplication by the scalar $-p^{n-\theta_{\mathrm{max}}}$.
\end{theo} 

\noindent Here, $\mathbf 1$ denotes the trivial representation and $\rho_{\theta_{\mathrm{max}}}$ is the inflation of a certain irreducible unipotent representation of the finite reductive quotient $J_{\theta_{\mathrm{max}}} / J_{\theta_{\mathrm{max}}}^+$ defined in \ref{MinimalEigenvalue2}. Using type theory, one may study the behaviour of such compactly induced representations to deduce the following proposition. Here, if $V$ is any smooth representation of $J$ and $\chi$ is any smooth character of the center $\mathrm{Z}(J) \simeq E^{\times}$, we denote by $V_{\chi}$ the largest quotient of $V$ on which $\mathrm{Z}(J)$ acts through $\chi$.

\begin{prop}[\ref{IrreducibleSubquotients} and \ref{MinimalEigenvalue2}]
Let $\chi$ be an unramified character of $E^{\times}$. 
\begin{enumerate}[label=\upshape (\arabic*), topsep = 0pt]
\item If $n=1$ or $n=2$ with $C$ non-split, we have $\theta_{\mathrm{max}} = 0$ and all irreducible subquotients of $V := \mathrm{c-Ind}_{J_0}^{J} \, \mathbf 1$ are supercuspidal. In particular $V_{\chi}$ is irreducible supercuspidal.
\item If $n\geq 3$ or if $n=2$ with $C$ split, then no irreducible subquotient of $V := \mathrm{c-Ind}_{J_{\theta_{\mathrm{max}}}}^{J} \, \mathbf 1$ is supercuspidal. In this case, $V_{\chi}$ does not contain any non-zero admissible subrepresentation of $J$. 
\item If $n=4$ with $C$ split, if $n=5$ or if $n=6$ with $\mathbb X$ non-split, we have $\theta_{\mathrm{max}} = 2$ and all irreducible subquotients of $V := \mathrm{c-Ind}_{J_2}^{J} \, \rho_2$ are supercuspidal. In particular $V_{\chi}$ is irreducible supercuspidal.
\item If $n\geq 7$ or if $n=6$ with $C$ split, then no irreducible subquotient of $V := \mathrm{c-Ind}_{J_{\theta_{\mathrm{max}}}}^{J} \, \rho_{\theta_{\mathrm{max}}}$ is supercuspidal. In this case, $V_{\chi}$ does not contain any non-zero admissible subrepresentation of $J$.
\end{enumerate}
\end{prop}

\noindent In particular, we obtain the following corollary.

\begin{corol}
Let $\chi$ be any unramified character of $E^{\times}$. \\
If $n\geq 3$ or $n=2$ with $C$ split then $\mathrm H_c^{2(n-1-\theta_{\mathrm{max}})}(\mathcal M^{\mathrm{an}})_{\chi}$ is not $J$-admissible.\\
If $n\geq 7$ or $n=6$ with $C$ split then $\mathrm H_c^{2(n-\theta_{\mathrm{max}})}(\mathcal M^{\mathrm{an}})_{\chi}$ is not $J$-admissible.
\end{corol}

\noindent This non-admissibility result was already observed in the unramified case, but it does not happen in the Lubin-Tate nor the Drinfeld cases.\\

\noindent Lastly, we introduce the PEL unitary Shimura variety $\mathrm{S}_{K^p}$ of signature $(1,n-1)$ at a ramified place, which is a smooth quasi-projective scheme over $\mathcal O_E$. It is given by a Shimura datum denoted $(\mathbb G,X)$. Let $\overline{S}^{\mathrm{ss}} := \varprojlim \overline{S}^{\mathrm{ss}}_{K^p}$ denote the supersingular locus in its special fiber. Via $p$-adic uniformization, the geometry of the supersingular locus is closely related to the geometry of $\mathcal M_{\mathrm{red}}$. At the level of cohomology, the following $(\mathbb G(\mathbb A_f^p)\times W)$-equivariant spectral sequence has been constructed in \cite{fargues}
$$F_2^{a,b} = \bigoplus_{\Pi\in\mathcal A_{\xi}(I)} \mathrm{Ext}_{J}^a \left (\mathrm H_c^{2(n-1)-b}(\mathcal M^{\mathrm{an}}, \overline{\mathbb Q_{\ell}})(1-n), \Pi_p\right) \otimes \Pi^p \implies \mathrm H_c^{a+b}(\overline{S}^{\mathrm{ss}},\mathcal L_{\xi}),$$
where $I$ is a certain inner form of $\mathbb G$ such that $I_{\mathbb A_f^p} = \mathbb G_{\mathbb A_f^p}$ and $I_{\mathbb Q_p} = J$, $\xi$ is a finite dimensional irreducible algebraic $\overline{\mathbb Q_{\ell}}$-representation of $\mathbb G$ of weight $w(\xi)\in\mathbb Z$, $\mathcal L_{\xi}$ is the associated local system on the Shimura variety $\mathrm S_{K^p}$, $\mathcal A_{\xi}(I)$ is the space of all automorphic representations of $I(\mathbb A)$ of type $\xi$ at infinity, and $\mathrm H_c^{\bullet}(\overline{S}^{\mathrm{ss}},\mathcal L_{\xi}) := \varinjlim_{K^p} \mathrm H_c^{\bullet}(\overline{S}^{\mathrm{ss}}_{K^p},\mathcal L_{\xi})$. The semisimple rank of $J$ is $\theta_{\mathrm{max}}$ so that $F_2^{a,b} = 0$ as soon as $a\geq \theta_{\mathrm{max}} +1$. In particular, if $\theta_{\mathrm{max}} \leq 1$ then all the differentials are zero so that the spectral sequence already degenerates on the second page. Using our knowledge on the cohomology of the Rapoport-Zink space, one can compute all the non-zero terms $F_2^{a,b}$. We deduce the following automorphic description of the cohomology of the supersingular locus.\\
Let $X^{\mathrm{un}}(J)$ denote the set of unramified characters of $J$. The subgroup of $J$ generated by all its open compact subgroups is denoted by $J^{\circ}$. It corresponds to all the elements of $J$ whose factor of similitude is a unit. Let us fix a square root $\pi_{\ell}$ of $p$ in $\overline{\mathbb Q_{\ell}}$. If $\Pi\in \mathcal A_{\xi}(I)$, we define $\delta_{\Pi_p} := \omega_{\Pi_p}(\pi^{-1}\cdot\mathrm{id})\pi_{\ell}^{-w(\xi)} \in \overline{\mathbb Q_{\ell}}^{\times}$ where $\omega_{\Pi_p}$ is the central character of $\Pi_p$, and $\pi^{-1}\cdot\mathrm{id}$ lies in the center of $J$. For any isomorphism $\iota:\overline{\mathbb Q_{\ell}} \simeq \mathbb C$ we have $|\iota(\delta_{\Pi_p})| = 1$. Eventually, if $x\in \overline{\mathbb Q_{\ell}}^{\times}$, we denote by $\overline{\mathbb Q_{\ell}}[x]$ the $1$-dimensional representation of the Weil group $W$ where the inertia acts trivially and $\mathrm{Frob}$ acts like multiplication by the scalar $x$.

\begin{theo}[\ref{CohomologySupersingular}]
There are $G(\mathbb A_f^p) \times W$-equivariant isomorphisms
\begin{align*}
\mathrm{H}^{0}_c(\overline{S}^{\mathrm{ss}}, \overline{\mathcal L_{\xi}}) & \simeq \bigoplus_{\substack{\Pi\in\mathcal A_{\xi}(I) \\ \Pi_p \in X^{\mathrm{un}}(J)}} \Pi^p \otimes \overline{\mathbb Q_{\ell}}[\delta_{\Pi_p}\pi_{\ell}^{w(\xi)}], \\
\mathrm{H}^{2}_c(\overline{S}^{\mathrm{ss}}, \overline{\mathcal L_{\xi}}) & \simeq \bigoplus_{\substack{\Pi\in\mathcal A_{\xi}(I) \\ \Pi_p^{J_1}\not = 0}} \Pi^p \otimes \overline{\mathbb Q_{\ell}}[\delta_{\Pi_p}\pi_{\ell}^{w(\xi)+2}].
\end{align*}
Moreover, there exists a $G(\mathbb A_f^p)\times W$-subspace $V \subset \mathrm{H}^{1}_c(\overline{S}^{\mathrm{ss}}, \overline{\mathcal L_{\xi}})$ such that 
$$V \simeq \bigoplus_{\Pi\in\mathcal A_{\xi}(I)} d(\Pi_p)\Pi^p \otimes \overline{\mathbb Q_{\ell}}[\delta_{\Pi_p}\pi_{\ell}^{w(\xi)}],$$
and with quotient space isomorphic to 
$$\mathrm{H}^{1}_c(\overline{S}^{\mathrm{ss}}, \overline{\mathcal L_{\xi}})/V \simeq \bigoplus_{\substack{\Pi\in\mathcal A_{\xi}(I) \\ \Pi_p^{J_1}\not = 0\\ \dim(\Pi_p) > 1}} (\nu-1-d(\Pi_p))\Pi^p \otimes \overline{\mathbb Q_{\ell}}[\delta_{\Pi_p}\pi_{\ell}^{w(\xi)}] \oplus \bigoplus_{\substack{\Pi\in\mathcal A_{\xi}(I) \\ \Pi_p \in X^{\mathrm{un}}(J)}} (\nu-d(\Pi_p))\Pi^p \otimes \overline{\mathbb Q_{\ell}}[\delta_{\Pi_p}\pi_{\ell}^{w(\xi)}],$$
where $\nu \in \mathbb Z_{\geq 0}$ is a multiplicity given by 
$$
\nu = \begin{cases}
1 & \text{if } n = 2 \text{ and } C \text{ is split},\\
p & \text{if } n = 3,\\
p^2 & \text{if } n = 4 \text{ and } C \text{ is non-split},
\end{cases}
$$ 
and $d(\Pi_p) := \dim\mathrm{Ext}_{J}^1(\mathrm{c-Ind}_{J^{\circ}}^J\,\mathbf 1,\Pi_p)$.
\end{theo}

\noindent The apparition of the multiplicities depending on $p$ in the group $\mathrm H_c^1$ is the main feature of this computation. Such multiplicities already appeared in the unramified case. If the Shimura variety is of Kottwitz-Harris-Taylor type, the cohomology of the whole special fiber $\overline{\mathrm S} := \varinjlim_{K^p} \overline{\mathrm S_{K^p}}$ has been computed in \cite{boyer}. In particular, no such multiplicities appear in the cohomology of $\overline{\mathrm S}$. Therefore, such multiplicities are also expected to appear in the cohomology of other non basic Newton strata of the special fiber.\\

\noindent \textbf{\textsc{Organisation of the paper:}} In the first section, we prove all the statements regarding the cohomology of the closed Bruhat-Tits strata by using Deligne-Lusztig theory only. Contrary to the introduction, we work over a general finite field $\mathbb F_q$ of characteristic $p$. However, only the case $q=p$ will be relevant in the context of the Rapoport-Zink space. Also, we work with the usual symplectic group $\mathrm{Sp}$ instead of the group of symplectic similitudes $\mathrm{GSp}$, because the associated Deligne-Lusztig varieties are the same in virtue of \ref{SameDLVarieties} and \ref{SameUnipotent}.  We recall the general definition of Deligne-Lusztig varieties, and we explain the combinatorics of symbols applied to the classification of unipotent representations of finite symplectic groups. We then translate Lusztig's computation of the cohomology of the Coxeter varieties in \cite{cox} in terms of symbols, and we finally proceed to investigate the cohomology of a closed Bruhat-Tits stratum.\\
In section $2$, we introduce the Rapoport-Zink space and recall the results from \cite{RTW} where the Bruhat-Tits stratification on the special fiber is built. We detail the combinatorics of vertex lattices, and we give a formula for the number of strata contained in or containing a fixed given stratum. Eventually, we also recall the $p$-adic uniformization of the supersingular locus of the associated Shimura variety.\\
In section $3$, we move the Bruhat-Tits stratification to the generic fiber $\mathcal M^{\mathrm{an}}$ by considering the analytical tubes $U_{\Lambda}$, and we study the associated \v{C}ech spectral sequence. This section contains all our results on the cohomology of the Rapoport-Zink space.\\
In the last section, we use the $p$-adic uniformization and our knowledge acquired so far on the cohomology of the Rapoport-Zink space, in order to compute the cohomology of the supersingular locus for small values of $n$.  

\noindent \textbf{\textsc{Notations:}} Throughout the paper, we fix an integer $n\geq 1$ and an odd prime number $p$. If $k$ is a perfect field of characteristic $p$, we denote by $W(k)$ the ring of Witt vectors and by $W(k)_{\mathbb Q}$ its fraction field, which is an unramified extension of $\mathbb Q_p$. We denote by $\sigma_k: x \mapsto x^p$ the Frobenius of $\mathrm{Aut}(k/\mathbb F_p)$, and we use the same notation for its lift to $\mathrm{Aut}(W(k)_{\mathbb Q}/\mathbb Q_p)$. If $k'/k$ is a perfect field extension then $(\sigma_{k'})_{|k} = \sigma_k$, so we can remove the subscript and write $\sigma$ unambiguously instead. If $q = p^e$ is a power of $p$, we write $\mathbb F_{q}$ for the field with $q$ elements. We fix an algebraic closure $\mathbb F$ of $\mathbb F_p$.\\
We fix $\epsilon \in \mathbb Z_{p}^{\times}$ such that $-\epsilon$ is not a square in $\mathbb Z_{p}$. We define $E_1 := \mathbb Q_p[\sqrt{-p}]$ and $E_2 := \mathbb Q_p[\sqrt{\epsilon p}]$. Any quadratic ramified extension of $\mathbb Q_p$ is isomorphic to either $E_1$ or $E_2$. We will denote by $E$ either $E_1$ or $E_2$ with uniformizer $\pi$ equal to $\sqrt{-p}$ or $\sqrt{\epsilon p}$ respectively. In both cases $\pi^2$ is a uniformizer in $\mathbb Z_{p}$. We write  $\mathcal O_E$ for the ring of integers and $\kappa(E) = \mathbb F_p$ for the residue field. Let $\overline{\,\cdot\,} \in \mathrm{Gal}(E/\mathbb Q_p)$ be the non-trivial Galois involution, so that $\overline{\pi} = -\pi$. 

\noindent \textbf{\textsc{Acknowledgement:}} This paper is part of a PhD thesis under the supervision of Pascal Boyer and Naoki Imai. I am grateful for their wise guidance throughout the research.

\section{On the cohomology of a closed Bruhat-Tits stratum}

\subsection{The closed Deligne-Lusztig variety isomorphic to a closed Bruhat-Tits stratum}

\paragraph{}\label{DefDLVariety}Let $q$ be a power of $p$ and let $\mathbf G$ be a connected reductive group over $\mathbb F$, together with a split $\mathbb F_q$-structure given by a geometric Frobenius morphism $F$. For $\mathbf H$ any $F$-stable subgroup of $\mathbf G$, we write $H := \mathbf H^F$ for its group of $\mathbb F_q$-rational points. Let $(\mathbf T,\mathbf B)$ be a pair consisting of a maximal $F$-stable torus $\mathbf T$ contained in an $F$-stable Borel subgroup $\mathbf B$. Let $(\mathbf W,\mathbf S)$ be the associated Coxeter system, where $\mathbf W = \mathrm{N}_{\mathbf G}(\mathbf T)/\mathbf T$. Since the $\mathbb F_q$-structure on $\mathbf G$ is split, the Frobenius $F$ acts trivially on $\mathbf W$. For $I\subset \mathbf S$, let $\mathbf P_I, \mathbf U_I, \mathbf L_I$ be respectively the standard parabolic subgroup of type $I$, its unipotent radical and its unique Levi complement containing $\mathbf T$. Let $\mathbf W_I$ be the subgroup of $\mathbf W$ generated by $I$.\\ 
For $\mathbf P$ any parabolic subgroup of $\mathbf G$, the associated \textbf{generalized parabolic Deligne-Lusztig variety} is
$$X_{\mathbf P} := \{g\mathbf P\in \mathbf G/\mathbf P\,|\,g^{-1}F(g)\in \mathbf P F(\mathbf P)\}.$$
We say that the variety is \textbf{classical} (as opposed to generalized) when in addition the parabolic subgroup $\mathbf P$ contains an $F$-stable Levi complement. Note that $\mathbf P$ itself needs not be $F$-stable.\\
We may give an equivalent definition using the Coxeter system $(\mathbf W,\mathbf S)$. For $I\subset \mathbf S$, let ${}^I\mathbf W^{I}$ be the set of elements $w\in \mathbf W$ which are $I$-reduced-$I$. For $w\in\, ^I\mathbf W^{I}$, the associated generalized parabolic Deligne-Lusztig variety is
$$X_I(w)  := \{g\mathbf P_I\in \mathbf G/\mathbf P_I \,|\, g^{-1}F(g)\in \mathbf P_IwF(\mathbf P_I)\}.$$
The variety $X_I(w)$ is classical when $w^{-1}Iw = I$, and it is defined over $\mathbb F_q$. The dimension is given by $\dim X_I(w) = l(w)$ where $l(w)$ denotes the length of $w$ with respect to $\mathbf S$.

\paragraph{}\label{SameDLVarieties}Let $\mathbf G$ and $\mathbf G'$ be two reductive connected group over $\mathbb F$ both equipped with an $\mathbb F_q$-structure. We denote by $F$ and $F'$ the respective Frobenius morphisms. Let $f: \mathbf G\rightarrow \mathbf G'$ be an $\mathbb F_q$-isotypy, that is a homomorphism defined over $\mathbb F_q$ whose kernel is contained in the center of $\mathbf G$ and whose image contains the derived subgroup of $\mathbf G'$. Then, according to \cite{dm} proof of Proposition 11.3.8, we have $\mathbf G' = f(\mathbf G)\mathrm{Z}(\mathbf G')^{0}$, where $\mathrm{Z}(\mathbf G')^{0}$ is the connected component of unity of the center of $\mathbf G'$. Thus intersecting with $f(\mathbf G)$ defines a bijection between parabolic subgroups of $\mathbf G'$ and those of $f(\mathbf G)$. Let $\mathbf P$ be a parabolic subgroup of $\mathbf G$ and let $\mathbf P' = f(\mathbf P)\mathrm{Z}(\mathbf G')^{0}$ be the corresponding parabolic of $\mathbf G'$. Then the map $g\mathbf P \mapsto f(g\mathbf P)$ induces an isomorphism $f:X_{\mathbf P} \xrightarrow{\sim} X_{\mathbf P'}$ which is compatible with the actions of $G$ and $G'$ via $f$. Therefore $\mathbf G$ and $\mathbf G'$ generate the same Deligne-Lusztig varieties.

\paragraph{}\label{FiniteHermitianSpace} Let $\theta\geq 0$ and let $V$ be a $2\theta$-dimensional $\mathbb F_{q}$-vector space equipped with a non-degenerate symplectic form $(\cdot,\cdot):V\times V \to \mathbb F_{q}$. Fix a basis $(e_1,\ldots,e_{2\theta})$ in which $(\cdot,\cdot)$ is described by the matrix 
$$\begin{pmatrix}
0 & A_{\theta} \\
-A_{\theta} & 0
\end{pmatrix},$$
where $A_{\theta}$ denotes the matrix having $1$ on the anti-diagonal and $0$ everywhere else. If $k$ is a perfect field extension of $\mathbb F_q$, let $V_k := V \otimes_{\mathbb F_q} k$ denote the scalar extension to $k$ equipped with its induced $k$-symplectic form $(\cdot,\cdot)$. Let $\tau:V_k \xrightarrow{\sim} V_k$ denote the map $\mathrm{id}\otimes \sigma$. If $U \subset V_k$, let $U^{\perp}$ denote its orthogonal.\\
We consider the finite symplectic group $\mathrm{Sp}(V,(\cdot,\cdot)) \simeq \mathrm{Sp}(2\theta,\mathbb F_q)$. It can be identified with $G = \mathbf G^F$ where $\mathbf G$ is the symplectic group $\mathrm{Sp}(V_{\mathbb F},(\cdot,\cdot)) \simeq \mathrm{Sp}(2\theta,\mathbb F)$ and $F$ is the Frobenius raising the entries of a matrix to their $q$-th power. Let $\mathbf T \subset \mathbf G$ be the maximal torus of diagonal symplectic matrices and let $\mathbf B \subset \mathbf G$ be the Borel subgroup of upper-triangular symplectic matrices. The Weyl system of $(\mathbf T,\mathbf B)$ is identified with $(W_{\theta},\mathbf S)$ where $W_{\theta}$ is the finite Coxeter group of type $B_{\theta}$ and $\mathbf S = \{s_1,\ldots ,s_{\theta}\}$ is the set of simple reflexions. They satisfy the following relations 
\begin{align*}
s_{\theta}s_{\theta-1}s_{\theta}s_{\theta-1}  & = s_{\theta-1}s_{\theta}s_{\theta-1}s_{\theta}, & s_is_{i-1}s_i & = s_{i-1}s_is_{i-1}, & & \forall \; 2\leq i \leq \theta-1, \\
s_is_j & = s_js_i, & & & & \forall\; |i-j| \geq 2.
\end{align*} 
Concretely, the simple reflexion $s_i$ acts on $V$ by exchanging $e_i$ and $e_{i+1}$ as well as $e_{2\theta-i}$ and $e_{2\theta-i+1}$ for $1\leq i \leq \theta-1$, whereas $s_{\theta}$ exchanges $e_{\theta}$ and $e_{\theta+1}$. The Frobenius $F$ acts trivially on $W_{\theta}$.

\paragraph{}\label{IsomorphismDLVariety} We define the following subset of $\mathbf S$ 
$$I:=\{s_1,\ldots ,s_{\theta-1}\} = \mathbf S\setminus \{s_{\theta}\}.$$
We consider the generalized Deligne-Lusztig variety $X_{I}(s_{\theta})$. Since $s_{\theta}s_{\theta-1}s_{\theta} \not \in I$, it is not a classical Deligne-Lusztig variety. Let $S_{\theta} := \overline{X_{I}(s_{\theta})}$ be its closure in $\mathbf G/\mathbf P_I$. This normal projective variety occurs as a closed Bruhat-Tits stratum in the special fiber of the ramified unitary PEL Rapoport-Zink space of signature $(1,n-1)$, as established in \cite{RTW}. In loc. cit. the authors describe the geometry of $S_{\theta}$. We summarize their analysis.

\begin{prop}[\cite{RTW} 5.3, 5.4]
Let $k$ be a perfect field extension of $\mathbb F_q$. The $k$-rational points of $S_{\theta}$ are given by
$$S_{\theta}(k) \simeq \{U \subset V_k \,|\, U^{\perp} = U \text{ and }U \overset{\leq 1}{\subset} U+\tau(U)\},$$
where $\overset{\leq 1}{\subset}$ denotes an inclusion of subspaces with index at most $1$. There is a decomposition 
$$S_{\theta} = X_{I}(\mathrm{id}) \sqcup X_{I}(s_{\theta}),$$
where $X_{I}(\mathrm{id})$ is closed and of dimension $0$, and $X_{I}(s_{\theta})$ is open, dense of dimension $\theta$. They correspond respectively to points $U$ having $U = \tau(U)$ and $U \subsetneq U + \tau(U)$.\\
If $\theta\geq 2$ then $S_{\theta}$ is singular at the points of $X_{I}(\mathrm{id})$. When $\theta=1$, we have $S_1 \simeq \mathbb P^1$.
\end{prop}

\paragraph{}\label{Stratification}For $0 \leq \theta' \leq \theta$, define 
$$I_{\theta'} := \{s_1,\ldots ,s_{\theta-\theta'-1}\},$$
and $w_{\theta'} := s_{\theta+1-\theta'}\ldots s_{\theta}$. In particular $I_0 = I$, $I_{\theta-1} = I_{\theta} = \emptyset$, $w_0 = \mathrm{id}$ and $w_1 = s_{\theta}$. 

\begin{prop}[\cite{RTW} 5.5]
There is a stratification into locally closed subvarieties
$$S_{\theta} = \bigsqcup_{\theta'=0}^{\theta} X_{I_{\theta'}}(w_{\theta'}).$$
The stratum $X_{I_{\theta'}}(w_{\theta'})$ corresponds to points $U$ such that $\dim(U+\tau(U)+\ldots+\tau^{\theta'+1}(U)) = \theta + \theta'$. The closure in $S_{\theta}$ of a stratum $X_{I_{\theta'}}(w_{\theta'})$ is the union of all the strata $X_{I_{t}}(w_{t})$ for $t\leq \theta'$. The stratum $X_{I_{\theta'}}(w_{\theta'})$ is of dimension $\theta'$, and $X_{I_{\theta}}(w_{\theta})$ is open, dense and irreducible. In particular $S_{\theta}$ is irreducible.
\end{prop}

\begin{rk}
This stratification plays the role of the Ekedahl-Oort stratification $\mathcal M_{\Lambda} = \bigsqcup_{t} \mathcal M_{\Lambda}(t)$ of the closed Bruhat-Tits strata in the unramified case, see \cite{vw2}. 
\end{rk}

\paragraph{}\label{IsomorphismEOStratum} It turns out that the strata $X_{I_{\theta'}}(w_{\theta'})$ are related to Coxeter varieties for symplectic groups of smaller sizes. For $0 \leq \theta' \leq \theta$, define
$$K_{\theta'} := \{s_1,\ldots s_{\theta-\theta'-1}, s_{\theta-\theta'+1}, \ldots , s_\theta\} = \mathbf S \setminus \{s_{\theta-\theta'}\}.$$
Note that $K_0 = I_0 = I$ and $K_{\theta} = \mathbf S$. We have $I_{\theta'} \subset K_{\theta'}$ with equality if and only if $\theta'=0$.

\noindent \begin{prop}
There is an $\mathrm{Sp}(2\theta,\mathbb F_p)$-equivariant isomorphism
$$X_{I_{\theta'}}(w_{\theta'}) \simeq \mathrm{Sp}(2\theta,\mathbb F_q)/U_{K_{\theta'}} \times_{L_{K_{\theta'}}} X_{I_{\theta'}}^{\mathbf L_{K_{\theta'}}}(w_{\theta'}),$$
where $X_{I_{\theta'}}^{\mathbf L_{K_{\theta'}}}(w_{\theta'})$ is a Deligne-Lusztig variety for $\mathbf L_{K_{\theta'}}$. The zero-dimensional variety $\mathrm{Sp}(2\theta,\mathbb F_q)/U_{K_{\theta'}}$ has a left action of $\mathrm{Sp}(2\theta,\mathbb F_q)$ and a right action of $L_{K_{\theta'}}$.
\end{prop}

\begin{proof}
It is similar to \cite{mullerBT} 3.4 Proposition.
\end{proof}

\paragraph{}\label{DecompositionDLVarieties} The Levi complement $\mathbf L_{K_{\theta'}}$ is isomorphic to $\mathrm{GL}(\theta-\theta') \times \mathrm{Sp}(2\theta')$, and its Weyl group is isomorphic to $\mathfrak S_{\theta-\theta'}\times W_{\theta'}$. Via this decomposition, the permutation $w_{\theta'}$ corresponds to $\mathrm{id}\times w_{\theta'}$. The Deligne-Lusztig variety $X_{I_{\theta'}}^{\mathbf L_{K_{\theta'}}}(w_{\theta'})$ decomposes as a product 
$$X_{I_{\theta'}}^{\mathbf L_{K_{\theta'}}}(w_{\theta'}) = X_{\mathbf I_{\theta'}}^{\mathrm{GL(\theta-\theta')}}(\mathrm{id}) \times X_{\emptyset}^{\mathrm{Sp}(2\theta')}(w_{\theta'}).$$
The variety $X_{\mathbf I_{\theta'}}^{\mathrm{GL(\theta-\theta')}}(\mathrm{id})$ is just a single point, but $X_{\emptyset}^{\mathrm{Sp}(2\theta')}(w_{\theta'})$ is the Coxeter variety for the symplectic group of size $2\theta'$. Indeed, $w_{\theta'}$ is a Coxeter element, ie. the product of all the simple reflexions of the Weyl group of $\mathrm{Sp}(2\theta')$.

\subsection{Unipotent representations of the finite symplectic group}

\paragraph{}\label{SameUnipotent} Recall that a (complex) irreducible representation of a finite group of Lie type $G = \mathbf G^{F}$ is said to be \textbf{unipotent}, if it occurs in the Deligne-Lusztig induction of the trivial representation of some maximal rational torus. Equivalently, it is unipotent if it occurs in the cohomology (with coefficient in $\overline{\mathbb Q_{\ell}}$ with $\ell \not = p$) of some Deligne-Lusztig variety of the form $X_{\mathbf B}$, with $\mathbf B$ a Borel subgroup of $\mathbf G$ containing a maximal rational torus.\\
Let $\mathbf G, \mathbf G'$ and let $f:\mathbf G \to \mathbf G'$ be an $\mathbb F_q$-isotypy as in \ref{SameDLVarieties}. If $\mathbf B$ is such a Borel in $\mathbf G$, then $\mathbf B' := f(\mathbf B)\mathrm{Z}(\mathbf G')^{0}$ is such a Borel in $\mathbf B'$, and $f$ induces an isomorphism $X_{\mathbf B} \xrightarrow{\sim} X_{\mathbf B'}$ compatible with the actions. As a consequence, the map 
$$\rho \mapsto f\circ \rho$$
defines a bijection between the sets of equivalence classes of unipotent representations of $G'$ and of $G$. We will use this observation later in the case $\mathbf G = \mathrm{Sp}(2\theta)$ and $\mathbf G' = \mathrm{GSp}(2\theta)$, the symplectic group and the group of symplectic similitudes, the morphism $f$ being the inclusion. 

\paragraph{}In this section, we recall the classification of the unipotent representations of the finite symplectic groups. The underlying combinatorics is described by Lusztig's notion of symbols. Our reference is \cite{geck} Section 4.4.

\begin{defi}
Let $\theta \geq 1$ and let $d$ be an odd positive integer. The set of \textbf{symbols of rank $\theta$ and defect $d$} is 
$$\mathcal Y^1_{d,\theta} := \left\{ S = (X,Y) \,\bigg |\, \begin{array}{l}
X = (x_1,\ldots ,x_{r+d})\\
Y = (y_1,\ldots ,y_r)
\end{array}
\text{ with }x_i,y_j \in \mathbb Z_{\geq 0}, 
\begin{array}{l}
x_{i+1}-x_i \geq 1,\\
y_{j+1}-y_j \geq 1,
\end{array}
\mathrm{rk}(S) = \theta\right\} \bigg / (\text{shift}),$$
where the shift operation is defined by $\text{shift}(X,Y) := (\{0\}\sqcup(X+1),\{0\}\sqcup(Y+1))$, and where the rank of $S$ is given by
$$\mathrm{rk}(S) := \sum_{s\in S} s - \left\lfloor\frac{(\#S-1)^2}{4}\right\rfloor.$$
\end{defi}

\noindent Note that the formula defining the rank is invariant under the shift operation, therefore it is well defined. By \cite{classical}, we have $\mathrm{rk}(S) \geq \left\lfloor\frac{d^2}{4}\right\rfloor$ so in particular $\mathcal Y^1_{d,\theta}$ is empty for $d$ big enough. We write $\mathcal Y^1_{\theta}$ for the union of the $\mathcal Y^1_{d,\theta}$ with $d$ odd, this is a finite set.

\begin{ex}
In general, a symbol $S = (X,Y)$ will be written 
$$S = \setlength\arraycolsep{2pt}
\begin{pmatrix}
x_1 & \ldots & x_{r} & \ldots & x_{r+d}\\
y_1 & \ldots & y_{r} & &
\end{pmatrix}.$$
We refer to $X$ and $Y$ as the first and second rows of $S$. The $6$ elements of $\mathcal Y^1_{2}$ are given by 
\begin{align*}
\setlength\arraycolsep{2pt}
\begin{pmatrix}
2 \\
{}
\end{pmatrix},
&&
\begin{pmatrix}
0 & 1 \\
2 &
\end{pmatrix},
&&
\begin{pmatrix}
0 & 2 \\
1 &
\end{pmatrix},
&&
\begin{pmatrix}
1 & 2 \\
0 &
\end{pmatrix},
&&
\begin{pmatrix}
0 & 1 & 2\\
1 & 2 &
\end{pmatrix},
&&
\begin{pmatrix}
0 & 1 & 2 \\
  &   &
\end{pmatrix}.
\end{align*}
The last symbol has defect $3$ whereas all the other symbols have defect $1$.
\end{ex}

\paragraph{}\label{ClassificationUnip} The symbols can be used to classify the unipotent representations of the finite symplectic group. 

\begin{theo}[\cite{classical} Theorem 8.2]
There is a natural bijection between $\mathcal Y^1_{\theta}$ and the set of equivalence classes of unipotent representations of $\mathrm{Sp}(2\theta,\mathbb F_q)$.
\end{theo}

\noindent If $S \in \mathcal Y^1_{\theta}$ we write $\rho_S$ for the associated unipotent representation of $\mathrm{Sp}(2\theta,\mathbb F_q)$. The classification is done so that the symbols
\begin{align*}
\setlength\arraycolsep{2pt}
\begin{pmatrix}
\theta \\
{}
\end{pmatrix},
&&
\begin{pmatrix}
0 & \ldots & \theta - 1 & \theta \\
1 & \ldots & \theta & 
\end{pmatrix},
\end{align*}
correspond respectively to the trivial and the Steinberg representations.

\paragraph{}\label{HookFormula}Let $S = (X,Y)$ be a symbol and let $k\geq 1$. A \textbf{$k$-hook} $h$ in $S$ is an integer $z\geq k$ such that $z\in X,z-k \not \in X$ or $z\in Y,z-k \not \in Y$. A \textbf{$k$-cohook} $c$ in $S$ is an integer $z\geq k$ such that $z \in X, z-k \not \in Y$ or $z\in Y, z-k \not \in X$. The integer $k$ is referred to as the \textbf{length} of the hook $h$ or the cohook $c$, it is denoted $\ell(h)$ or $\ell(c)$. The \textbf{hook formula} gives an expression of $\dim(\rho_S)$ in terms of hooks and cohooks. 

\begin{prop}[\cite{geck} Proposition 4.4.17]
We have 
$$\dim(\rho_S) = q^{a(S)} \frac{\prod_{i=1}^{\theta} \left(q^{2i}-1\right)}{2^{b'(S)}\prod_{h}\left(q^{\ell(h)}-1\right)\prod_{c}\left(q^{\ell(c)}+1\right)},$$
where the products in the denominator run over all the hooks $h$ and all the cohooks $c$ in $S$, and the numbers $a(S)$ and $b'(S)$ are given by
\begin{align*}
a(S) = \sum_{\{s,t\}\subset S} \min(s,t) - \sum_{i\geq 1} \binom{\#S-2i}{2}, & & b'(S) = \left\lfloor\frac{\#S-1}{2}\right\rfloor - \#\left(X\cap Y\right).
\end{align*}
\end{prop}

\paragraph{}\label{CuspidalUnipotents}For $\delta \geq 0$, we define the symbol 
$$S_{\delta} := 
\begin{pmatrix}
0 & \ldots & 2\delta \\
  &        & 
\end{pmatrix} \in \mathcal Y^1_{2\delta+1,\delta(\delta+1)}.$$

\begin{defi}
The \textbf{core} of a symbol $S \in \mathcal Y^1_{d,\theta}$ is defined by $\mathrm{core}(S) := S_{\delta}$ where $d = 2\delta + 1$. We say that $S$ is $\textbf{cuspidal}$ if $S = \mathrm{core}(S)$.
\end{defi}

\begin{rk}
In general, we have $\mathrm{rk}(\mathrm{core}(S)) \leq \mathrm{rk}(S)$ with equality if and only if $S$ is cuspidal. 
\end{rk}

\noindent The next theorem states that cuspidal unipotent representations correspond to cuspidal symbols.

\begin{theo}[\cite{geck} Theorem 4.4.28]
The group $\mathrm{Sp}(2\theta,\mathbb F_q)$ admits a cuspidal unipotent representation if and only if $\theta = \delta(\delta+1)$ for some $\delta \geq 0$. When this is the case, the cuspidal unipotent representation is unique and given by $\rho_{S_{\delta}}$.
\end{theo}

\paragraph{}\label{UnipotentCuspidalSupport}The determination of the cuspidal unipotent representations leads to a description of the unipotent Harish-Chandra series. 

\begin{defi}
Let $\delta \geq 0$ such that $\theta = \delta(\delta+1) + a$ for some $a\geq 0$. We write 
$$L_{\delta} \simeq \mathrm{GL}(1,\mathbb F_q)^a \times \mathrm{Sp}(2\delta(\delta+1),\mathbb F_q)$$ 
for the block-diagonal Levi complement in $\mathrm{Sp}(2\theta,\mathbb F_q)$, with one middle block of size $2\delta(\delta+1)$ and other blocks of size $1$. We write $\rho_{\delta} := (\mathbf{1})^a \boxtimes \rho_{S_{\delta}}$, which is a cuspidal representation of $L_{\delta}$.
\end{defi}

\begin{prop}[\cite{geck} Proposition 4.4.29] 
Let $S\in \mathcal Y^{1}_{\theta,d}$. The cuspidal support of $\rho_S$ is $(L_{\delta},\rho_{\delta})$ where $d = 2\delta+1$.
\end{prop}

\noindent In particular, the defect of the symbol $S$ of rank $\theta$ classifies the unipotent Harish-Chandra series of $\mathrm{Sp}(2\theta,\mathbb F_p)$. 

\paragraph{}\label{ComputeInduction}As it will be needed later, we explain how to compute a Harish-Chandra induction of the form 
$$\mathrm R_L^G \, \mathbf 1 \boxtimes \rho_{S'},$$
where $G = \mathrm{Sp}(2\theta,\mathbb F_q)$, $L$ is a block-diagonal Levi complement of the form $L \simeq \mathrm{GL}(a,\mathbb F_q) \times \mathrm{Sp}(2\theta',\mathbb F_q)$ and $S' \in \mathcal Y^{1}_{d,\theta'}$ is a symbol.

\begin{defi}
Let $S = (X,Y) \in \mathcal Y^1_{d,\theta}$ and let $h$ be a $k$-hook of $S$ given by some integer $z$. Assume that $z \in X$ and $z-k \not \in X$ (resp. $z \in Y$ and $z-k \not \in Y$). The \textbf{leg length} of $h$ is given by the number of integers $s \in X$ (resp. $Y$) such that $z - k < s < z$.\\
Consider the symbol $S' = (X',Y')$ obtained by deleting $z$ and replacing it with $z-k$ in the same row. We say that $S'$ is obtained from $S$ by \textbf{removing a $k$-hook}, or equivalently that $S$ is obtained from $S'$ by \textbf{adding a $k$-hook}. 
\end{defi}

\begin{theo}[\cite{fong} Statement 4.B']
Let $S' = (X',Y') \in \mathcal Y^1_{d,\theta'}$. We have 
$$\mathrm R_L^G \, \mathbf 1 \boxtimes \rho_{S'} = \sum_{S} \rho_S$$
where $S$ runs over all the symbols in $\mathcal Y^{1}_{d,\theta}$ such that, for some $a_1,a_2 \geq 0$ with $a = a_1 + a_2$, $S$ is obtained from $S'$ by adding an $a_1$-hook of leg length $0$ to its first row and an $a_2$-hook of leg length $0$ to its second row.
\end{theo}

\noindent This computation is a consequence of the Howlett-Lehrer comparison theorem \cite{howlett} as well as the Pieri rule for Coxeter groups of type $B$, see \cite{geck2} 6.1.9. We will use it in concrete examples in the following sections.

\paragraph{}\label{ComputeRestriction}There is a similar rule to compute Harish-Chandra restrictions. Let $0\leq \theta' \leq \theta$ and consider the embedding $G' \hookrightarrow L \hookrightarrow G$ where $G' = \mathrm{Sp}(2\theta',\mathbb F_q), G = \mathrm{Sp}(2\theta,\mathbb F_q)$ and $L$ is the block diagonal Levi complement $\mathrm{GL}(a,\mathbb F_q) \times \mathrm{Sp}(2\theta',\mathbb F_q)$ where $a = \theta - \theta'$. We write ${}^*\mathrm R_{G'}^G$ for the composition of the Harish-Chandra restriction functor ${}^*\mathrm R_L^G$ with the usual restriction from $L$ to $G'$. 

\begin{theo}
Let $S = (X,Y) \in \mathcal Y^1_{d,\theta}$. We have 
$${}^*\mathrm R_{G'}^G \, \rho_{S} = \sum_{S'} \rho_{S'}$$
where $S'$ runs over all the symbols in $\mathcal Y^{1}_{d,\theta'}$ such that, for some $a_1,a_2 \geq 0$ with $a = a_1 + a_2$, $S'$ is obtained from $S$ by removing an $a_1$-hook of leg length $0$ to its first row and an $a_2$-hook of leg length $0$ to its second row.
\end{theo}

\subsection{The cohomology of the Coxeter variety for the symplectic group}

\paragraph{}\label{KnownResultsCoxeter}In this section we compute the cohomology of Coxeter varieties of finite symplectic groups, in terms of the classification of the unipotent characters that we recalled in \ref{ClassificationUnip}.

\begin{notation}
We write $X^k := X_{\emptyset}(\mathrm{cox})$ for the Coxeter variety attached to the symplectic group $\mathrm{Sp}(2k,\mathbb F_q)$, and $\mathrm{H}_c^{\bullet}(X^k)$ instead of $\mathrm{H}_c^{\bullet}(X^k\otimes \mathbb F,\overline{\mathbb Q_{\ell}})$ where $\ell \not = p$.
\end{notation}

\noindent We first recall known facts on the cohomology of $X^k$ from Lusztig's work.

\noindent \begin{theo}[\cite{cox}]
The following statements hold.
\begin{enumerate}[label=\upshape (\arabic*), topsep = 0pt]
\item The variety $X^k$ has dimension $k$ and is affine. The cohomology group $\mathrm{H}^i_c(X^k)$ is zero unless $k\leq i \leq 2k$.
\item The Frobenius $F$ acts in a semisimple manner on the cohomology of $X^k$. 
\item The groups $\mathrm{H}^{2k-1}_c(X^k)$ and $\mathrm{H}^{2k}_c(X^k)$ are irreducible as $\mathrm{Sp}(2k,\mathbb F_q)$-representations, and the latter is the trivial representation. The Frobenius $F$ acts with eigenvalues respectively $q^{k-1}$ and $q^{k}$.
\item The group $\mathrm H^{k+i}_c(X^k)$ for $0 \leq i \leq k-2$ is the direct sum of two eigenspaces of $F$, for the eigenvalues $q^{i}$ and $-q^{i+1}$. Each eigenspace is an irreducible unipotent representation of $\mathrm{Sp}(2k,\mathbb F_q)$.
\item The sum $\bigoplus_{i\geq 0} \mathrm H_c^i(X^k)$ is multiplicity-free as a representation of $\mathrm{Sp}(2k,\mathbb F_q)$.
\end{enumerate}
\end{theo}

\noindent In other words, there exists a uniquely determined family of pairwise distinct symbols $S^k_0,\ldots ,S^k_k$ and $T^k_0,\ldots ,T^k_{k-2}$ in $\mathcal Y^{1}_k$ such that
\begin{align*}
\forall 0\leq i \leq k-2, & & \mathrm H^{k+i}_c(X^k) & \simeq \rho_{S^k_i} \oplus \rho_{T^k_i}, \\
\forall k-1\leq i \leq k, & & \mathrm H^{k+1}_c(X^k) & \simeq \rho_{S^k_{i}}.
\end{align*} 
The representation $\rho_{S^k_i}$ (resp. $\rho_{T^k_i}$) corresponds to the eigenspace of the Frobenius $F$ on $\bigoplus_{i\geq 0} \mathrm H_c^i(X^k)$ attached to $p^i$ (resp. to $-p^{i+1}$). Moreover, we know that $\rho_{S^k_k}$ is the trivial representation, therefore 
$$\setlength\arraycolsep{2pt}
S^k_k = 
\begin{pmatrix}
k \\
{}
\end{pmatrix}.$$
Lusztig also gives a formula computing the dimension of the eigenspaces. Specializing to the case of the symplectic group, it reduces to the following statement.

\begin{prop}[\cite{cox}]
For $0\leq i \leq k$ we have 
$$\deg(\rho_{S^k_i}) = q^{(k-i)^2} \prod_{s=1}^{k-i} \frac{q^{s+i}-1}{q^s-1}\prod_{s=0}^{k-i-1} \frac{q^{s+i}+1}{q^s+1}.$$
For $0 \leq j \leq k-2$ we have 
$$\deg(\rho_{T^k_j}) = q^{(k-j-1)^2}\frac{(q^{k-1}-1)(q^k-1)}{2(q+1)}\prod_{s=1}^{k-j-2}\frac{q^{s+j}-1}{q^s-1}\prod_{s=2}^{k-j-1}\frac{q^{s+j}+1}{q^s+1}.$$
\end{prop}

\paragraph{}\label{CohomologyCoxeter}Our goal in this section is to determine the symbols $S^k_i$ and $T^k_j$ explicitly. This is done in the following proposition.

\begin{prop}
For $0\leq i \leq k$ and $0\leq j \leq k-2$, we have 
\begin{align*}
\setlength\arraycolsep{2pt}
S^k_i = 
\begin{pmatrix}
0 & \ldots & k-i-1 & \,\,\,k \\
1 & \ldots & k-i & 
\end{pmatrix},
&&
T^k_j = 
\begin{pmatrix}
0 & \ldots & k-j-3 & k-j-2 & k-j-1 & k \\
1 & \ldots & k-j-2 & & & 
\end{pmatrix}.
\end{align*}
\end{prop}

\noindent We note that the statement is coherent with the two dimension formulae that we provided earlier. That is, the degree of $\rho_{S^k_i}$ (resp. of $\rho_{T^k_j}$) computed with the hook formula \ref{HookFormula}, agrees with the dimension of the eigenspace of $p^i$ (resp. of $-p^{j+1}$) in the cohomology of $X^k$ as given in the previous paragraph.

\begin{proof} We use induction on $k\geq 0$. Since we already know that $S^k_k$ is the symbol corresponding to the trivial representation, the proposition is proved for $k=0$. Thus we may assume $k\geq 1$. We consider the block diagonal Levi complement $L \simeq \mathrm{GL}(1,\mathbb F_q) \times \mathrm{Sp}(2(k-1),\mathbb F_q)$, and we write ${}^*\mathrm R_{k-1}^{k}$ for the composition of the Harish-Chandra restriction from $\mathrm{Sp}(2k,\mathbb F_q)$ to $L$, with the usual restriction from $L$ to $\mathrm{Sp}(2(k-1),\mathbb F_q)$. As in the proof of \cite{mullerBT} 6.3 Proposition, for all $0\leq i \leq k$ we have an $\mathrm{Sp}(2(k-1),\mathbb F_q)\times \langle F \rangle$-equivariant isomorphism 
\begin{equation}\label{eq}
{}^*\mathrm R_{k-1}^{k} \left(\mathrm H^{k+i}_c(X^k)\right) \simeq \mathrm H^{k-1+i}_c(X^{k-1}) \oplus \mathrm H^{k-1+(i-1)}_c(X^{k-1})(1)\tag{$*$}.
\end{equation}
Here, $(1)$ denotes the Tate twist. This recursive formula is established by Lusztig in \cite{cox} Corollary 2.10. The right-hand side is known by induction hypothesis whereas the left-hand side can be computed using \ref{ComputeRestriction} Theorem. We establish the proposition by comparing the different eigenspaces of $F$ on both sides.\\
If $S\in \mathcal Y^1_{d,k}$ is any symbol, the restriction ${}^*\mathrm R_{k-1}^{k}\,\rho_{S}$ is the sum of all the representations $\rho_{S'}$ where $S'$ is obtained from $S$ by removing a $1$-hook from any of its rows.\\
We distinguish different cases depending on the values of $k$ and $i$.
\begin{enumerate}[label={--},noitemsep,topsep=0pt]
\item \textbf{Case }$\mathbf{k=1}$. We only need to determine $S^1_0$. For $i=0$, the right-hand side of \eqref{eq} is $\rho_{S^0_0}$ with eigenvalue $1$. Thus, the symbol $S^1_0 \in \mathcal Y^1_{1}$ has defect $1$ and admits only one $1$-hook. If we remove this hook we obtain $S^0_0$. Therefore, $S^1_0$ must be one of the two following symbols
\begin{align*}
\setlength\arraycolsep{2pt}
\begin{pmatrix}
0 & 1\\
1
\end{pmatrix},
&&
\begin{pmatrix}
1\\
{} 
\end{pmatrix}.
\end{align*}
By \ref{KnownResultsCoxeter}, we know that $\rho_{S^1_0}$ has degree $q$, thus $S^1_0$ must be equal to the former symbol.
\end{enumerate}

\noindent From now, we assume $k\geq 2$ and we determine $S^k_i$ for $0\leq i < k$. 

\begin{enumerate}[label={--},noitemsep,topsep=0pt]
\item \textbf{Case }$\mathbf{k=2}$ \textbf{ and }$\mathbf{i=0}$. The eigenspace attached to $1$ on the right-hand side of \eqref{eq} is $\rho_{S^{1}_0}$. Thus, the symbol $S^2_0 \in \mathcal Y^1_k$ has defect $1$ and admits only one $1$-hook. If we remove this hook we obtain $S^{1}_0$. Therefore, $S^2_0$ must be one of the two following symbols
\begin{align*}
\setlength\arraycolsep{2pt}
\begin{pmatrix}
0 & 1 & 2\\
1 & 2 &
\end{pmatrix},
&&
\begin{pmatrix}
0 & 1\\
2 & 
\end{pmatrix}.
\end{align*}
By \ref{KnownResultsCoxeter}, we know that $\rho_{S^2_0}$ has degree $q^4$, thus $S^2_0$ must be equal to the former symbol.

\item \textbf{Case }$\mathbf{k>2}$ \textbf{ and }$\mathbf{i=0}$. The eigenspace attached to $1$ on the right-hand side of \eqref{eq} is $\rho_{S^{k-1}_0}$. Thus, the symbol $S^k_0 \in \mathcal Y^1_k$ has defect $1$ and admits only one $1$-hook. If we remove this hook we obtain $S^{k-1}_0$. The only such symbol is
$$\setlength\arraycolsep{2pt}
S^k_0 = 
\begin{pmatrix}
0 & \ldots & k-1 & k\\
1 & \ldots & k &
\end{pmatrix}.$$
\item \textbf{Case }$\mathbf{1\leq i \leq k-1}$. The eigenspace attached to $p^i$ on the right-hand side of \eqref{eq} is $\rho_{S^{k-1}_i} \oplus \rho_{S^{k-1}_{i-1}}$. Thus, the symbol $S^k_i \in \mathcal Y^1_k$ has defect $1$ and admits only two $1$-hooks. If we remove one of these hooks we obtain either $S^{k-1}_i$ or $S^{k-1}_{i-1}$. The only such symbol is
$$\setlength\arraycolsep{2pt}
S^k_i = 
\begin{pmatrix}
0 & \ldots & k-i-1 & \,\,\,k \\
1 & \ldots & k-i & 
\end{pmatrix}.$$
\end{enumerate}

\noindent It remains to determine $T^k_j$ for $0\leq j \leq k-2$. 

\begin{enumerate}[label={--},noitemsep,topsep=0pt]
\item \textbf{Case }$\mathbf{k=2}$. The eigenspace attached to $-p$ on the right-hand side of \eqref{eq} is $0$. Thus, the symbol $T^2_0 \in \mathcal Y^1_2$ has no hook at all, implying that it is cuspidal in the sense of \ref{CuspidalUnipotents}. Since $\mathrm{Sp}(4,\mathbb F_q)$ admits only $1$ unipotent cuspidal representation, we deduce that 
$$\setlength\arraycolsep{2pt}
T^2_0 = 
\begin{pmatrix}
0 & 1 & 2 \\
& &
\end{pmatrix}.$$
\item \textbf{Case }$\mathbf{k=3}$. First when $j=0$, the eigenspace attached to $-p$ on the right-hand side of \eqref{eq} is $\rho_{T^{2}_0}$. Thus, the symbol $T^3_0 \in \mathcal Y^1_3$ has defect $3$ and admits only one $1$-hook. If we remove this hook we obtain $T^{2}_0$. Therefore, $T^3_0$ must be one of the two following symbols
\begin{align*}
\setlength\arraycolsep{2pt}
\begin{pmatrix}
0 & 1 & 2 & 3\\
1 & & &
\end{pmatrix},
&&
\begin{pmatrix}
0 & 1 & 3\\
& & 
\end{pmatrix}.
\end{align*}
By \ref{KnownResultsCoxeter}, we know that $\rho_{T^3_0}$ has degree $q^4\frac{(q^2-1)(q^3-1)}{2(q+1)}$, thus $T^3_0$ must be equal to the former symbol.\\
Then when $j=1$, the eigenspace attached to $-p^2$ on the right-hand side of \eqref{eq} is $\rho_{T^{2}_0}$. Thus, the symbol $T^3_1 \in \mathcal Y^1_3$ has defect $3$ and admits only one $1$-hook. If we remove this hook we obtain $T^{2}_0$. Thus $T^3_1$ is also one of the two symbols above. We can deduce that it is equal to the latter by comparing the dimensions or by using the fact that the symbols $T_j^k$ are pairwise distinct.
\end{enumerate}

\noindent From now, we assume $k\geq 4$ and we determine $T^k_j$ for $0\leq j \leq k-2$. 

\begin{enumerate}[label={--},noitemsep,topsep=0pt]
\item \textbf{Case }$\mathbf{k=4}$ \textbf{ and }$\mathbf{j=0}$. The eigenspace attached to $-p$ on the right-hand side of \eqref{eq} is $\rho_{T^{3}_0}$. Thus, the symbol $T^4_0 \in \mathcal Y^1_k$ has defect $3$ and admits only one $1$-hook. If we remove this hook we obtain $T^{3}_0$. Therefore, $T^4_0$ must be one of the two following symbols
\begin{align*}
\setlength\arraycolsep{2pt}
\begin{pmatrix}
0 & 1 & 2 & 3 & 4\\
1 & 2 & & &
\end{pmatrix},
&&
\begin{pmatrix}
0 & 1 & 2 & 3\\
2 & & & 
\end{pmatrix}.
\end{align*}
By \ref{KnownResultsCoxeter}, we know that $\rho_{T^4_0}$ has degree $q^9\frac{(q^3-1)(q^4-1)}{2(q+1)}$, thus $T^4_0$ must be equal to the former symbol.

\item \textbf{Case }$\mathbf{k>4}$ \textbf{ and }$\mathbf{j=0}$. The eigenspace attached to $-p$ on the right-hand side of \eqref{eq} is $\rho_{T^{k-1}_0}$. Thus, the symbol $T^k_0 \in \mathcal Y^1_k$ has defect $3$ and admits only one $1$-hook. If we remove this hook we obtain $T^{k-1}_0$. The only such symbol is
$$\setlength\arraycolsep{2pt}
T^k_0 = 
\begin{pmatrix}
0 & \ldots & k-3 & k-2 & k-1 & k\\
1 & \ldots & k-2 & & &
\end{pmatrix}.$$

\item \textbf{Case }$\mathbf{k=4}$ \textbf{ and }$\mathbf{j=k-2}$. The eigenspace attached to $-p^{3}$ on the right-hand side of \eqref{eq} is $\rho_{T^{3}_{1}}$. Thus, the symbol $T^4_{2} \in \mathcal Y^1_k$ has defect $3$ and admits only one $1$-hook. If we remove this hook we obtain $T^{3}_{1}$. Therefore, $T^4_2$ must be one of the two following symbols
\begin{align*}
\setlength\arraycolsep{2pt}
\begin{pmatrix}
0 & 1 & 4\\
& &
\end{pmatrix},
&&
\begin{pmatrix}
0 & 2 & 3\\
& &
\end{pmatrix}.
\end{align*}
By \ref{KnownResultsCoxeter}, we know that $\rho_{T^4_2}$ has degree $q\frac{(q^3-1)(q^4-1)}{2(q+1)}$, thus $T^4_2$ must be equal to the former symbol.

\item \textbf{Case }$\mathbf{k>4}$ \textbf{ and }$\mathbf{j = k-2}$. The eigenspace attached to $-p^{k-1}$ on the right-hand side of \eqref{eq} is $\rho_{T^{k-1}_{k-3}}$. Thus, the symbol $T^k_{k-2} \in \mathcal Y^1_k$ has defect $3$ and admits only one $1$-hook. If we remove this hook we obtain $T^{k-1}_{k-3}$. The only such symbol is
$$\setlength\arraycolsep{2pt}
T^k_{k-2} = 
\begin{pmatrix}
0 & 1 & k \\
& & 
\end{pmatrix}.$$

\item \textbf{Case }$\mathbf{1\leq j \leq k-3}$. The eigenspace attached to $-p^{j+1}$ on the right-hand side of \eqref{eq} is $\rho_{T^{k-1}_{j}} \oplus \rho_{T^{k-1}_{j-1}}$. Thus, the symbol $T^k_j \in \mathcal Y^1_k$ has defect $3$ and admits only two $1$-hooks. If we remove one of these hooks we obtain either $T^{k-1}_j$ or $T^{k-1}_{j-1}$. The only such symbol is
$$\setlength\arraycolsep{2pt}
T^k_j = 
\begin{pmatrix}
0 & \ldots & k-j-3 & \;k-j-2 & \;k-j-1 & \;k \\
1 & \ldots & k-j-2 & & & 
\end{pmatrix}.$$
\end{enumerate}
\end{proof}

\subsection{On the cohomology of a closed Bruhat-Tits stratum}

\paragraph{}\label{FrobeniusEigenvalues} Recall from \ref{IsomorphismDLVariety} the $\theta$-dimensional normal projective variety $S_{\theta} := \overline{X_{I}(s_{\theta})}$ defined over $\mathbb F_{q}$. It is equipped with an action of the finite symplectic group $\mathrm{Sp}(2\theta,\mathbb F_q)$. We use the stratification of \ref{Stratification} Proposition to study its cohomology over $\overline{\mathbb Q_{\ell}}$. If $\lambda$ is a scalar, we write $\mathrm{H}^{\bullet}_c(S_{\theta})_{\lambda}$ to denote the eigenspace of the Frobenius $F$ associated to $\lambda$ (we do not in principle assume the eigenspace to be non zero). We give a series of statements before proving all of them at once in the remaining of this section.

\noindent \begin{prop}
The Frobenius $F$ acts semi-simply on $\mathrm{H}^{\bullet}_c(S_{\theta})$. Its eigenvalues form a subset of 
$$\{q^i\,|\, 0\leq i \leq \theta\}\cup \{-q^{j+1} \,|\, 0\leq j \leq \theta-2\}.$$
\end{prop}

\paragraph{}\label{CohomologyS}In a first statement, we give our results regarding the eigenspaces attached to a scalar of the form $q^i$ for some $i$. Recall from \ref{UnipotentCuspidalSupport} the cuspidal supports $(L_{\delta},\rho_{\delta})$ for the finite symplectic group $\mathrm{Sp}(2\theta,\mathbb F_q)$. 

\noindent \begin{theo} 
Let $0 \leq i \leq \theta$ and $\theta' \in \mathbb Z$.  
\begin{enumerate}[label=\upshape (\arabic*), topsep = 0pt]
		\item The eigenspace $\mathrm{H}^{\theta'+i}_c(S_{\theta})_{q^i}$ is zero when $\theta' < i$ or $\theta' > \theta$.
\end{enumerate}
We now assume that $0 \leq i \leq \theta' \leq \theta$.
\begin{enumerate}[label=\upshape (\arabic*), topsep = 0pt]
		\setcounter{enumi}{1}
		\item All the irreducible representations of $\mathrm{Sp}(2\theta,\mathbb F_q)$ in the eigenspace $\mathrm H_c^{\theta'+i}(S_{\theta})_{q^i}$ belong to the unipotent principal series, ie. they have cuspidal support $(L_0,\rho_0)$. 
		\item We have 
	\begin{align*}
\setlength\arraycolsep{2pt} \mathrm H_c^{0}(S_{\theta}) = \mathrm H_c^{0}(S_{\theta})_1 \simeq \rho_{\footnotesize \begin{pmatrix}
\theta \\
{}
\end{pmatrix}}, & & 
\mathrm H_c^{2\theta}(S_{\theta}) = \mathrm H_c^{2\theta}(S_{\theta})_{q^{\theta}} \simeq \rho_{\footnotesize \begin{pmatrix}
\theta \\
{}
\end{pmatrix}}.
	\end{align*}
	\item If $i+2 \leq \theta'$ then 
\begin{equation*}
\setlength\arraycolsep{2pt} \begin{split} 
\bigoplus_{0\leq d \leq \theta - \theta'-1} \rho_{\footnotesize \begin{pmatrix}
0 & \ldots & \theta' - i - 2 & \theta' - i - 1 & \theta' + d \\
1 & \ldots & \theta' - i - 1 & \theta - i - d &
\end{pmatrix}} 
& \oplus \\ 
\bigoplus_{\substack{1\leq d \leq \\ \min(i,\theta-\theta'-1)}} & \rho_{\footnotesize \begin{pmatrix}
0 & \ldots & \theta' - i - 2 & \theta' - i - 1 + d & \theta' \\
1 & \ldots & \theta' - i - 1 & \theta - i - d &
\end{pmatrix}}
\hookrightarrow \mathrm H_c^{\theta'+i}(S_{\theta})_{q^i}.
\end{split}
\end{equation*}
The cokernel of this map consists of at most $4$ irreducible representations of $\mathrm{Sp}(2\theta,\mathbb F_q)$. 
\item When $i = \theta' \not = \theta$, we have 
\begin{align*}
\setlength\arraycolsep{2pt}
\rho_{\footnotesize \begin{pmatrix}
\theta \\
{}
\end{pmatrix}} \hookrightarrow \mathrm H_c^{2i}(S_{\theta})_{q^i} & \text{ if } 2i < \theta, & 
\rho_{\footnotesize \begin{pmatrix}
\theta \\
{}
\end{pmatrix}} \oplus 
\rho_{\footnotesize \begin{pmatrix}
\theta-i & i+1 \\
0 &
\end{pmatrix}}\hookrightarrow \mathrm H_c^{2i}(S_{\theta})_{q^i} & \text{ if } 2i \geq \theta. 
\end{align*}
\item When $\theta' = \theta$ we have 
\begin{equation*}
\setlength\arraycolsep{2pt}
\mathrm H_c^{\theta+i}(S_{\theta})_{q^i} \simeq 0 \text{ or } \rho_{\footnotesize \begin{pmatrix}
0 & \ldots & \theta-i-1 & \theta \\
1 & \ldots & \theta-i
\end{pmatrix}}.
\end{equation*}
\item When $\theta' = 1$ and $i=0$, we have 
\begin{align*}
\setlength\arraycolsep{2pt}
\mathrm H_c^1(S_1) = 0, & & \mathrm H_c^1(S_{\theta}) = \mathrm H_c^1(S_{\theta})_1 \simeq 0 \text{ or } \rho_{\footnotesize \begin{pmatrix}
0 & 1 & \theta \\
1 & 2
\end{pmatrix}} \text{ when } \theta \geq 2. 
\end{align*}
\end{enumerate}
\end{theo}

\noindent We note that when $\theta' = \theta$, the formula of $(4)$ does not say anything about the eigenspace $\mathrm H_c^{\theta+i}(S_{\theta})_{q^i}$ since the sums are empty. However, by $(6)$ we understand that this eigenspace is either $0$ either irreducible.\\
We note also that the theorem does not give any information in the case $i+1 = \theta'$, except when $\theta' = 1$ and $i=0$ which corresponds to $(7)$. 

\paragraph{}\label{CohomologyT} In a second statement, we give our results regarding the eigenspaces attached to a scalar of the form $-q^{j+1}$ for some $j$.

\noindent \begin{theo} 
Let $0 \leq j \leq \theta - 2$ and $\theta' \in \mathbb Z$.  
\begin{enumerate}[label=\upshape (\arabic*), topsep = 0pt]
		\item The eigenspace $\mathrm{H}^{\theta'+j}_c(S_{\theta})_{-q^{j+1}}$ is zero when $\theta' < j+2$ or $\theta' > \theta$.
\end{enumerate}
We now assume that $2 \leq j+2 \leq \theta' \leq \theta$.
\begin{enumerate}[label=\upshape (\arabic*), topsep = 0pt]
		\setcounter{enumi}{1}
		\item All the irreducible representations of $\mathrm{Sp}(2\theta,\mathbb F_q)$ in the eigenspace $\mathrm H_c^{\theta'+j}(S_{\theta})_{-q^{j+1}}$ are unipotent with cuspidal support $(L_1,\rho_1)$. 
		\item We have
	\begin{equation*}
\setlength\arraycolsep{2pt} \mathrm H_c^{2\theta-2}(S_{\theta})_{-q^{\theta-1}} \simeq \rho_{\footnotesize \begin{pmatrix}
0 & 1 & \theta \\
{}
\end{pmatrix}}.
	\end{equation*}
	\item If $j+4 \leq \theta' \leq \theta$ then 
\begin{equation*}
\setlength\arraycolsep{2pt} \begin{split} 
\bigoplus_{0\leq d \leq \theta - \theta'-1} & \rho_{\footnotesize \begin{pmatrix}
0 & \ldots & \theta' - i - 4 & \theta' - i - 3 & \theta' - j - 2 & \theta' - j - 1 & \theta' + d \\
1 & \ldots & \theta' - j - 3 & \theta - j - 2 - d & & &
\end{pmatrix}} 
 \oplus \\ 
\bigoplus_{\substack{1\leq d \leq \\ \min(i,\theta-\theta'-1)}} & \rho_{\footnotesize \begin{pmatrix}
0 & \ldots & \theta' - i - 4 & \theta' - i - 3 & \theta' - j - 2 & \theta' - j - 1 + d & \theta' \\
1 & \ldots & \theta' - j - 3 & \theta - j - 2 - d & & &
\end{pmatrix}}
\hookrightarrow \mathrm H_c^{\theta'+j}(S_{\theta})_{-q^{j+1}}.
\end{split}
\end{equation*}
The cokernel of this map consists of at most $4$ irreducible representations of $\mathrm{Sp}(2\theta,\mathbb F_q)$. 
\item When $j+2 = \theta' \not = \theta$, we have 
\begin{align*}
\setlength\arraycolsep{2pt}
\rho_{\footnotesize \begin{pmatrix}
0 & 1 & \theta \\
{}
\end{pmatrix}} & \hookrightarrow \mathrm H_c^{2(j+1)}(S_{\theta})_{-q^{j+1}} & \text{ if } 2(j+1) < \theta, \\ 
\rho_{\footnotesize \begin{pmatrix}
0 & 1 & \theta \\
& &
\end{pmatrix}} \oplus 
\rho_{\footnotesize \begin{pmatrix}
0 & \theta-i-1 & i+2 \\
& &
\end{pmatrix}} & \hookrightarrow \mathrm H_c^{2(j+1)}(S_{\theta})_{-q^{j+1}} & \text{ if } 2(j+1) \geq \theta. 
\end{align*}
\item When $\theta' = \theta$ we have 
\begin{equation*}
\setlength\arraycolsep{2pt}
\mathrm H_c^{\theta+j}(S_{\theta})_{-q^{j+1}} \simeq 0 \text{ or } \rho_{\footnotesize \begin{pmatrix}
0 & \ldots & \theta-j-3 & \theta - j - 2 & \theta - j - 1 & \theta \\
1 & \ldots & \theta-j-2 & & &
\end{pmatrix}}.
\end{equation*}
\end{enumerate}
\end{theo}

\noindent We note that when $\theta' = \theta$, the formula of $(4)$ does not say anything about the eigenspace $\mathrm H_c^{\theta+j}(S_{\theta})_{-q^{j+1}}$ since the sums are empty. However, by $(6)$ we understand that this eigenspace is either $0$ either irreducible.\\
We note also that the theorem does not give any information in the case $j+3 = \theta'$.

\noindent \begin{rk}
A cuspidal representation occurs in the cohomology of $S_{\theta}$ only in the cases $\theta = 0$ and $\theta = 2$. When $\theta = 0$ it corresponds to $\mathrm H^0_c(S_{0})$ which is trivial. When $\theta = 2$ it corresponds to $\mathrm H^2_c(S_2)_{-q}$ as described by $(3)$ in the theorem above.
\end{rk}

\paragraph{}\label{TermsFirstPage}The remaining of this section is dedicated to proving the theorems stated above. Recall from \ref{Stratification} that we have a stratification $S_{\theta} = \bigsqcup_{\theta'=0}^{\theta} X_{I_{\theta'}}(w_{\theta'})$. It induces a spectral sequence on the cohomology whose first page is given by 
\begin{equation}\label{spectral}
E^{a,b}_1 = \mathrm H^{a+b}_c(X_{I_a}(w_a)) \implies \mathrm H^{a+b}_c(S_{\theta}).
\tag{$E$}
\end{equation}
Now, recall that the strata $X_{I_{\theta'}}(w_{\theta'})$ are related to Coxeter varieties for the finite symplectic group $\mathrm{Sp}(2\theta',\mathbb F_q)$. Using \ref{DecompositionDLVarieties}, the geometric isomorphism given in \ref{IsomorphismEOStratum} Proposition induces an isomorphism on the cohomology
\begin{equation}\label{eq2}
\mathrm H_c^{\bullet}(X_{I_{\theta'}}(w_{\theta'})) \simeq \mathrm R_{L_{K_{\theta'}}}^{\mathrm{Sp}(2\theta,\mathbb F_q)} \,\mathbf{1} \boxtimes \mathrm H_c^{\bullet}(X^{\mathrm{Sp}(2\theta')}(w_{\theta'})),
\tag{$**$}
\end{equation}
where $L_{K_{\theta'}}$ denotes the block-diagonal Levi complement isomorphic to $\mathrm{GL}(\theta-\theta',\mathbb F_q) \times \mathrm{Sp}(2\theta',\mathbb F_q)$. The variety $X^{\mathrm{Sp}(2\theta')}(w_{\theta'})$ is nothing but the Coxeter variety that we denoted by $X^{k'}$ in \ref{KnownResultsCoxeter}, and whose cohomology we have described. For $0 \leq i \leq \theta'$ and $0 \leq j \leq \theta'-2$, recall from \ref{CohomologyCoxeter} the symbols $S_i^{\theta'}$ and $T_j^{\theta'}$. We define
\begin{align*}
\mathrm R^S_{i,\theta'} := \mathrm R_{L_{K_{\theta'}}}^{\mathrm{Sp}(2\theta,\mathbb F_q)} \,\mathbf{1} \boxtimes \rho_{S_i^{\theta'}},
& & 
\mathrm R^T_{j,\theta'} := \mathrm R_{L_{K_{\theta'}}}^{\mathrm{Sp}(2\theta,\mathbb F_q)} \,\mathbf{1} \boxtimes \rho_{T_j^{\theta'}}.
\end{align*}
Then by \eqref{eq2}, we have 
\begin{align*}
\mathrm H_c^{\theta' + i}(X_{I_{\theta'}}(w_{\theta'})) & \simeq \mathrm R^S_{i,\theta'} \oplus \mathrm R^T_{i,\theta'} & & \forall 0 \leq i \leq \theta' - 2, \\ 
\mathrm H_c^{\theta' + i}(X_{I_{\theta'}}(w_{\theta'})) & \simeq \mathrm R^S_{i,\theta'} & & \forall \theta' - 1 \leq i \leq \theta'.
\end{align*}
The cohomology groups of other degrees vanish. The representation $\mathrm R^S_{i,\theta'}$ corresponds to the eigenvalue $q^i$ of $F$, whereas $\mathrm R^T_{j,\theta'}$ corresponds to $-q^{j+1}$. 

\noindent \begin{lem}
Let $0 \leq \theta' \leq \theta$, $0 \leq i \leq \theta'$ and $0 \leq j \leq \theta'-2$.
\begin{enumerate}[label={--},noitemsep,topsep=0pt,leftmargin=0pt]
\item If $i < \theta'$, the representation $\mathrm R^S_{i,\theta'}$ is the multiplicity-free sum of the unipotent representations $\rho_{S}$ where $S \in \mathcal Y^1_{1,\theta}$ runs over the following $4$ distinct families of symbols 

\begin{flalign*}
\setlength\arraycolsep{2pt}
& \text{(S1)} & &
\begin{pmatrix}
0 & \ldots & \theta' - i - 2 & \theta' - i - 1 & \theta' + d \\
1 & \ldots & \theta' - i - 1 & \theta - i - d &
\end{pmatrix}
& & \forall 0 \leq d \leq \theta - \theta', \\
& \text{(S2)} & & 
\begin{pmatrix}
0 & \ldots & \theta' - i - 2 & \theta' - i - 1 + d & \theta' \\
1 & \ldots & \theta' - i - 1 & \theta - i - d &
\end{pmatrix}
& & \forall 1 \leq d \leq \min(i,\theta-\theta'), \\
& \text{(S Exc 1)} & &
\begin{pmatrix}
0 & \ldots & \theta' - i - 1 & \theta' - i & \theta \\
1 & \ldots & \theta' - i & \theta' - i + 1 &
\end{pmatrix}
& & \text{if } \theta' \not = \theta,\\
& \text{(S Exc 2)} & & 
\begin{pmatrix}
0 & \ldots & \theta' - i - 1 & \theta - i - 1 & \theta' + 1 \\
1 & \ldots & \theta' - i & \theta' - i + 1 &
\end{pmatrix}
& & \text{if } \theta' \not = \theta, \theta - 1 \text{ and } \theta \leq \theta' + i + 1.\\
\end{flalign*}

\item The representation $\mathrm R^S_{\theta',\theta'}$ is the multiplicity-free sum of the unipotent representations $\rho_{S}$ where $S \in \mathcal Y^1_{1,\theta}$ runs over the following $2$ distinct families of symbols 

\begin{flalign*}
\setlength\arraycolsep{2pt}
& \text{(S1')} & &
\begin{pmatrix}
0 & \theta' + 1 + d \\
\theta - \theta' - d & 
\end{pmatrix}
& & \forall 0 \leq d \leq \theta - \theta', \\
& \text{(S2')} & & 
\begin{pmatrix}
d & \theta' + 1 \\
\theta - \theta' - d & 
\end{pmatrix}
& & \forall 1 \leq d \leq \min(\theta',\theta-\theta').
\end{flalign*}

\item If $j + 2 < \theta'$, the representation $\mathrm R^T_{j,\theta'}$ is the multiplicity-free sum of the unipotent representations $\rho_{T}$ where $T \in \mathcal Y^1_{3,\theta}$ runs over the following $4$ distinct families of symbols
 
\begin{flalign*}
\setlength\arraycolsep{2pt}
& \text{(T1)} & &
\begin{pmatrix}
0 & \ldots & \theta' - j - 4 & \theta' - j - 3 & \theta' - j - 2 & \theta' - j - 1 & \theta' + d \\
1 & \ldots & \theta' - j - 3 & \theta - j - 2 - d & &
\end{pmatrix}
& & \forall 0 \leq d \leq \theta - \theta', \\
& \text{(T2)} & & 
\begin{pmatrix}
0 & \ldots & \theta' - j - 4 & \theta' - j - 3 & \theta' - j - 2 & \theta' - j - 1 + d & \theta' \\
1 & \ldots & \theta' - j - 3 & \theta - j - 2 - d & &
\end{pmatrix}
& & \begin{array}{@{}l@{}} \forall 1 \leq d \leq \\ \min(j,\theta-\theta'), \end{array} \\
& \text{(T Exc 1)} & &
\begin{pmatrix}
0 & \ldots & \theta' - j - 2 & \theta' - j - 1 & \theta' - j & \theta \\
1 & \ldots & \theta' - j - 1 & & &
\end{pmatrix}
& & \text{if } \theta' \not = \theta,\\
& \text{(T Exc 2)} & & 
\begin{pmatrix}
0 & \ldots & \theta' - j - 2 & \theta' - j - 1 & \theta - j - 1 & \theta' + 1 \\
1 & \ldots & \theta' - j - 1 & & &
\end{pmatrix}
& & \begin{array}{@{}l@{}} \text{if } \theta' \not = \theta, \theta - 1 \\ \text{ and } \theta \leq \theta' + j + 1. \end{array}\\
\end{flalign*}

\item The representation $\mathrm R^T_{\theta'-2,\theta'}$ is the multiplicity-free sum of the unipotent representations $\rho_{T}$ where $T \in \mathcal Y^1_{3,\theta}$ runs over the following $2$ distinct families of symbols 

\begin{flalign*}
\setlength\arraycolsep{2pt}
& \text{(T1')} & &
\begin{pmatrix}
0 & 1 & 2 & \theta' + 1 + d \\
\theta - \theta' - d & & &
\end{pmatrix}
& & \forall 0 \leq d \leq \theta - \theta', \\
& \text{(T2')} & & 
\begin{pmatrix}
0 & 1 & 2 + d & \theta' + 1 \\
\theta - \theta' - d & 
\end{pmatrix}
& & \forall 1 \leq d \leq \min(\theta'-2,\theta-\theta').
\end{flalign*}

\end{enumerate}
\end{lem}

\noindent This lemma results directly from the computational rule explained in \ref{ComputeInduction}. In concrete terms, an induction of the form 
$$\mathrm R_{L_{K_{\theta'}}}^{\mathrm{Sp}(2\theta,\mathbb F_q)} \,\mathbf{1} \boxtimes \rho_{S'}$$
is the sum of all the representations $\rho_S$ where $S$ is obtained from $S'$ by adding a hook of leg length $0$ to both rows, whose lengths sum to $\theta-\theta'$. We illustrate the arguments by looking at a concrete example. \\

\noindent With $\theta = 6, \theta' = 3$ and $i = 2$ let us explain the computation of 
$$\mathrm R^S_{2,3} = \mathrm R_{L_{K_{3}}}^{\mathrm{Sp}(12,\mathbb F_q)} \,\mathbf{1} \boxtimes \rho_{S_2^3}.$$
Recall that 
$$\setlength\arraycolsep{2pt}
S^3_2 = 
\begin{pmatrix}
0 & 3 \\
1 &
\end{pmatrix}.$$
For $0 \leq d \leq \theta - \theta' = 3$, we add a $d$-hook of leg length $0$ to the first row of $S^3_2$, and a $(3-d)$-hook of leg length $0$ to its second row.\\
We may always add the hooks to the last entries of each row. By doing so we obtain the representations corresponding to the family of symbols (S1):
\begin{align*}
\setlength\arraycolsep{2pt}
\begin{pmatrix}
0 & 3 \\
4 &
\end{pmatrix}, 
& & 
\begin{pmatrix}
0 & 4 \\
3 &
\end{pmatrix},
& & 
\begin{pmatrix}
0 & 5 \\
2 &
\end{pmatrix},
& & 
\begin{pmatrix}
0 & 6 \\
1 &
\end{pmatrix}.
\end{align*}
When $d \leq \min(\theta - \theta', i) = \min(3,2) = 2$, we may also add the first hook to the penultimate entry of the first row. Note that since $i < \theta'$, the first row of $S_i^{\theta'}$ has at least $2$ entries. By doing so, we obtain the representations corresponding to the family of symbols (S2):
\begin{align*}
\setlength\arraycolsep{2pt}
\begin{pmatrix}
1 & 3 \\
3 &
\end{pmatrix}, 
& & 
\begin{pmatrix}
2 & 3 \\
2 &
\end{pmatrix}.
\end{align*}
Now, recall that symbols are equal up to shifts. Therefore, one may rewrite $S_2^3$ as 
$$\setlength\arraycolsep{2pt}
S^3_2 = \text{shift}(S^3_2) = 
\begin{pmatrix}
0 & 1 & 4 \\
0 & 2
\end{pmatrix}.$$
Written this way, we notice that a $1$-hook can be added to the first entry of the second row, which is a $0$. Then one must add to the first row a hook of length $d = \theta - \theta' - 1 = 2$. One may always add it to the last entry, which results in the first \enquote{exceptional} representation (S Exc 1). Moreover if $d \leq i$, which is the case here, one may also add this hook to the penultimate entry of the first row, which leads to the second \enquote{exceptional} representation (S Exc 2):
\begin{align*}
\setlength\arraycolsep{2pt}
\begin{pmatrix}
0 & 1 & 6 \\
1 & 2 &
\end{pmatrix}, 
& & 
\begin{pmatrix}
0 & 3 & 4 \\
1 & 2
\end{pmatrix}.
\end{align*}
The sum of the representations attached to all the $8$ symbols written above is isomorphic to $\mathrm R^S_{2,3}$.\\

\noindent We also explain in detail the special case $i = \theta$. Thus we compute 
$$\mathrm R^S_{\theta,\theta} = \mathrm R_{L_{K_{\theta'}}}^{\mathrm{Sp}(2\theta,\mathbb F_q)} \,\mathbf{1} \boxtimes \rho_{S_{\theta'}^{\theta'}}.$$
Recall that 
$$\setlength\arraycolsep{2pt}
S^{\theta'}_{\theta'} = 
\begin{pmatrix}
\theta' \\
{}
\end{pmatrix}$$
corresponds to the trivial representation of $\mathrm{Sp}(2\theta',\mathbb F_q)$. In order to compute this induction, we shift the symbol $S^{\theta'}_{\theta'}$ first:
$$\setlength\arraycolsep{2pt}
S^{\theta'}_{\theta'} = 
\begin{pmatrix}
0 & \theta'+1 \\
0 &
\end{pmatrix}.$$
For $0 \leq d \leq \theta - \theta'$, we add a $d$-hook of leg length $0$ to the first row and a $(\theta-\theta'-d)$-hook of leg length $0$ to the second row. We may always add the hooks to the last entries of each row. By doing so, we obtain the representations corresponding to the family of symbols $(S1')$. Moreover when $d \leq \min(\theta', \theta - \theta')$, we may also add the first hook to the $0$ in the first row. It leads to the representations corresponding to the family of symbols $(S2')$.\\
In particular, we notice that the symbol of $(S1')$ with $d = \theta - \theta'$ corresponds to the trivial representation of $\mathrm{Sp}(2\theta,\mathbb F_q)$.

\paragraph{}Now, we have an explicit description of the terms $E^{a,b}_1$ in the first page of the spectral sequence \eqref{spectral}. In the Figure 1, we draw the shape of the first page.\\

\begin{figure}[h]
\hspace{-30pt}
\begin{tikzcd}[sep=small]
	\, & \, & \, & \, & \, & \, & \mathrm R_{\theta,\theta}^S\\
	\, & \, & \, & \, & \, & \quad \mathrm R_{\theta-1,\theta-1}^S \quad \arrow{r} & \quad \mathrm R_{\theta-1,\theta}^S\\
	\, & \, & \, & \, & \quad \mathrm R_{\theta-2,\theta-2}^S \quad \arrow{r} & \quad \mathrm R_{\theta-2,\theta-1}^S \quad \arrow{r} & \mathrm R_{\theta-2,\theta}^S \oplus \mathrm R_{\theta-2,\theta}^T \\
	\, & \, & \, & \reflectbox{$\ddots$} & \, & \, & \vdots\\
    \, & \, & \quad \mathrm R_{2,2}^S \quad \arrow{r} & \ldots \arrow{r} & \mathrm R_{2,\theta-2}^S \oplus \mathrm R_{2,\theta-2}^T \arrow{r} & \mathrm R_{2,\theta- 1}^S \oplus \mathrm R_{2,\theta-1}^T \arrow{r} & \mathrm R_{2,\theta}^S \oplus \mathrm R_{2,\theta}^T\\
    \, & \quad \mathrm R_{1,1}^S \quad \arrow{r} & \mathrm R_{1,2}^S \arrow{r} & \ldots \arrow{r} & \mathrm R_{1,\theta-2}^S \oplus \mathrm R_{1,\theta-2}^T \arrow{r} & \mathrm R_{1,\theta-1}^S \oplus \mathrm R_{1,\theta-1}^T \arrow{r} & \mathrm R_{1,\theta}^S \oplus \mathrm R_{1,\theta}^T \\
    \mathrm R_{0,0}^S \arrow{r} & \mathrm R_{0,1}^S \arrow{r} & \mathrm R_{0,2}^S \oplus \mathrm R_{0,2}^T \arrow{r} & \ldots \arrow{r} & \mathrm R_{0,\theta-2}^S \oplus \mathrm R_{0,\theta-2}^T \arrow{r} & \mathrm R_{0,\theta-1}^S \oplus \mathrm R_{0,\theta-1}^T \arrow{r} & \mathrm R_{0,\theta}^S \oplus \mathrm R_{0,\theta}^T
\end{tikzcd}
\caption{The first page of the spectral sequence.}
\end{figure}

\noindent First, since the Frobenius $F$ acts with the eigenvalue $q^i$ (resp. $-q^{j+1}$) on the representations $\mathrm R^S_{i,\theta'}$ (resp. $\mathrm R^T_{j,\theta'}$), \ref{FrobeniusEigenvalues} Proposition as well as point $(1)$ of \ref{CohomologyS} and \ref{CohomologyT} Theorems follow from the triangular shape of the spectral sequence. Point $(2)$ also follows from \ref{TermsFirstPage} Lemma.\\
Next, we notice that on the $b$-th row of the first page $E_1$, the eigenvalues of $F$ which occur are $q^b$ and $-q^{b+1}$. In particular, the eigenvalues on different rows are all distinct. It follows that all the arrows in the deeper pages of the sequence are zero, therefore it degenerates on the second page. Moreover, the filtration induced by the spectral sequence on the abutment splits, so that $\mathrm H_c^{k}(S_{\theta})$ is isomorphic to the direct sum of the terms $E_2^{k-b,b}$ on the $k$-th diagonal of the second page. \\

\noindent We prove point $(3)$ of \ref{CohomologyS} and \ref{CohomologyT} Theorems. By the shape of the spectral sequence, we see that 
\begin{align*}
\setlength\arraycolsep{2pt}
\mathrm H^{2\theta}_c(S_{\theta}) = \mathrm H^{2\theta}_c(S_{\theta})_{q^{\theta}} \simeq \mathrm R_{\theta,\theta}^S \simeq \rho_{\footnotesize \begin{pmatrix}
\theta \\
{}
\end{pmatrix}}, & &
\mathrm H^{2\theta-2}_c(S_{\theta})_{-q^{\theta-1}} \simeq \mathrm R_{\theta-2,\theta}^T \simeq \rho_{\footnotesize \begin{pmatrix}
0 & 1 & \theta \\
& &
\end{pmatrix}}.
\end{align*}
Moreover, by the spectral sequence we know that $\mathrm H^0_c(S_{\theta})$ is a subspace of $\mathrm R_{0,0}^{S}$, thus the Frobenius $F$ acts like the identity. Since $S_{\theta}$ is projective and irreducible, the cohomology group $\mathrm H^0_c(S_{\theta}) = \mathrm H^0(S_{\theta})$ is trivial. \\

\noindent We now prove point $(4)$ of \ref{CohomologyS} and \ref{CohomologyT} Theorems. Let $2 \leq i+2 \leq \theta' \leq \theta-1$. By extracting the eigenvalue $q^i$ in the spectral sequence, we have a chain 
\begin{center}
\begin{tikzcd}
\ldots \arrow{r} & \mathrm R_{i,\theta'-1}^S \arrow[r,"u"] & \mathrm R_{i,\theta'}^S \arrow[r,"v"] & \mathrm R_{i,\theta'+1}^S \arrow{r} & \ldots
\end{tikzcd}
\end{center}

\noindent The quotient $\mathrm{Ker}(v)/\mathrm{Im}(u)$ is isomorphic to the eigenspace $\mathrm H^{\theta'+i}_c(S_{\theta})_{q^i}$.\\
The middle term $\mathrm R_{i,\theta'}^S$ is the sum of the representations $\rho_S$ where $S$ runs over the families of symbols (S1), (S2), (S Exc 1) and (S Exc 2) as in \ref{TermsFirstPage} Lemma. All these symbols are written in their \enquote{reduced} form, meaning that they can not be written as the shift of another symbol. Let us look at the length of the second row of these symbols. If $S$ belongs to (S1) or (S2), then the second row has length $\theta' - i$. If $S$ belongs to (S Exc 1) or (S Exc 2), then the second row has length $\theta' - i + 1$.\\ 
We may do a similar analysis for the left term (resp. the right term) by replacing $\theta'$ with $\theta' - 1$ (resp. $\theta' + 1$). In the left term $\mathrm R_{i,\theta'-1}^S$, all the representations corresponding to the families (S1) and (S2) have second row of length $\theta'-i-1$. No such representation occurs in the middle term, therefore they all automatically lie in the $\mathrm{Ker}(u)$. Then, in the left term the representation corresponding to (S Exc 1) occurs since $\theta' - 1 \not = \theta$. We observe that it is equivalent to the representation $\rho_S$ occuring in $\mathrm R_{i,\theta'}^S$ with $S$ in the family (S1) and $d=\theta-\theta'$. Further, assume that $\theta \leq \theta' + i$ so that the representation corresponding to (S Exc 2) occurs in $\mathrm R_{i,\theta'-1}^S$. Then we observe that it is equivalent to the representation $\rho_S$ occuring in $\mathrm R_{i,\theta'}^S$ with $S$ in the family (S2) and $d=\theta-\theta' = \min(i,\theta-\theta')$. Hence, it follows that $\mathrm{Im}(u)$ consists of at most $2$ irreducible subrepresentations of $\mathrm R_{i,\theta'}^S$, and they correspond to the symbols of (S1) and (S2) with $d = \theta - \theta'$.\\
Next, all the subrepresentations $\rho_S$ of $\mathrm R_{i,\theta'}^S$ with $S$ in (S1) or (S2) belong to $\mathrm{Ker}(v)$, since no component of $\mathrm R_{i,\theta'+1}^S$ correspond to a symbol whose second row has length $\theta' - i$. Since $\theta' \not = \theta$, the represensation corresponding to (S Exc 1) occurs in $\mathrm R_{i,\theta'}^S$. We observe that it is equivalent to the representation $\rho_S$ occuring in $\mathrm R_{i,\theta'+1}^S$ with $S$ in the family (S1) and $d=\theta-\theta'-1$. Assume that $\theta' \leq \theta - 2$ and $\theta \leq \theta' + i + 1$, so that the representation corresponding to (S Exc 2) occurs in $\mathrm R_{i,\theta'}^S$. Then we observe that it is equivalent to the representation $\rho_S$ occuring in $\mathrm R_{i,\theta'+1}^S$ with $S$ in the family (S2) and $d=\theta-\theta'-1 = \min(i,\theta-\theta'-1)$. Therefore, it is not possible to tell whether the components of $\mathrm R^S_{i,\theta'}$ corresponding to (S Exc 1) and (S Exc 2) are in $\mathrm{Ker}(v)$ or not.\\
In all cases, we conclude that $\mathrm{Ker}(v)/\mathrm{Im}(u)$ contains at least all the representations corresponding to the symbols $S$ in (S1) and (S2) with $d<\theta-\theta'$. With this description we miss up to four irreducible representations, which correspond to (S1) and (S2) with $d = \theta - \theta'$, (S Exc 1) and (S Exc 2). This proves point (4) of \ref{CohomologyS} Theorem.\\
The point (4) of \ref{CohomologyT} Theorem is proved by identical arguments.\\

\noindent We now prove point $(5)$ of \ref{CohomologyS} and \ref{CohomologyT} Theorems. We consider $i = \theta' \not = \theta$. By extracting the eigenvalue $q^i$ in the spectral sequence, we have a chain 
\begin{center}
\begin{tikzcd}
\mathrm R_{i,i}^S \arrow[r,"u"] & \mathrm R_{i,i+1}^S \arrow[r] & \ldots
\end{tikzcd}
\end{center}

\noindent The kernel $\mathrm{Ker}(u)$ is isomorphic to the eigenspace $\mathrm H^{2i}_c(S_{\theta})_{q^i}$. The left term $\mathrm R_{i,i}^S$ is the sum of the representations $\rho_{S'}$ where $S'$ runs over the families of symbols (S1') and (S2'). We observe that the representation $\rho_{S'}$ with $S'$ in (S1') corresponding to some $0 \leq d' \leq \theta - i - 1$ is equivalent to the component $\rho_{S}$ of $\mathrm R_{i,i+1}^S$ with $S$ in (S1) corresponding to $d = d'$. Similarly, we observe that the representation $\rho_{S'}$ with $S'$ in (S2') corresponding to some $1 \leq d' \leq \min(i,\theta - i - 1)$ is equivalent to the component $\rho_{S}$ of $\mathrm R_{i,i+1}^S$ with $S$ in (S2) corresponding to $d = d'$.\\
Therefore, the representation $\rho_S$ corresponding to $S$ in (S1') with $d' = \theta - i$ belongs to $\mathrm{Ker}(u)$. This is no other than the trivial representation. Moreover, if $\min(i,\theta - i - 1) \not = \min(i,\theta-i)$, ie. if $2i\geq \theta$, then the representation $\rho_S$ corresponding to $S$ in (S2') with $d' = \theta - i$ also belongs to $\mathrm{Ker}(u)$. This proves point (5) of \ref{CohomologyS} Theorem.\\
The point (5) of \ref{CohomologyT} Theorem is proved by identical arguments.\\

\noindent Points (6) of \ref{CohomologyS} and \ref{CohomologyT} Theorems follows easily from the shape of the spectral sequence. Indeed, it suffices to notice that all the terms $\mathrm R_{i,\theta}^S$ and $\mathrm R_{j,\theta}^T$ in the rightmost column of the sequence are irreducible. Thus, they may either vanish, either remain the same in the second page.

\noindent Lastly we prove point (7) of \ref{CohomologyS}. Assume first that $\theta = 1$. The $0$-th row of the spectral sequence is given by 
\begin{center}
\begin{tikzcd}
\setlength\arraycolsep{2pt} 
\rho_{\footnotesize \begin{pmatrix}
1\\
{}
\end{pmatrix}} \oplus
\rho_{\footnotesize \begin{pmatrix}
0 & 1\\
1 &
\end{pmatrix}} \arrow[r,"u"] & 
\setlength\arraycolsep{2pt} 
\rho_{\footnotesize \begin{pmatrix}
0 & 1\\
1 &
\end{pmatrix}}
\end{tikzcd}
\end{center}

\noindent We have $\mathrm H^1_c(S_{1}) \simeq \mathrm{Coker}(u)$. Since we already know that $\mathrm H^0_c(S_{1}) \simeq \mathrm{Ker}(u)$ is the trivial representation of $\mathrm{Sp}(2,\mathbb F_q)$, we see that $u$ must be surjective. Therefore $\mathrm H^1_c(S_{1}) = 0$.\\

\noindent \begin{rk}
The vanishing of $\mathrm H^1_c(S_{1})$ also follows directly from the fact that $S_1 \simeq \mathbb P^1$.
\end{rk}

\noindent Let us now assume $\theta \geq 2$. The first terms of the $0$-th row of the spectral sequence are 

\begin{center}
\begin{tikzcd}
\mathrm R_{0,0}^S \arrow[r,"u"] & \mathrm R_{0,1}^S \arrow[r,"v"] & \mathrm R_{0,2}^S \arrow{r} & \ldots
\end{tikzcd}
\end{center}

\noindent We have $\mathrm H^1_c(S_{\theta}) = \mathrm H^1_c(S_{\theta})_1 \simeq \mathrm{Ker}(v)/\mathrm{Im}(u)$. The middle term $\mathrm R_{0,1}^S$ is the sum of all the representations corresponding to the following symbols 
\begin{align*}
\setlength\arraycolsep{2pt}
\begin{pmatrix}
0 & 1 & \theta\\
1 & 2 &
\end{pmatrix}, & & 
\begin{pmatrix}
0 & 1+d\\
\theta-d &
\end{pmatrix}, & &
\forall 0 \leq d \leq \theta-1.
\end{align*}
On the other hand, the left term $\mathrm R_{0,0}^S$ is the sum of all the representations corresponding to the following symbols 
\begin{align*}
\setlength\arraycolsep{2pt}
\begin{pmatrix}
0 & 1+d\\
\theta-d &
\end{pmatrix}, & &
\forall 0 \leq d \leq \theta.
\end{align*}
Since we already know that $\mathrm H^0_c(S_{\theta}) \simeq \mathrm{Ker}(u)$ is the trivial representation of $\mathrm{Sp}(2\theta,\mathbb F_q)$, we see that $\mathrm{Im}(u)$ contains all the components of $\mathrm R_{0,1}^S$ associated to a symbol whose second row has length $1$. Therefore, $\mathrm H^1_c(S_{\theta})$ is either $0$ either irreducible, depending on whether the remaining component 
$$\setlength\arraycolsep{2pt}
\begin{pmatrix}
0 & 1 & \theta\\
1 & 2 &
\end{pmatrix}$$
is in $\mathrm{Ker}(v)$ or not. This proves point (7) and concludes the proof of \ref{CohomologyS} and \ref{CohomologyT} Theorems.

\section{The geometry of the ramified PEL unitary Rapoport-Zink space of signature $(1,n-1)$}

\subsection{The Bruhat-Tits stratification}

\paragraph{} Recall that $E = \mathbb Q_p[\pi]$ is a quadratic ramified extension of $\mathbb Q_p$ with $\pi = \sqrt{-p}$ (case $E=E_1$) or $\pi = \sqrt{\epsilon p}$ (case $E=E_2$). If $k$ is any perfect field over $\mathbb F_p$, we define $E_k := E\otimes_{\mathbb Q_p} W(k)_{\mathbb Q}$ with the embedding $E\hookrightarrow E_k, x\mapsto x\otimes 1$. We still write $\overline{\,\cdot\,}$ and $\sigma$ for $\overline{\,\cdot\,}\otimes \mathrm{id}$ and $\mathrm{id}\otimes \sigma$ respectively on $E_k$. We define 
$$E' := 
\begin{cases} 
E & \text{if } E=E_1,\\
E_{\mathbb F_{p^2}} & \text{if } E=E_2.
\end{cases}$$ 
Eventually we write $\widecheck{E} := E_{\mathbb F}$, where $\mathbb F := \overline{\mathbb F_p}$. In \cite{RTW}, the authors introduce the ramified PEL unitary Rapoport-Zink space $\mathcal M$ of signature $(1,n-1)$ as a moduli space which classifies the deformations of a given $p$-divisible group $\mathbb X$ equipped with additional structures, called the \textbf{framing object}. The latter is defined over $\mathbb F$ and the Rapoport-Zink space $\mathcal M$ is defined over $\mathcal O_{\widecheck{E}}$. For our purpose, it will be convenient to define this space over $\mathcal O_{E_k}$ where $k$ is the smallest possible perfect extension of $\mathbb F_p$. Therefore we start by defining the framing object over a finite field. Denote by $A_k$ the Dieudonné ring over $k$, that is the $p$-adic completion of the associative ring $W(k)\langle F,V \rangle$ with two indeterminates satisfying the relations $FV = VF = p$, $F\lambda = \lambda^{\sigma}F$ and $V\lambda^{\sigma} = \lambda V$ for all $\lambda \in W(k)$. We denote by $\mathbb D(\cdot)$ the \textit{covariant} Dieudonné module functor from the category of $p$-divisible groups over $k$, to the category of $A_k$-modules that are free of finite rank over $W(k)$. A $p$-divisible group $X$ over $k$ is called \textbf{superspecial} if $\mathbb D(X) \otimes_{W(k)} W(L) \simeq A_{1,1}^{\oplus g}\otimes_{W(k)} W(L)$ for some $g\geq 1$ and some algebraically closed extension $L/k$, where $A_{1,1} := A_k/A_k(F-V)$ seen as a quotient of left $A_k$-modules. In particular, if $X$ is superspecial then $2g = \mathrm{height}(X) = 2\dim(X)$. Eventually, we define non-negative integers $m,m^+$ and $m^-$ via the formula 
$$n = \begin{cases} 
2m+1 & \text{if } n \text{ is odd}, \\
2m^+ = 2(m^-+1) & \text{if } n \text{ is even}. 
\end{cases}$$

\paragraph{}\label{FramingObject} Assume first that $E = E_1$, in which case $\mathbb X$ can be defined over $\mathbb F_p$. According to \cite{LiOort} §1.2, there exists an elliptic curve $\mathcal E$ over $\mathbb F_p$ whose relative Frobenius $\mathcal F: \mathcal E \rightarrow \mathcal E$ satisfies $\mathcal F^2+[p] = 0$. Its Dieudonné module is isomorphic to $\mathbb D(\mathcal E) \simeq A_{\mathbb F_p} / A_{\mathbb F_p}(F+V)$. If $L$ is any extension of $k$ containing $\mathbb F_{p^4}$ then $\mathbb D(\mathcal E) \otimes W(L) \simeq A_{1,1} \otimes W(L)$ so that $\mathcal E$ is supersingular. By \cite{Tate} Theorem 2, the endomorphism algebra $\mathrm{End}^{\circ}(\mathcal E)$ is commutative and isomorphic to $\mathbb Q[\mathcal F]$. Thus we have an action of $\mathcal O_E$ on the $p$-divisible group $\mathcal E[p^{\infty}]$ via the choice of an embedding 
$$\iota_{\mathcal E} : E \xrightarrow{\sim} \mathrm{End}^{\circ}(\mathcal E) \otimes \mathbb Q_p = \mathrm{End}^{\circ}(\mathcal E[p^{\infty}]).$$
Eventually we have a canonical principal polarization $\lambda_{\mathcal E}:\mathcal E \xrightarrow{\sim} \mathcal E^{\vee}$. Next, as in \cite{RTW} we define $\mathbb X_2^{+} := \mathcal E[p^{\infty}] \times \mathcal E[p^{\infty}]$ with diagonal $\mathcal O_E$-action and polarization induced by the $2\times 2$ matrix having $1$'s on the anti-diagonal and $0$'s on the diagonal. In the same manner, define $\mathbb X_2^{-}$ but with polarization induced by a $2\times 2$ diagonal matrix having coefficients $u_1,u_2 \in \mathbb Z_{p}^{\times}$ such that $-u_1u_2$ is not a norm of $E$. The framing object $\mathbb X$ is given by any of the three following cases
\begin{align*}
& (\mathbb X_2^+)^{m}\times \mathcal E[p^{\infty}] & & \text{when }n\text{ is odd,}\\
& (\mathbb X_2^+)^{m^+} & & \text{when }n\text{ is even (split case),}\\
& (\mathbb X_2^+)^{m^-}\times \mathbb X_2^- & & \text{when }n\text{ is even (non-split case),}
\end{align*}
with diagonal $\mathcal O_E$-action $\iota_{\mathbb X}$ and polarization $\lambda_{\mathbb X}$.\\

\noindent Assume now that $E = E_2$. Since there is no supersingular elliptic curve over $\mathbb F_p$ whose endomorphism algebra at $p$ contains $E$, the framing object $\mathbb X$ may only be defined over $\mathbb F_{p^2}$ in this case. There exists an elliptic curve $\mathcal E'$ over $\mathbb F_{p^2}$ whose Dieudonné module is isomorphic to $A_{1,1}$. The endomorphism algebra $\mathrm{End}^{\circ}(\mathcal E)$ is a central simple algebra over $\mathbb Q$ of degree $4$ which ramifies only at $p$ and infinity. At $p$, it is a quaternion algebra over $\mathbb Q_p$ generated by elements $i,j$ such that $i^2 = -\epsilon$, $j^2 =p$ and $ij = -ji$. By fixing an embedding of $E$, we obtain an $\mathcal O_E$-action on $\mathcal E'[p^{\infty}]$ and we equip it with its natural polarization. We may then proceed with defining $\mathbb X_2^+$, $\mathbb X_2^-$ and $\mathbb X$ exactly as in the previous paragraph, except that we use $\mathcal E'$ instead of $\mathcal E$. 

\paragraph{}\label{RZ-Space} Let $\mathrm{Nilp}$ denote the category of $\mathcal O_{E'}$-schemes where $\pi$ is locally nilpotent. For $S\in \mathrm{Nilp}$, a \textbf{unitary $p$-divisible group of signature $(1,n-1)$} over $S$ is a triple $(X,\iota_X,\lambda_X)$ where 
\begin{enumerate}[label={--}]
\item $X$ is a $p$-divisible group over $S$.
\item $\iota_X: \mathcal O_{E}\rightarrow \mathrm{End}(X)$ is a $\mathcal O_{E}$-action on $X$ such that the induced action on its Lie algebra satisfies the Kottwitz and the Pappas conditions: 
\begin{align*}
\forall a \in \mathcal O_E, & & & \mathrm{char}(\iota(a)\,|\,\mathrm{Lie}(X)) = (T - a)^1(T-\overline{a})^{n-1}, \\
\forall n \geq 3, & & & \bigwedge^n (\iota(\pi)-\pi\,|\,\mathrm{Lie}(X)) = 0 \text{ and } \bigwedge^2(\iota(\pi)+\pi\,|\,\mathrm{Lie}(X)) = 0.
\end{align*}
\item $\lambda_X:X \xrightarrow{\sim} {}^tX$ is a principal polarization, where ${}^tX$ denotes the Serre dual of $X$. We assume that the associated Rosati involution induces $\overline{\,\cdot\,}$ on $\mathcal O_E$.
\end{enumerate}

\noindent Note that $\mathrm{char}(\iota(a)\,|\,\mathrm{Lie}(X))$ is a polynomial with coefficients in $\mathcal O_S$. The Kottwitz condition compares it with a polynomial with coefficients in $\mathcal O_E \subset \mathcal O_{E'}$ via the structure morphism $S \rightarrow \mathcal O_{E'}$. For instance, the framing object $(\mathbb X,\iota_{\mathbb X},\lambda_{\mathbb X})$ defined in the previous paragraph is an exemple of unitary $p$-divisible group of signature $(1,n-1)$ over $\kappa(E')$.\\
The following set-valued functor $\mathcal M$ defines a moduli problem classifying deformations of $\mathbb X$ by quasi-isogenies. More precisely, for $S \in \mathrm{Nilp}$ the set $\mathcal M(S)$ consists of all isomorphism classes of tuples $(X,\iota_X,\lambda_X,\rho_X)$ such that 
\begin{enumerate}[label={--}]
\item $(X,\iota_X,\lambda_X)$ is a unitary $p$-divisible group of signature $(1,n-1)$ over $S$,
\item $\rho_X: X\times_S \overline{S} \rightarrow \mathbb X\times_{\kappa(E')} \overline{S}$ is an $\mathcal O_E$-linear quasi-isogeny compatible with the polarizations, in the sense that ${}^t\rho_X \circ \lambda_{\mathbb X} \circ \rho_X$ is a $\mathbb Q_p^{\times}$-multiple of $\lambda_X$.
\end{enumerate}

\noindent In the second condition, $\overline{S}$ denotes the special fiber of $S$. By \cite{RZ} Corollary 3.40, this moduli problem is represented by a separated formal scheme $\mathcal M$ over $\mathrm{Spf}(\mathcal O_{E'})$ called a \textbf{Rapoport-Zink space}. It is formally locally of finite type and flat over $\mathcal O_{E'}$. Let $\mathcal M_{\mathrm{red}}$ denote the reduced special fiber of $\mathcal M$, which is a scheme locally of finite type over $\mathrm{Spec}(\kappa(E'))$. 

\noindent \begin{rk}
In \cite{RZ}, Corollary 3.40 is stated under the assumption that the residue field $\kappa(E')$ contains $\mathbb F_{p^s}$, where $s > 0$ is an exponent appearing in a decency condition for the isocristal of $\mathbb X$. In general, we say that an isocristal $N$ with Frobenius $F$ is \textbf{decent} if it is generated by elements $n\in N$ such that $F^s n = p^r n$ for some integers $r\geq 0$ and $s>0$ (loc. cit. Definition 2.13). By construction, the isocristal of $\mathbb X$ is decent. If $E = E_2$, then we have $F^2 = p\mathrm{id}$ on $\mathbb D(\mathbb X)_{\mathbb Q}$, so that $s = 2$ and $\kappa(E')$ contains $\mathbb F_{p^2}$. However, if $E=E_1$ then $F^2 = -p\mathrm{id}$ on $\mathbb D(\mathbb X)$, so that a decency equation is given by $F^4 = p^2\mathrm{id}$. In this case $s=4$ and $\kappa(E') = \mathbb F_p$ does not contain $\mathbb F_{p^4}$.\\
Nonetheless, to our understanding, the condition that $\kappa(E')$ contains $\mathbb F_{p^s}$ can be relaxed. It seems to be used only in loc. cit. Lemma 3.37 in order to scale a bilinear form by a suitable unit so that it corresponds to a polarization of isocristals. In our case, since the isocristal $\mathbb D(\mathbb X)$ already comes from a polarized $p$-divisible group, this lemma does not seems necessary.
\end{rk}

\paragraph{}We have a decomposition 
$$\mathcal M = \bigsqcup_{i\in \mathbb Z} \mathcal M_i$$
into a disjoint union of open and closed formal connected subschemes, where the points of $\mathcal M_i$ correspond to those tuples $(X,\iota_X,\lambda_X,\rho_X)$ such that ${}^t\rho_X \circ \lambda_{\mathbb X} \circ \rho_X = c\lambda_X$ with $c \in \mathbb Q_p^{\times}$ having $p$-adic valuation $i$.

\paragraph{}Dieudonné theory can be used to describe the rational points of $\mathcal M$ over a perfect field extension $k$ of $\kappa(E')$. Let $N := \mathbb D(\mathbb X)_{\mathbb Q}$ denote the Dieudonné isocristal of $\mathbb X$. The $\mathcal O_E$-action $\iota_{\mathbb X}$ induces an $E'$-vector space structure on $N$ of dimension $n$. The polarization $\lambda_{\mathbb X}$ induces a $W(\kappa(E'))_{\mathbb Q}$-bilinear skew-symmetric form $\langle\cdot,\cdot\rangle$ on $N$ such that 
\begin{align*}
\forall x,y\in N, & & \langle \mathbf F x, y \rangle &  = \langle x , \mathbf V y \rangle,\\
\forall a \in E, & & \langle a\,\cdot,\cdot \rangle & = \langle \cdot , \overline{a} \,\cdot \rangle,
\end{align*}
where $\mathbf F$ and $\mathbf V$ denote respectively the Frobenius and the Verschiebung on $N$. Let $\tau := \eta \pi \mathbf V^{-1}: N\xrightarrow{\sim} N$ where $\eta \in W(\kappa(E'))_{\mathbb Q}$ is $1$ if $E=E_1$ and a square root of $-\epsilon^{-1}$ if $E=E_2$. Notice that we have $(\eta\pi)^2 = -p$ in both cases. Let $C := N^{\tau}$ be the subset of vectors in $N$ which are fixed by $\tau$. It is naturally an $E$-vector space of dimension $n$ and the natural map $C\otimes_E E' \xrightarrow{\sim} N$ is an isomorphism under which $\tau$ corresponds to $\mathrm{id}\otimes \sigma$. If $x,y\in C$ then 
\begin{align*}
\langle x,y \rangle = \langle \tau(x),\tau(y) \rangle = \langle \eta\pi\mathbf V^{-1}x,\eta\pi\mathbf V^{-1}y\rangle = -p^{-1}(\pi\eta)^{2} \langle x, y\rangle^{\sigma} = \langle x,y \rangle^{\sigma}.
\end{align*}
Therefore the restriction of $\langle \cdot,\cdot \rangle$ to $C$ takes value in $\mathbb Q_p$. We define an $E$-hermitian form $(\cdot,\cdot)$ on $C$ by the formula 
\begin{align*}
\forall x,y\in C, & & (x,y) := \langle \pi x, y \rangle + \langle x , y \rangle \pi \in E.
\end{align*}
Let $k$ be any perfect field extension of $\mathbb F_p$. We extend $(\cdot,\cdot)$ to an $E_k$-hermitian form on $C_k := C \otimes_{E} E_k$ by the formula 
\begin{align*}
\forall x,y\in C, \forall a,b \in E_k, & & (x \otimes a, y \otimes b) := a\overline{b}(x,y).
\end{align*}
We still denote by $\tau$ the map $\mathrm{id}\otimes \sigma$ on $C_k$. For $M$ an $\mathcal O_{E_k}$-lattice in $C_k$, we define its dual lattice $M^{\sharp} := \{x \in C_k \,|\, (x,M) \in \mathcal O_{E_k}\}.$

\paragraph{}Let $k$ be a perfect extension of $\kappa(E')$. The $k$-rational points of $\mathcal M$ are classified by the following proposition. 

\begin{prop}[\cite{RTW} Proposition 2.4]
There is a bijection 
$$\mathcal M_i(k) \simeq \left\{M \subset C_k \text{ an }\mathcal O_{E_k}\text{-lattice} \,|\, M = p^iM^{\sharp}, \pi\tau(M) \subset M \subset \pi^{-1}\tau(M), M \overset{\leq 1}{\subset} M+\tau(M) \right\}.$$
\end{prop} 

\noindent The notation $\overset{\leq 1}{\subset}$ denotes an inclusion of $\mathcal O_{E_k}$-lattices with index at most $1$.

\paragraph{}\label{WittDecomposition} Recall the integers $m,m^+$ and $m^-$ that we defined depending on the parity of $n$. For $k\geq 0$, let $A_k$ denote the $k\times k$ matrix with $1$ on the antidiagonal and $0$ everywhere else. We fix some scalars $u_1,u_2\in \mathbb Z_{p}^{\times}$ such that $-u_1u_2 \not \in \mathrm{Norm}_{E/\mathbb Q_p}(E^{\times})$. We then define the three matrices 

\begin{align*}
T_{\mathrm{odd}} := A_{2m+1}, & & T_{\mathrm{even}}^+ := A_{2m^+}, & & T_{\mathrm{even}}^- := 
\begin{pmatrix}
& & & A_{m^-} \\
& u_1 & 0 & \\
& 0 & u_2 & \\
A_{m^-} & & &
\end{pmatrix}.
\end{align*}

\noindent By construction, $C$ has a basis in which $(\cdot,\cdot)$ is given by $T_{\mathrm{odd}}, T_{\mathrm{even}}^+$ or $T_{\mathrm{even}}^-$ when $n$ is odd, when $n$ is even and $\mathbb X = (\mathbb X_2^{+})^{m^+}$ (split case) or when $n$ is even and $\mathbb X = (\mathbb X_2^{+})^{m^-} \times \mathbb X_2^-$ (non-split case) respectively. We denote such a basis by $e = (e_{-j},e_0^{\mathrm{an}},e_j)_{1\leq j \leq m}$ when $n$ is odd, and if $n$ is even by $e = (e_{-j},e_j)_{1\leq j \leq m^+}$ in the split case and by $e = (e_{-j},e_0^{\mathrm{an}},e_1^{\mathrm{an}},e_j)_{1 \leq j \leq m^-}$ in the non-split case.

\begin{rk}
The integers $m,m^+$ and $m^-$ correspond to the Witt index of $C$ in each of the three cases.
\end{rk} 

\paragraph{}Let $J = \mathrm{Aut}(\mathbb X)$ be the group of automorphisms of $\mathbb X$ compatible with the additional structures. By \cite{RTW} Lemma 2.3, we have an isomorphism $J \simeq \mathrm{GU}(C,(\cdot,\cdot))$. As a reductive group over $\mathbb Q_p$, $J$ is quasi-split if and only if $n$ is odd or $n$ is even and $C$ is split. Let
$$c:J\mapsto \mathbb Q_{p}^{\times}$$
denote the multiplier character. For instance, $\pi^k\mathrm{id} \in J$ has multiplier $\pi^{2k}\in \mathbb Q_p^{\times}$. We define a surjective morphism $\alpha: J \mapsto \mathbb Z$ by $\alpha(g) := v_p(c(g))$ where $v_p$ is the $p$-adic valuation. We denote by $J^{\circ}$ the kernel of $\alpha$. Then $J^{\circ}$ is the subgroup generated by all compact subgroups of $J$.\\
The group $J$ acts on $\mathcal M$ via 
$$g\cdot (X,\iota_X,\lambda_X,\rho_X) := (X,\iota_X,\lambda_X,g\circ\rho_X).$$
An element $g\in J$ induces an isomorphism $g:\mathcal M_i \xrightarrow{\sim} \mathcal M_{i+\alpha(g)}$. 

\paragraph{}For $i\in \mathbb Z$ we define 
$$\mathcal L_i := \left\{ \Lambda \subset C \text{ an } \mathcal O_E\text{-lattice} \,|\, p^i\Lambda^{\sharp} \subset \Lambda \subset \pi^{-1}p^i\Lambda^{\sharp}\right\}.$$
We also write $\mathcal L$ for the (disjoint) union of the $\mathcal L_i$'s. Elements of $\mathcal L$ are called \textbf{vertex lattices}. If $\Lambda$ is a vertex lattice, its \textbf{orbit type} $t(\Lambda)$ is the lattice index $[\Lambda:p^i\Lambda^{\sharp}]$. According to \cite{RTW}, $t(\Lambda)$ is an even integer between $0$ and $n$.\\
The group $J$ acts on $\mathcal L$ via $g\cdot \Lambda := g(\Lambda)$. An element $g\in J$ defines a type preserving, inclusion preserving bijection $g:\mathcal L_i \xrightarrow{\sim} \mathcal L_{i+\alpha(g)}$. With arguments similar to those used in the unramified case in \cite{vw1}, one may prove the following proposition. 

\begin{prop}
Two vertex lattices $\Lambda,\Lambda' \in \mathcal L$ are in the same $J$-orbit if and only if $t(\Lambda) = t(\Lambda')$. 
\end{prop}

\paragraph{}Recall the basis $e$ of $C$ that we fixed in \ref{WittDecomposition}. For a family of integers $(r_j,s)$ where 
\begin{enumerate}[label={--},noitemsep,topsep=0pt]
\item $1\leq j \leq m$ and $s\in \mathbb Z$ if $n$ is odd,
\item $1\leq j \leq m^+$ and $s = \emptyset$ if $n$ is even and $C$ is split,
\item $1\leq j \leq m^-$ and $s\in \mathbb Z^2$ if $n$ is even and $C$ is non-split,
\end{enumerate}
we denote by 
$$\Lambda(r_{-j} \,;\, s \,;\, r_j)$$
the $\mathcal O_E$-lattice generated by the vectors $p^{r_{\pm j}}e_{\pm j}$, and by $p^{s_0}e_0^{\mathrm{an}}$ and $p^{s_1}e_{1}^{\mathrm{an}}$ when it makes sense.

\begin{prop}
A lattice $\Lambda = \Lambda(r_{-j} \,;\, s \,;\, r_j)$ is a vertex lattice if and only if  for some $i\in \mathbb Z$, $r_{-j}+r_j \in \{2i-1,2i\}$ for all $j$, and $s$ is respectively given by $i, \emptyset$ or $(i,i)$ depending on whether $n$ is odd or even with $C$ split or not. When $\Lambda$ is a vertex lattice, its orbit type is given by 
$$t(\Lambda) = 2\#\{j \,|\, r_{-j}+r_j = 2i-1\}.$$
\end{prop}

\noindent This is proved in the same way as \cite{muller} 1.2.4 Proposition. In particular, when $n$ is even and $C$ is non-split there is no vertex lattice of orbit type $n$. Let $t_{\mathrm{max}}$ denote the maximal type of a vertex lattice. We have 
$$t_{\mathrm{max}} = \begin{cases}
n-1 & \text{if } n \text{ is odd},\\
n & \text{if } n \text{ is even and } C \text{ is split},\\
n-2 & \text{if } n \text{ is even and } C \text{ is non-split}.
\end{cases}$$
We also write $t_{\mathrm{max}} = 2\theta_{\mathrm{max}}$, so that $\theta_{\mathrm{max}} = m,m^+$ or $m^-$ depending on whether $n$ is odd, $n$ is even with $C$ split or $n$ is even with $C$ non-split respectively. 

\paragraph{}The set $\mathcal L$ of vertex lattices can be given the structure of a polysimplicial complex, by declaring that an $s$-simplex in $\mathcal L$ is a subset $\{\Lambda_0,\ldots,\Lambda_s\} \subset \mathcal L_i$ for some $i \in \mathbb Z$ such that, up to reordering, we have 
$$\Lambda_0 \subset \Lambda_1 \subset \ldots \subset \Lambda_s.$$
Depending on whether $n$ is odd or even with $C$ split or not, for an $s$-simplex to exist we must have $s$ between $0$ and $\theta_{\mathrm{max}}$. We fix a specific maximal simplex in each case.
\begin{align*}
\intertext{If $n$ is odd, for $0\leq \theta \leq \theta_{\mathrm{max}}$ we define}
\Lambda_{\theta} & := \Lambda(0^{\theta_{\mathrm{max}}} \,;\, 0 \,;\, 0^{\theta_{\mathrm{max}}-\theta}, -1^{\theta}).
\intertext{If $n$ is even and $C$ is split, for $0\leq \theta \leq \theta_{\mathrm{max}}$ we define}
\Lambda_{\theta} & := \Lambda(0^{\theta_{\mathrm{max}}} \,;\, 0^{\theta_{\mathrm{max}}-\theta}, -1^{\theta}).
\intertext{If $n$ is even and $C$ is non-split, for $0\leq \theta \leq \theta_{\mathrm{max}}$ we define} 
\Lambda_{\theta} & := \Lambda(0^{\theta_{\mathrm{max}}} \,;\, 0,0 \,;\, 0^{\theta_{\mathrm{max}}-\theta}, -1^{\theta}).
\end{align*}
In each case we have $\Lambda_{\theta} \in \mathcal L_0$ and $\Lambda_{\theta} \subset \Lambda_{\theta+1}$. Moreover the orbit type of $\Lambda_{\theta}$ is $2\theta$. 

\paragraph{}For $\Lambda \in \mathcal L$, let $J_{\Lambda}$ denote the fixator of $\Lambda$ in $J$. Let $J_{\Lambda}^+$ be its pro-unipotent radical, and write $\mathcal J_{\Lambda} := J_{\Lambda}/J_{\Lambda}^{+}$ for the \textbf{finite reductive quotient}. It is a finite group of Lie type over $\mathbb F_p$.\\
We define also the quotients 
\begin{align*}
V_{\Lambda}^0 := \Lambda/p^i\Lambda^{\sharp}, & & V_{\Lambda}^1 := \pi^{-1}p^i\Lambda^{\sharp}/\Lambda.
\end{align*}
They are both $\mathcal O_E/\pi\mathcal O_E \simeq \mathbb F_p$-vector spaces of dimension respectively $t(\Lambda)$ and $n - t(\lambda)$. Both spaces inherit a perfect $\mathbb F_p$-bilinear form, which we denote by the same notation $\{\cdot,\cdot\}$, induced respectively by $\pi p^{-i}(\cdot,\cdot)$ and by $p^{1-i}(\cdot,\cdot)$. Then $\{\cdot,\cdot\}$ is symplectic on $V_{\Lambda}^0$ whereas it is symmetric on $V_{\Lambda}^1$. If $k$ is a perfect field extension of $\mathbb F_p$, we denote by $V_{\Lambda,k}^0$ and $V_{\Lambda,k}^1$ the scalar extensions to $k$, equipped with their perfect $k$-bilinear forms $\{\cdot,\cdot\}$, and we denote by $\tau$ the map $\mathrm{id}\otimes\sigma$ on both spaces. If $U$ is a subspace, we denote by $U^{\perp}$ its orthogonal.\\
We denote by $\mathrm{GSp}(\cdot)$ and $\mathrm{GO}(\cdot)$ the associated groups of symplectic or orthogonal similitudes. Then we have a natural isomorphism 
$$\mathcal J_{\Lambda} \simeq \mathrm{G}(\mathrm{GSp}(V_{\Lambda}^0)\times \mathrm{GO}(V_{\Lambda}^1)),$$
where the right-hand side is the subgroup of $\mathrm{GSp}(V_{\Lambda}^0)\times \mathrm{GO}(V_{\Lambda}^1)$ with both factors sharing the same multiplier in $\mathbb F_p^{\times}$. Let $\mathcal J_{\Lambda}^{\circ}$ be the connected component of unity and let $J_{\Lambda}^{\circ}$ be its preimage in $J_{\Lambda}$. We recall some known facts about parahoric subgroups of $J$, see for instance \cite{luststevens} section 2 for a complete summary. 

\begin{prop}
Let $\Lambda \in \mathcal L$.
\begin{enumerate}[label={--},noitemsep,topsep=0pt]
\item The fixator $J_{\Lambda}$ is a maximal compact subgroup of $J$. All maximal compact subgroups arise this way.
\item The subgroup $J_{\Lambda}^{\circ}$ is a parahoric subgroup of $J$. It is a maximal parahoric subgroup unless $n$ is even, $C$ is split and $t(\Lambda) = n-2$. All maximal parahoric subgroups of $J$ arise this way.
\item If $t(\Lambda) \not = n$ then $J_{\Lambda}^{\circ}$ has index $2$ in $J_{\Lambda}$. If $t(\Lambda) = n$ then $J_{\Lambda}^{\circ} = J_{\Lambda}$.
\end{enumerate}
\end{prop}

\noindent We note that the condition $t(\Lambda) = n$ can only occur when $n$ is even and $C$ is split. Besides, in this case any vertex lattice $\Lambda \in \mathcal L_i$ of orbit type $n-2$ is contained in precisely two different vertex lattices $\Lambda_1, \Lambda_2 \in \mathcal L_i$ of orbit type $n$ (see \ref{NumberOfLattices}). Then the parahoric subgroup $J_{\Lambda}^{\circ}$ is the intersection of the two maximal parahoric subgroups $J_{\Lambda_1}^{\circ}$ and $J_{\Lambda_2}^{\circ}$. 

\begin{notation}
If $\Lambda$ is one of the $\Lambda_{\theta}$'s, we write $J_{\theta}$, $V^0_{\theta}$ and $V^1_{\theta}$ instead of $J_{\Lambda_{\theta}}$, $V^0_{\Lambda_{\theta}}$ and $V^1_{\Lambda_{\theta}}$ respectively. 
\end{notation}

\paragraph{}In \cite{RTW} section 6, the authors attach to any vertex lattice $\Lambda \in \mathcal L_i$ a closed projective subscheme $\mathcal M_{\Lambda} \hookrightarrow \mathcal M_{i,\mathrm{red}}$, which is called a \textbf{closed Bruhat-Tits stratum}. Its rational points are described by the following proposition. 

\begin{prop}[\cite{RTW} Corollary 6.3]
Let $k$ be a perfect field extension of $\kappa(E')$ and let $\Lambda \in \mathcal L_i$. We have a natural bijection 
$$\mathcal M_{\Lambda}(k) \simeq \left\{M\in \mathcal M_i(k) \,|\, M \subset \Lambda_k := \Lambda \otimes_{\mathcal O_E} \mathcal O_{E_k}\right\}.$$
\end{prop}

\noindent By mapping $M \in \mathcal M_{\Lambda}(k)$ to its image $\overline{M} := M/p^i\Lambda_k^{\sharp}$ in $V^0_{\Lambda,k}$, one obtains a bijection between $\mathcal M_{\Lambda}(k)$ and the set 
$$\{U\subset V^0_{\Lambda,k} \,|\, U^{\perp} = U \text{ and }U \overset{\leq 1}{\subset} U+\tau(U)\}.$$
The action of $J$ on $\mathcal M$ restricts to an action of $J_{\Lambda}$ on $\mathcal M_{\Lambda}$. This action factors through the finite reductive quotient $\mathcal J_{\Lambda}$, and the $\mathrm{GO}(V_{\Lambda}^1)$-component acts trivially. Therefore we obtain an action of $\mathrm{GSp}(V_{\Lambda}^0) \simeq \mathrm{GSp}(2\theta,\mathbb F_p)$ where $t(\Lambda) = 2\theta$ on $\mathcal M_{\Lambda}$. The main theorem of loc. cit. is the construction of a natural isomorphism between the closed Bruhat-Tits stratum $\mathcal M_{\Lambda}$ and the closed Deligne-Lusztig variety $S_{\theta}$ that we introduced in \ref{IsomorphismDLVariety}.

\begin{theo}[\cite{RTW} Proposition 6.7]
Let $\Lambda\in \mathcal L$ and write $t(\Lambda) = 2\theta$ for its orbit type. There is a natural isomorphism 
$$\mathcal M_{\Lambda} \xrightarrow{\sim} S_{\theta} \otimes_{\mathbb F_p} \kappa(E')$$
which is $\mathrm{GSp}(2\theta,\mathbb F_p)$-equivariant.
\end{theo}

\noindent In particular, the variety $\mathcal M_{\Lambda}$ is always defined over $\mathbb F_p$. We note also that the $\mathrm{GSp}(2\theta,\mathbb F_p)$-action on $S_{\theta}$ is induced from \ref{SameDLVarieties}. Identifying the unipotent representations of $\mathrm{Sp}(2\theta,\mathbb F_p)$ and of $\mathrm{GSp}(2\theta,\mathbb F_p)$ as in \ref{SameUnipotent}, the theorems \ref{CohomologyS} and \ref{CohomologyT} give us a certain knowledge of the cohomology of the closed Bruhat-Tits stratum $\mathcal M_{\Lambda}$. 

\paragraph{}The closed subschemes $\mathcal M_{\Lambda}$ form the \textbf{Bruhat-Tits stratification} of the reduced special fiber $\mathcal M_{\mathrm{red}}$, whose incidence relations mimic the combinatorics of vertex lattices. 

\begin{theo}[\cite{RTW} Theorem 6.10]
Let $i\in \mathbb Z$ and $\Lambda,\Lambda'\in\lcal_i$.
\begin{enumerate}[label=\upshape (\arabic*), topsep = 0pt, noitemsep]
		\item The inclusion $\Lambda\subset \Lambda'$ is equivalent to the scheme-theoretic inclusion $\mathcal M_{\Lambda}\subset \mathcal M_{\Lambda'}$. It implies $t(\Lambda)\leq t(\Lambda')$ with equality if and only if $\Lambda = \Lambda'$.
		\item The three following assertions are equivalent.
		\begin{align*}
		\mathrm{(i)}\;\Lambda\cap \Lambda'\in \lcal_i. & & \mathrm{(ii)}\; \Lambda\cap \Lambda' \text{ contains a lattice of }\lcal_i. & & \mathrm{(iii)}\;\mathcal M_{\Lambda}\cap \mathcal M_{\Lambda'} \not = \emptyset.
		\end{align*}
		If these conditions are satisfied, then $\mathcal M_{\Lambda}\cap \mathcal M_{\Lambda'}=\mathcal M_{\Lambda\cap \Lambda'}$ scheme-theoretically.
		\item If $k$ is a perfect field field extension of $\kappa(E')$ then $\mathcal M_i(k)=\bigcup_{\Lambda\in \lcal_i}\mathcal M_{\Lambda}(k)$.
	\end{enumerate}
\end{theo}

\noindent It follows in particular that $\mathcal M_{\mathrm{red}}$ has pure dimension $\theta_{\mathrm{max}}$.

\subsection{Counting the Bruhat-Tits strata}

\label{CountingBTStrata}In this short section we give a formula for the number of closed Bruhat-Tits strata of a certain dimension, which are included in or which contain a fixed stratum. Let $d\geq 0$ and let $V$ be a $d$-dimensional $\mathbb F_p$-vector space equipped with a non-degenerate symmetric or symplectic form $\{\cdot,\cdot\}:V\times V \rightarrow \mathbb F_p$. Define the integer $\delta$ by 
$$d = 
\begin{cases}
2\delta & \text{if } \{\cdot,\cdot\} \text{ is symplectic, or if it is symmetric, } d \text{ is even and } V \text{ is split,}\\
2(\delta + 1) & \text{if } \{\cdot,\cdot\} \text{ is symmetric, } d \text{ is even and } V \text{ is not split,}\\
2\delta + 1 & \text{if } \{\cdot,\cdot\} \text{ is symmetric, } d \text{ is odd}.
\end{cases}
$$
Thus $\delta$ corresponds to the Witt index of $V$. For $0\leq r \leq \delta$ we define 
$$N(r,V) := \{U\subset V \,|\, \dim U = r \text{ and } U \subset U^{\perp}\}.$$
Let $i \in \mathbb Z$ and let $\Lambda \in \mathcal L_i$ be a vertex lattice. Write $t(\Lambda) = 2\theta$ so that $0 \leq \theta \leq \theta_{\mathrm{max}}$. 

\begin{prop}
\begin{enumerate}[label=\upshape (\arabic*), topsep = 0pt, noitemsep]
\item The set of vertex lattices $\Lambda'\in\mathcal L_i$ of orbit type $t(\Lambda') = 2\theta'$ such that $\Lambda' \subset \Lambda$ is in bijection with $N(\theta-\theta',V^0_{\Lambda})$.
\item The set of vertex lattices $\Lambda'\in\mathcal L_i$ of orbit type $t(\Lambda') = 2\theta'$ such that $\Lambda \subset \Lambda'$ is in bijection with $N(\theta'-\theta,V^1_{\Lambda})$.
\end{enumerate}
\end{prop}

\noindent The bijection is established by mapping $\Lambda'$ to the image of $p^i(\Lambda')^{\sharp}$ in $V^0_{\Lambda}$ in case (1), and to its own image in $V^1_{\Lambda}$ in case (2). The following statement gives the cardinality of $N(r,V)$.

\begin{prop}
Let $\delta$ be the Witt index of $V$ and let $0\leq r \leq \delta$.
\begin{enumerate}[label={--},noitemsep,topsep=0pt]
\item If $\{\cdot,\cdot\}$ is symplectic, or if it is symmetric and $d$ is odd, then
$$\#N(r,V) = \prod_{i=1}^{r} \frac{p^{2(i+\delta-r)}-1}{p^i-1}.$$
\item If $\{\cdot,\cdot\}$ is symmetric, $d$ is even and $V$ is not split then
$$\#N(r,V) = \frac{p^{\delta+1}+1}{p^{\delta+1-r}+1}\prod_{i=1}^r \frac{p^{2(i+\delta-r)}-1}{p^i-1}.$$
\item If $\{\cdot,\cdot\}$ is symmetric, $d$ is even and $V$ is split then
$$\#N(r,V) = \frac{p^{\delta-r}+1}{p^{\delta}+1}\prod_{i=1}^{r} \frac{p^{2(i+\delta-r)}-1}{p^i-1}.$$
\end{enumerate}
\end{prop}

\noindent The proof of this proposition is very similar to \cite{muller} 1.4.2 Proposition, therefore we omit it.

\begin{rk}\label{NumberOfLattices}
Assume that $n$ is even and that $C$ is split. Let $\Lambda\in \mathcal L_i$ with orbit type $n-2 = 2(\theta_{\mathrm{max}}-1)$. The set of vertices $\Lambda'\in \mathcal L_i$ of maximal orbit type $n = 2\theta_{\mathrm{max}}$ which contain $\Lambda$ is in bijection with $N(1,V^1_{\Lambda})$. The space $V^1_{\Lambda}$ has dimension $2$ with a symmetric form and is split. According to the formula above, the number of such lattices $\Lambda'$ is 
$$\frac{p^{1-1}+1}{p^1+1}\prod_{i=1}^{1}\frac{p^{2i}-1}{p^{i}-1} = 2.$$
We recover the fact stated in \cite{RTW} proof of Proposition 3.4 that in the even split case, vertex lattices of orbit type $n-2$ correspond in fact to an edge in the Bruhat-Tits building of $J$.
\end{rk}

\subsection{Shimura variety and $p$-adic uniformization of the basic stratum}

\paragraph{}In this section, we introduce the integral model of a Shimura variety whose supersingular locus is uniformized by the Rapoport-Zink space $\mathcal M$. We follow the construction of \cite{RTW} Section 7. Let $\mathbb E$ be an imaginary quadratic field in which $p>2$ ramifies, and let $\overline{\,\cdot\,}$ denote the non-trivial element of $\mathrm{Gal}(\mathbb E/\mathbb Q)$. Let $\mathbb V$ be an $n$-dimensional hermitian $\mathbb E$-vector space of signature $(1,n-1)$ at infinity. Let $\mathbb G$ denote the group of unitary similitudes of $\mathbb V$ as a reductive group over $\mathbb Q$.\\
First, we give the moduli description of the canonical model of the Shimura variety associated to the data above. Let $K\subset \mathbb G(\mathbb A_f)$ be an open compact subgroup. For a locally noetherian $\mathbb E$-scheme $S$, let $\mathrm{Sh}_K(S)$ denote the set of isomorphism classes of tuples $(A,\lambda,\iota,\overline{\eta})$ where
\begin{enumerate}[label={--}]
\item $A$ is an abelian scheme over $S$.
\item $\lambda: A\rightarrow \widehat{A}$ is a polarization.
\item $\iota:\mathbb E\rightarrow \mathrm{End}(A)\otimes \mathbb Q$ is a $\mathbb E$-action on $A$ such that $\iota(\overline{x}) = \iota(x)^{\dagger}$ where $\cdot^{\dagger}$ denotes the Rosati involution associated to $\lambda$, and such that the Kottwitz determinant condition is satisfied:
$$\forall x \in \mathbb E,\, \det(T-\iota(x)\,|\,\mathrm{Lie}(A)) = (T-x)^1(T-\overline{x})^{n-1} \in \mathbb E[T].$$
\item $\overline{\eta}$ is a $K$-level structure, that is a $K$-orbit of isomorphisms of $\mathbb E\otimes \mathbb A_f$-modules $\mathrm{H}_1(A,\mathbb A_f) \xrightarrow{\sim} \mathbb V\otimes \mathbb A_f$ that is compatible with the other data.
\end{enumerate}

\noindent According to \cite{KR} Proposition 4.3, when $K$ is small enough the functor $\mathrm{Sh}_K$ is represented by a smooth quasi-projective scheme over $\mathbb E$. As the level $K$ varies, the Shimura varieties sit together in a projective system $(\mathrm{Sh}_K)_K$ on which $\mathbb G(\mathbb A_f)$ acts by Hecke correspondences.

\paragraph{}In order to define integral models for these Shimura varieties, let us assume that there exists a self-dual $\mathcal O_{\mathbb E}$-lattice $\Gamma$ in $\mathbb V$. Let $K \subset \mathbb G(\mathbb A_f)$ denote the stabilizer of $\Gamma$. Let $K^p \subset K\cap \mathbb G(\mathbb A_f^p)$ be an open compact subgroup. For an $\mathcal O_{\mathbb E,(p)}$-scheme $S$, let $\mathrm S_{K^p}(S)$ denote the set of isomorphism classes of tuples $(A,\lambda,\iota,\overline{\eta}^{p})$ where
\begin{enumerate}[label={--}]
\item $A$ is an abelian scheme over $S$.
\item $\lambda: A\rightarrow \widehat{A}$ is a polarization of order prime to $p$.
\item $\iota:\mathcal O_{\mathbb E} \rightarrow \mathrm{End}(A)\otimes \mathbb Z_{(p)}$ is an $\mathcal O_{\mathbb E}$-action on $A$ such that $\iota(\overline{x}) = \iota(x)^{\dagger}$ where $\cdot^{\dagger}$ denotes the Rosati involution associated to $\lambda$, and such that the Kottwitz determinant condition is satisfied:
$$\forall x \in \mathcal O_{\mathbb E},\, \det(T-\iota(x)\,|\,\mathrm{Lie}(A)) = (T-x)^1(T-\overline{x})^{n-1} \in \mathcal O_{\mathbb E}[T].$$
\item $\overline{\eta}^p$ is a $K^p$-level structure, that is a $K^p$-orbit of isomorphisms of $\mathbb E\otimes \mathbb A_f^p$-modules $\mathrm{H}_1(A,\mathbb A_f^p) \xrightarrow{\sim} \mathbb V\otimes \mathbb A_f^p$ that is compatible with the other data.
\end{enumerate}

\noindent By \cite{RTW} Section 7, when $K^p$ is small enough the functor $\mathrm S_{K^p}$ is represented by a smooth quasi-projective $\mathcal O_{\mathbb E,(p)}$-scheme. As the level $K^p$ varies, these integral models form a projective system on which $\mathbb G(\mathbb A_f^p)$ acts by Hecke correspondences. We have natural isomorphisms 
$$\mathrm{Sh}_{K_pK^p} \simeq \mathrm S_{K^p}\otimes_{\mathcal O_{\mathbb E,(p)}} \mathbb E,$$
which are compatible as $K^p$ varies, where $K_p$ is the stabilizer of $\Gamma\otimes \mathbb Z_p$ in $\mathbb V_{\mathbb Q_p}$.

\begin{rk}
From now on, the notation $\mathrm S_{K^p}$ will be used to denote the $\mathbb \mathcal O_{\mathbb E_p}$-scheme obtained by base change.
\end{rk} 

\paragraph{}\label{Uniformization} Let $\overline{S}_{K^p}$ denote the special fiber of the Shimura variety over the residue field $\kappa(\mathbb E_p)$. Let $\overline{S}_{K^p}^{\mathrm{ss}}$ denote the \textbf{supersingular locus}, it is a closed projective subscheme of $\overline{S}_{K^p}$. Let $\widehat{\mathrm S}_{K^p}^{\mathrm{ss}}$ denote the formal completion of $\mathrm S_{K^p}$ along the supersingular locus. Eventually, let $\widehat{\mathrm S}_{K^p}^{\mathrm{ss},\mathrm{an}}$ denote the Berkovich generic fiber of the formal scheme $\widehat{\mathrm S}_{K^p}^{\mathrm{ss}}$, a smooth analytic space over $\mathbb E_p$.\\
From now on, we write $E := \mathbb E_p$ and $\kappa(\mathbb E_p) = \kappa(E) = \mathbb F_p$. Let $(A^0,\lambda^0,\iota^0,\overline{\eta}^0)$ be the $\kappa(E')$-rational point of $\overline{S}_{K^p}^{\mathrm{ss}}$ given by the product of elliptic curves used to define the framing object $\mathbb X$ as in \ref{FramingObject}. In particular the $p$-divisible group $A^0[p^{\infty}]$ is identified with $\mathbb X$ and we may consider the associated Rapoport-Zink space $\mathcal M$ over $\mathrm{Spf}(\mathcal O_{E'})$. As observed in \cite{RTW} Remark 7.1, when $n$ is even the discriminants of the hermitian spaces $C$ and $\mathbb V\otimes E$ are different, so that one space is split precisely when the other is non-split. Let $I$ denote the group of quasi-isogenies of $A^0$ which respect all additional structures. Since $A^0$ is in the basic stratum, $I$ can be seen as an inner-form of $\mathbb G$ such that $I_{\mathbb A_f^p} \simeq \mathbb G_{\mathbb A_f^p}$ and $I_{\mathbb Q_p} \simeq J$. One may therefore think of $I(\mathbb Q)$ as a subgroup both of $\mathbb G(\mathbb A_f^p)$ and of $J$ at the same time. The $p$-adic uniformization theorem gives a geometric link between the Rapoport-Zink space and the supersingular locus of the Shimura variety. 

\noindent \begin{theo}[\cite{RTW}]
There is an isomorphism of formal schemes over $\mathrm{Spf}(\mathcal O_{E'})$
$$\Theta_{K^p}: I(\mathbb Q)\backslash \left( \mathcal M \times \mathbb G(\mathbb A_f^p) / K^p \right) \xrightarrow{\sim} \widehat{\mathrm S}_{K^p}^{\mathrm{ss}}\otimes_{\mathcal O_E} \mathcal O_{E'}$$
which is compatible with the $G(\mathbb A_f^p)$-action as the level $K^p$ varies.
\end{theo}

\noindent As in \cite{muller} 3.6, one also obtains uniformization isomorphisms $(\Theta_{K^p})_s$ and $\Theta_{K^p}^{\mathrm{an}}$ for the special and the generic fibers respectively. 

\paragraph{}\label{BTStrataOnShimura}Let $g_1,\ldots,g_s \in \mathbb G(\mathbb A_f^p)$ be a system of representatives for the double coset space $I(\mathbb Q)\backslash \mathbb G(\mathbb A_f^p) / K^p$, and let $\Gamma_k := I(\mathbb Q) \cap g_kK^pg_k^{-1}$ for $ 1\leq k \leq s$. These are subgroups of $J$ which are discrete and cocompact modulo the center. The uniformization theorem for the special fiber may be written as 
$$(\Theta_{K^p})_s: \bigsqcup_{k=1}^{s} \Gamma_k \backslash \mathcal M_{\mathrm{red}} \xrightarrow{\sim} \overline{S}_{K^p}^{\mathrm{ss}} \otimes_{\mathbb F_p} \kappa(E').$$
Let $\Phi_{K^p}^k$ be the composition $\mathcal M_{\mathrm{red}}\rightarrow \Gamma_k \backslash \mathcal M_{\mathrm{red}} \rightarrow \overline{S}_{K^p}^{\mathrm{ss}}$, and let $\Phi_{K^p}$ be the union of the $\Phi_{K^p}^k$. By the same arguments as \cite{vw2} Section 6.4, the surjection $\Phi_{K^p}$ is a local isomorphism. Moreover the restriction of $\Phi_{K^p}^k$ to any closed Bruhat-Tits stratum $\mathcal M_{\Lambda} \subset \mathcal M_{\mathrm{red}}$ is an isomorphism onto its image. We will denote this scheme-theoretic image by $\overline{\mathrm S}_{K^p,\Lambda,k}$. For varying $\Lambda$ and $k$, these subschemes constitute the closed strata of the \textbf{Bruhat-Tits stratification} of the supersingular locus of the Shimura variety. 

\section{On the cohomology of the Rapoport-Zink space}

\subsection{The spectral sequence associated to the Bruhat-Tits open cover of $\mathcal M^{\mathrm{an}}$}

\paragraph{}Let $\mathcal M^{\mathrm{an}}$ denote the Berkovich generic fiber of the Rapoport-Zink space. It is a smooth analytic space over $E'$ of dimension $n-1$. Let $\mathrm{red}:\mathcal M^{\mathrm{an}}\rightarrow \mathcal M_{\mathrm{red}}$ denote the reduction map. Write $\mathcal M^{\mathrm{an}}_i := \mathrm{red}^{-1}(\mathcal M_{\mathrm{red},i})$ so that 
$$\mathcal M^{\mathrm{an}} = \bigsqcup_{i\in \mathbb Z} \mathcal M^{\mathrm{an}}_i,$$
each $\mathcal M^{\mathrm{an}}_i$ being connected. For a vertex lattice $\Lambda \in \mathcal L_i$, let 
$$U_{\Lambda} := \mathrm{red}^{-1}(\mathcal M_{\Lambda}) \subset \mathcal M^{\mathrm{an}}_i$$
denote the analytical tube of the closed Bruhat-Tits stratum indexed by $\Lambda$. Since the map $\mathrm{red}$ is anticontinuous, it is an open subspace of the generic fiber. The group $J$ acts on $\mathcal M^{\mathrm{an}}$ and the map $\mathrm{red}$ is $J$-equivariant. The action restricts to an action of the maximal compact subgroup $J_{\Lambda}$ on $U_{\Lambda}$.

\paragraph{} We fix a prime number $\ell \not = p$ and we consider the cohomology groups 
$$\mathrm H_c^{\bullet}(\mathcal M^{\mathrm{an}},\overline{\mathbb Q_{\ell}}) := \varinjlim_U \varprojlim_k \mathrm H_c^{\bullet}(U\,\widehat{\otimes}\,\mathbb C_p,\mathbb Z/\ell^k\mathbb Z)\otimes \overline{\mathbb Q_{\ell}},$$
where $U$ runs over all relatively compact open subspaces of $\mathcal M_{\mathrm{an}}$. These cohomology groups are representations of $J\times W$ where $W$ is the absolute Weil group of $E$. The $W$-action on the cohomology group is defined in the following specific way. The inertia $I \subset W$ acts on the coefficients $\mathbb C_p$, whereas the action of the Frobenius is given by \textbf{Rapoport and Zink's descent datum} on $\mathcal M \,\widehat{\otimes}\, \mathcal O_{\widecheck{E}}$. As we recalled in \cite{muller} 4.1.2, this descent datum is an isomorphism $\alpha_{\mathrm{RZ}}:\mathcal M \,\widehat{\otimes}\, \mathcal O_{\widecheck{E}} \xrightarrow{\sim} \sigma^{*}(\mathcal M \,\widehat{\otimes}\, \mathcal O_{\widecheck{E}})$, where $\sigma \in \mathrm{Gal}(\widecheck E/E) \simeq \mathrm{Gal}(\mathbb F/\mathbb F_p)$ is the arithmetic Frobenius. The right-hand side can be identified with the Rapoport-Zink space for $(\mathbb X\otimes \mathbb F)^{(p)}$. This isomorphism is induced by the relative Frobenius $\mathcal F_{\mathbb X}: \mathbb X\otimes \mathbb F \rightarrow (\mathbb X\otimes \mathbb F)^{(p)}$, via $(X,\iota,\lambda,\rho) \mapsto (X,\iota,\lambda,\mathcal F_{\mathbb X} \circ \rho)$. We fix a lift $\mathrm{Frob} \in W$ of the \textit{geometric} Frobenius. Then the action of $\mathrm{Frob}$ on $\mathrm H_c^{\bullet}(\mathcal M^{\mathrm{an}})$ is induced by $\alpha_{\mathrm{RZ}}^{-1}$.\\
Via covariant Dieudonné theory, the relative Frobenius $\mathcal F_{\mathbb X}$ corresponds to the Verschiebung morphism $\mathbf V$ on $C_{\mathbb F} = C\otimes_{E} \widecheck{E}$. If $k$ is any perfect extension of $\kappa(\widecheck E)$, the Verschiebung sends a $k$-rational point $M\in \mathcal M(k)$ to $\mathbf V M = \eta^{\sigma^{-1}}\pi\tau^{-1}(M) = \pi\tau^{-1}(M)$ (since $\eta$ is a scalar unit). Hence, the descent datum $\alpha_{\mathrm{RZ}}$ sends a $k$-point $M$ to $\pi\tau^{-1}(M)$.

\begin{rk}
The descent datum $\alpha_{\mathrm{RZ}}$ is not effective. The rational structure of $\mathcal M$ over $\mathcal O_{E'}$ is induced by $\pi\alpha_{\mathrm{RZ}}^{-1}$ if $E=E_1$, and by $(\pi\alpha_{\mathrm{RZ}}^{-1})^2$ if $E=E_2$. It maps a $k$-point $M$ to $\tau(M)$ in the first case, and to $\tau^2(M)$ in the second case.\\
In any case, we will denote by $\tau$ the action on the cohomology induced by $\pi\alpha_{\mathrm{RZ}}^{-1}$, and we refer to it as the \textbf{rational Frobenius}. We have $\tau = (\pi^{-1}\cdot\mathrm{id},\mathrm{Frob}) \in J\times W$, the $\pi^{-1}$ coming from contravariance of cohomology with compact support. 
\end{rk}

\paragraph{}One also defines in a similar way the cohomology of the connected components $\mathcal M^{\mathrm{an}}_i$. Any element $g\in J$ induces an isomorphism 
$$g:\mathrm H_c^{\bullet}(\mathcal M^{\mathrm{an}}_i,\overline{\mathbb Q_{\ell}}) \xrightarrow{\sim} \mathrm H_c^{\bullet}(\mathcal M^{\mathrm{an}}_{i+\alpha(g)},\overline{\mathbb Q_{\ell}}).$$
Besides, $\mathrm{Frob}$ induces an isomorphism between the cohomology of $\mathcal M^{\mathrm{an}}_i$ and that of $\mathcal M^{\mathrm{an}}_{i+1}$. Let $(J\times W)^{\circ}$ be the subgroup of elements $(g,u\mathrm{Frob}^j)$ where $u\in I$ and $\alpha(g) = -j$. In fact we have $(J\times W)^{\circ} = (J^{\circ}\times I)\tau^{\mathbb Z}$. Then each cohomology group $\mathrm H_c^{\bullet}(\mathcal M^{\mathrm{an}}_i,\overline{\mathbb Q_{\ell}})$ is a $(J\times W)^{\circ}$-representation and we have 
$$\mathrm H_c^{\bullet}(\mathcal M^{\mathrm{an}},\overline{\mathbb Q_{\ell}}) \simeq \mathrm{c-Ind}_{(J\times W)^{\circ}}^{J\times W}\,\mathrm H_c^{\bullet}(\mathcal M^{\mathrm{an}}_i,\overline{\mathbb Q_{\ell}}).$$

\paragraph{}\label{SpectralSequenceGeneric} We write $\mathcal L_i^{\mathrm{max}}$ for the subset of $\mathcal L_i$ consisting only of lattices of orbit type $t_{\mathrm{max}}$, and we write $\mathcal L^{\mathrm{max}}$ for the disjoint union of the $\mathcal L_i^{\mathrm{max}}$. The collection $\{U_{\Lambda}\}_{\Lambda\in \mathcal L^{\mathrm{max}}}$ forms a locally finite open cover of $\mathcal M^{\mathrm{an}}$. By \cite{fargues} Proposition 4.2.2, we obtain a $J$-equivariant spectral sequence
$$E_1^{a,b} = \bigoplus_{\gamma \in I_{-a+1}} \mathrm H_c^b(U(\gamma),\overline{\mathbb Q_{\ell}}) \implies \mathrm H_c^{a+b}(\mathcal M^{\mathrm{an}},\overline{\mathbb Q_{\ell}}).$$
Here for $s\geq 1$ the index set is given by 
$$I_s := \left\{ \gamma \in \mathcal L^{\mathrm{max}} \,\middle | \, \#\gamma = s \text{ and } U(\gamma) := \bigcap_{\Lambda\in\gamma}U_{\Lambda} \not = \emptyset \right\}.$$
Note that if $\gamma \in I_s$ then there exists $i\in \mathbb Z$ such that $\gamma \subset \mathcal L_i^{\mathrm{max}}$ and $U(\gamma) = U_{\Lambda(\gamma)}$ where $\Lambda(\gamma) := \bigcap_{\Lambda\in\gamma}\Lambda \in \mathcal L_i$. There is a natural action of $J$ on $I_s$, and an element $g\in J$ induces an isomorphism between the cohomology of $U(\gamma)$ and that of $U(g\cdot\gamma)$.

\begin{rk}
The descent datum $\pi\alpha_{\mathrm{RZ}}^{-1}$, mapping a $k$-point $M$ to $\tau(M)$, induces the $\mathbb F_p$-rational structure on $\mathcal M_{\Lambda}\otimes \mathbb F$. It induces an action of $\tau$ on the cohomology of $U_{\Lambda}$. Since for any $\gamma \in I_s$ we have $\pi\cdot\gamma \in I_s$, each term $E_1^{a,b}$ carries a $W$-representation. The spectral sequence is then $J\times W$-equivariant.
\end{rk}

\paragraph{}\label{SpectralSequenceIsInduced} One may also consider the similar $(J\times W)^{\circ}$-equivariant spectral sequence for the $0$-th connected component $\mathcal M^{\mathrm{an}}_0$,
$$E'^{a,b}_1: \bigoplus_{\gamma \in I'_{-a+1}} \mathrm H^b_c(U(\gamma),\overline{\mathbb Q_{\ell}}) \implies \mathrm H^{a+b}_c(\mathcal M^{\mathrm{an}}_0,\overline{\mathbb Q_{\ell}}),$$
where $I'_{-a+1}$ is defined as $I_{-a+1}$ except that we only consider subsets $\gamma$ of $\mathcal L_0^{(m)}$. The compact induction operation $\mathrm{c-Ind}_{J^{\circ}}^{J}$ is an additive exact functor between the categories of smooth representations of $J^{\circ}$ and of $J$. We have 
$$E = \mathrm{c-Ind}_{J^{\circ}}^J \,E',$$
in the sense that all terms and differentials in the sequence $E$ are the compact induction of the terms and differentials of $E'$. This observation will play a crucial role later in \ref{VanishingExt1}.

\paragraph{}\label{TrivialNearbyCycles} For $\Lambda \in \mathcal L$, the cohomology groups $\mathrm H_c^{\bullet}(U_{\Lambda},\overline{\mathbb Q_{\ell}})$ are representations of the subgroup $(J_{\Lambda}\times I)\cdot\tau^{\mathbb Z} \subset (J\times W)^{\circ}$. They are related to the cohomology of the special fiber $\mathcal M_{\Lambda}$ by the following proposition. 

\begin{prop}
Let $\Lambda \in \mathcal L$. There is a natural $(J_{\Lambda}\times I)\cdot\tau^{\mathbb Z}$-equivariant isomorphism 
$$\mathrm H_c^b(U_{\Lambda},\overline{\mathbb Q_{\ell}}) \xrightarrow{\sim} \mathrm H_c^{2(n-1)-b}(\mathcal M_{\Lambda},\overline{\mathbb Q_{\ell}})^{\vee}(n-1),$$
where on the right-hand side the inertia $I$ acts trivially and the rational Frobenius $\tau$ acts like the Frobenius $F$.
\end{prop}

\begin{proof}
By the same arguments as in \cite{muller} 4.1.5 Proposition, we have an isomorphism 
$$\mathrm H^b(U_{\Lambda},\overline{\mathbb Q_{\ell}}) \xrightarrow{\sim} \mathrm H^b(\mathcal M_{\Lambda},\overline{\mathbb Q_{\ell}}).$$
This requires the fact that the integral model of the Shimura variety $\mathrm S_{K^p}$ with hyperspecial level at $p$ is smooth, so that the nearby cycles sheaf is trivial. Since $\mathcal M_{\Lambda}$ is projective, the right-hand side coincide with the cohomology with compact support. On the other hand, we apply Poincaré duality on the left-hand side to obtain 
$$\mathrm H_c^b(U_{\Lambda},\overline{\mathbb Q_{\ell}}) \simeq \mathrm H^{2(n-1)-b}(U_{\Lambda},\overline{\mathbb Q_{\ell}})^{\vee}(n-1) \simeq \mathrm H^{2(n-1)-b}_c(\mathcal M_{\Lambda},\overline{\mathbb Q_{\ell}})^{\vee}(n-1).$$
\end{proof}

\begin{rk}
The cohomology groups $\mathrm H_c^{\bullet}(\mathcal M_{\Lambda})$ decompose as a sum of irreducible unipotent representations of $\mathrm{GSp}(2\theta,\mathbb F_p)$, inflated to $J_{\Lambda}$. The smallest field of definition of unipotent representations of classical groups is $\mathbb Q$ by \cite{LusztigRational}, therefore they are autodual. Thus, we have a $\mathrm{GSp}(2\theta,\mathbb F_p)$-equivariant isomorphism $\mathrm H_c^{\bullet}(\mathcal M_{\Lambda})^{\vee} \simeq \mathrm H_c^{\bullet}(\mathcal M_{\Lambda})$, but it is not equivariant for the action of the Frobenius $F$.
\end{rk}

\noindent The situation is less favorable than in the unramified case since the Frobenius action on the cohomology of $\mathcal M_{\Lambda}$, and consequently of $U_{\Lambda}$ as well, is not pure (at least when $t(\Lambda) \geq 6$). Therefore, \cite{muller} 4.1.8 Corollary does not seem to hold in general, ie. one may not deduce from the previous proposition that the spectral sequence degenerates on the second page, splits and that $\tau$ acts semi-simply on the abutment. However, the spectral sequence does eventually degenerate in deeper pages since the non-zero terms $E_1^{a,b}$ are concentrated in a finite strip. In particular, the inertia acts trivially on $\mathrm H_c^{\bullet}(\mathcal M^{\mathrm{an}})$.

\paragraph{}\label{SpectralSequenceGenericBis} In order to analyze the $J$-action on $E_1^{a,b}$, we rewrite the direct sum by making compactly induced representations appear. For $s\geq 1$ we define 
$$I_s^{(\theta)} := \{\gamma \in I_s \,|\, t(\Lambda(\gamma)) = 2\theta\}.$$
We denote by $N(\Lambda_{\theta})$ the set $N(\theta_{\mathrm{max}} - \theta,V^1_{\theta})$ that we defined in \ref{CountingBTStrata}. It corresponds to the set of lattices $\Lambda \in \mathcal L_0$ of orbit type $t_{\mathrm{max}}$ containing $\Lambda_{\theta}$. We then define 
$$K_s^{(\theta)} := \{\gamma \subset N(\Lambda_{\theta}) \,|\, \#\gamma = s \text{ and } \Lambda(\gamma) = \Lambda_{\theta}\}.$$
If $\gamma \in I_s^{(\theta)}$ then there exists some $g \in J$ such that $g\cdot\Lambda(\gamma) = \Lambda_{\theta}$. Therefore $g\cdot\gamma \in K_s^{(\theta)}$, and the coset $J_{\theta}\cdot g$ is uniquely determined. This induces a bijection of orbit sets 
$$J\backslash I_s^{(\theta)} \xrightarrow{\sim} J_{\theta}\backslash K_s^{(\theta)}.$$
Let us write $k_{s,\theta} := \# K_s^{(\theta)} \geq 0$. 

\begin{prop}
We have 
$$E_1^{a,b} = \bigoplus_{\theta = 0}^{\theta_{\mathrm{max}}} \left(\mathrm{c-Ind}_{J_{\theta}}^J \, \mathrm H_c^b(U_{\Lambda_{\theta}},\overline{\mathbb Q_{\ell}})\right)^{k_{-a+1,\theta}}.$$
\end{prop}

\noindent The proof is strictly identic to \cite{muller} 4.1.10 Proposition. 

\paragraph{}\label{MinimalEigenvalue1} By definition, we have $k_{s,\theta_{\mathrm{max}}} = \delta_{1,s}$ and $k_{1,\theta} = \delta_{\theta_{\mathrm{max}},\theta}$. Observe the term $E_1^{0,2(n-1-\theta_{\mathrm{max}})}$. Using \ref{TrivialNearbyCycles} Proposition, it is isomorphic to 
$$\mathrm{c-Ind}_{J_{\theta_{\mathrm{max}}}}^{J} \, \mathrm H_c^{2(n-1-\theta_{\mathrm{max}})}(U_{\Lambda_{\theta_\mathrm{max}}},\overline{\mathbb Q_{\ell}}) \simeq \mathrm{c-Ind}_{J_{\theta_{\mathrm{max}}}}^{J} \, \mathrm H_c^{2\theta_{\mathrm{max}}}(\mathcal M_{\Lambda_{\theta_{\mathrm{max}}}},\overline{\mathbb Q_{\ell}})^{\vee}(n-1).$$
According to \ref{CohomologyS} Theorem, we deduce that $E_1^{0,2(n-1-\theta_{\mathrm{max}})} \simeq \mathrm{c-Ind}_{J_{\theta_{\mathrm{max}}}}^{J} \, \mathbf 1$ with the rational Frobenius $\tau$ acting like multiplication by $p^{n-1-\theta_{\mathrm{max}}}$. Let us consider another non-zero term $E_1^{a,b}$ in the spectral sequence. There must be some vertex lattice $\Lambda \in \mathcal L$ which can be obtained as the intersection of $-a+1$ vertex lattices of maximal type and such that $\mathrm H_c^{b}(U_{\Lambda},\overline{\mathbb Q_{\ell}})$ is not zero. If $a\leq -1$ then $t(\Lambda) = 2\theta$ with $\theta < \theta_{\mathrm{max}}$ and by \ref{TrivialNearbyCycles} Proposition, we have that $0 \leq 2(n-1)-b \leq 2\theta$. The possible Frobenius eigenvalues on $\mathrm H_c^{2(n-1)-b}(\mathcal M_{\Lambda})$ have the form $-p^{j+1}$ for some $j\geq 0$ and $p^i$ for some $0 \leq i \leq n-1-b+\lceil b/2 \rceil$ according to \ref{CohomologyS} Theorem. After taking dual and Tate twist by $n-1$, it follows that the possible Frobenius eigenvalues on $H_c^{b}(U_{\Lambda})$ have the form $-p^{j+1}$ for some $j$ and $p^i$ for some $b-\lceil b/2 \rceil \leq i \leq n-1$. Since $b \geq 2(n-1-\theta)$ we have $b-\lceil b/2 \rceil \geq n-1-\theta$. In particular, we conclude that $i > n-1-\theta_{\mathrm{max}}$ so that the eigenvalue $p^{n-1-\theta_{\mathrm{max}}}$ does not appear in $E_1^{a,b}$, and so neither in $E_k^{a,b}$ in the deeper pages.\\
To sum up, we have observed that the eigenvalue $p^{n-1-\theta_{\mathrm{max}}}$ of $\tau$ only appears in the term of coordinates $(0,2(n-1-\theta_{\mathrm{max}}))$ and nowhere else. The Frobenius equivariance of the spectral sequence forces all differentials connected to this term to be zero. It follows that for any $k\geq 1$, we have $E_k^{0,2(n-1-\theta_{\mathrm{max}})} \simeq E_1^{0,2(n-1-\theta_{\mathrm{max}})}$. Since this term is the last nonzero term of its diagonal, it contributes to a subspace of $\mathrm H_c^{2(n-1-\theta_{\mathrm{max}})}(\mathcal M^{\mathrm{an}})$. Thus, we have obtained the following statement. 

\begin{theo}
There is a $J\times W$-equivariant monomorphism 
$$\mathrm{c-Ind}_{J_{\theta_{\mathrm{max}}}}^{J} \, \mathbf 1 \hookrightarrow \mathrm H_c^{2(n-1-\theta_{\mathrm{max}})}(\mathcal M^{\mathrm{an}}),$$
where on the left-hand side the inertia acts trivially and $\tau$ acts like multiplication by the scalar $p^{n-1-\theta_{\mathrm{max}}}$.
\end{theo} 

\paragraph{}\label{IrreducibleSubquotients} In \cite{muller} 4.2, we made a summary of a general analysis of compactly induced representation in regards to type theory, using results of \cite{bk}, \cite{bushnell} and \cite{morris}. It leads to the following important consequence. For $n = 1$ and for $n=2$ with $C$ non-split, we have $\theta_{\mathrm{max}} = 0$ and the trivial representation $\mathbf 1$ is cuspidal for the maximal reductive quotient $\mathcal J_{0} \simeq \mathrm G(\mathrm{GSp}(0,\mathbb F_p)\times \mathrm{GO}(V_{\Lambda_0}^1))$. Here, note that $V_{\Lambda_0}^1$ has dimension $n$, and if $n=2$ with $C$ non-split then $\mathrm{GO}(V_{\Lambda_0}^1) \simeq \mathrm{GO}^{-}(2,\mathbb F_p)$ is the non-split finite orthogonal group in two variables. We then define 
$$\sigma_0 := \mathrm{c-Ind}_{\mathrm N_{J}(J_0)}^J \, \mathbf 1,$$ 
where $\mathrm N_{J}(J_0)$ is the normalizer of $J_0$ in $J$. According to \cite{morris} 4.1 Proposition, $\sigma_0$ is an irreducible supercuspidal representation of $J$.\\
If $n\geq 3$ or if $n=2$ with $C$ split, then the trivial representation is not cuspidal for $\mathcal J_{\theta_{\mathrm{max}}}$.\\
Eventually, if $V$ is any smooth representation of $J$ and if $\chi$ is any smooth character of $\mathrm{Z}(J) \simeq E^{\times}$, we denote by $V_{\chi}$ the largest quotient of $V$ on which $\mathrm{Z}(J)$ acts through $\chi$.

\begin{prop}
Let $\chi$ be an unramified character of $E^{\times}$. 
\begin{enumerate}[label=\upshape (\arabic*), topsep = 0pt]
\item If $n=1$ or $n=2$ with $C$ non-split, all irreducible subquotients of $V := \mathrm{c-Ind}_{J_0}^{J} \, \mathbf 1$ are supercuspidal, and are isomorphic to an unramified twist of $\sigma_0$. In this case, we have $V_{\chi} \simeq \sigma_0\otimes\chi$.
\item If $n\geq 3$ or if $n=2$ with $C$ split, then no irreducible subquotient of $V := \mathrm{c-Ind}_{J_{\theta_{\mathrm{max}}}}^{J} \, \mathbf 1$ is supercuspidal. In this case, $V_{\chi}$ does not contain any non-zero admissible subrepresentation of $J$.
\end{enumerate}
\end{prop}

\noindent Combining this proposition with the previous paragraph, we deduce the following corollary.

\begin{corol}
If $n\geq 3$ or $n=2$ with $C$ split, and if $\chi$ is any unramified character of $E^{\times}$, then $\mathrm H_c^{2(n-1-\theta_{\mathrm{max}})}(\mathcal M^{\mathrm{an}})_{\chi}$ is not $J$-admissible.
\end{corol}

\paragraph{}\label{MinimalEigenvalue2} We may repeat the exact same arguments by looking this time at the smallest eigenvalue of the form $-p^{j+1}$ in the spectral sequence. To alleviate the notations, for any integer $\theta\geq 2$ we will write 
$$\rho_{\theta} := \rho_{\footnotesize \begin{pmatrix}
0 & 1 & \theta \\
{}
\end{pmatrix}}.$$
Assume that $n\geq 5$ or that $n=4$ with $C$ split, so that $\theta_{\mathrm{max}} \geq 2$. By \ref{TrivialNearbyCycles} Proposition and by \ref{CohomologyT} Theorem (3), the eigenspace of $\tau$ for $-p^{n-\theta_{\mathrm{max}}}$ in $E_1^{0,2(n-\theta_{\mathrm{max}})}$ is isomorphic to 
$$\mathrm{c-Ind}_{J_{\theta_{\mathrm{max}}}}^J\, \rho_{\theta_{\mathrm{max}}},$$
where $\rho_{\theta_{\mathrm{max}}}$ also stands for the inflation to $J_{\theta_{\mathrm{max}}}$ of the representation $\rho_{\theta_{\mathrm{max}}} \boxtimes \mathbf 1$ of the reductive quotient $\mathcal J_{\theta_{\mathrm{max}}} \simeq \mathrm G(\mathrm{GSp}(2\theta_{\mathrm{max}},\mathbb F_p) \times \mathrm{GO}(V_{\Lambda_{\theta_{\mathrm{max}}}}^1))$. We observe that the eigenvalue $-p^{n-\theta_{\mathrm{max}}}$ does not appear anywhere else in the spectral sequence. Therefore, we have 

\begin{theo}
Assume that $n\geq 5$ or that $n=4$ with $C$ split. There is a $J\times W$-equivariant monomorphism 
$$\mathrm{c-Ind}_{J_{\theta_{\mathrm{max}}}}^{J} \, \rho_{\theta_{\mathrm{max}}} \hookrightarrow \mathrm H_c^{2(n-\theta_{\mathrm{max}})}(\mathcal M^{\mathrm{an}}),$$
where on the left-hand side the inertia acts trivially and $\tau$ acts like multiplication by the scalar $-p^{n-\theta_{\mathrm{max}}}$.
\end{theo}

\noindent If $n = 4$ with $C$ split, if $n=5$ or if $n=6$ with $C$ non-split, we have $\theta_{\mathrm{max}} = 2$ and $\rho_{2}\boxtimes \mathbf 1$ is a cuspidal representation of $\mathcal J_{2} \simeq \mathrm G(\mathrm{GSp}(4,\mathbb F_p) \times \mathrm{GO}(V_{\Lambda_{\theta_{\mathrm{max}}}}^1))$ (note that $V_{\Lambda_{\theta_{\mathrm{max}}}}^1$ has dimension $n-4$). In this case, we define 
$$\sigma_1 := \mathrm{c-Ind}_{\mathrm N_{J}(J_2)}^J \, \widetilde{\rho_2},$$
where $\widetilde{\rho_2}$ is defined the following way. One may prove that $\mathrm{N}_{J}(J_2) = \mathrm Z(J)J_2$, and the restriction of $\rho_2$ to $\mathrm Z(J) \cap J_2$ is trivial since $\rho_2$ is unipotent. Thus, we may extend $\rho_2$ to a representation $\widetilde{\rho_2}$ of the normalizer $\mathrm{N}_{J}(J_2)$ by letting the center $\mathrm Z(J)$ act trivially. Then $\sigma_1$ is an irreducible supercuspidal representation of $J$. \\
If $n\geq 7$ or if $n=6$ with $C$ split, then $\rho_{\theta_{\mathrm{max}}}\boxtimes \mathbf 1$ is not a cuspidal representation of $\mathcal J_{\theta_{\mathrm{max}}}$. We deduce the following consequences.

\begin{prop}
Let $\chi$ be an unramified character of $E^{\times}$. 
\begin{enumerate}[label=\upshape (\arabic*), topsep = 0pt]
\item If $n=4$ with $C$ split, if $n=5$ or if $n=6$ with $C$ non-split, all irreducible subquotients of $V := \mathrm{c-Ind}_{J_2}^{J} \, \rho_2$ are supercuspidal, and are isomorphic to an unramified twist of $\sigma_1$. In this case, we have $V_{\chi} \simeq \sigma_1\otimes\chi$.
\item If $n\geq 7$ or if $n=6$ with $C$ split, then no irreducible subquotient of $V := \mathrm{c-Ind}_{J_{\theta_{\mathrm{max}}}}^{J} \, \rho_{\theta_{\mathrm{max}}}$ is supercuspidal. In this case, $V_{\chi}$ does not contain any non-zero admissible subrepresentation of $J$.
\end{enumerate}
\end{prop}

\begin{corol}
If $n\geq 7$ or $n=6$ with $C$ split, and if $\chi$ is any unramified character of $E^{\times}$, then $\mathrm H_c^{2(n-\theta_{\mathrm{max}})}(\mathcal M^{\mathrm{an}})_{\chi}$ is not $J$-admissible.
\end{corol}

\paragraph{}\label{HighestCohomologyGroup} We finish this section with the following observation regarding the cohomology group (with compact support) of highest degree. 

\begin{prop}
There is an isomorphism 
$$\mathrm H_c^{2(n-1)}(\mathcal M^{\mathrm{an}},\overline{\mathbb Q_{\ell}}) \simeq \mathrm{c-Ind}_{J^{\circ}}^J \, \mathbf 1,$$
where $\mathbf 1$ denotes the trivial representation, and where $\mathrm{Frob}$ acts like $p^{n-1}\cdot\mathrm{id}$. 
\end{prop}

\noindent The proof is identic to \cite{muller} 4.1.13 Proposition.

\subsection{The spectral sequence for small values of $n$}

\paragraph{} As computed in the first section, the cohomology of $S_{\theta}$ for $\theta \leq 1$ is entirely known with pure Frobenius action. Thus, as long as $\theta_{\mathrm{max}} \leq 1$, ie. if $n\leq 3$ or if $n=4$ with $C$ non-split, we understand all the terms $E_1^{a,b}$ in the spectral sequence. In order to get a better image of the situation, let us have a deeper look at the first page in these cases.

\paragraph{}If $\theta_{\mathrm{max}} = 0$, ie. if $n = 1$ or $n = 2$ and $C$ is non-split, then $C$ is anisotropic and $\mathcal L_i$ is a singleton. There is only one non-zero term in the spectral sequence, equal to $\mathrm{c-Ind}_{J^{\circ}}^{J}\, \mathbf 1$, which computes the cohomology group $\mathrm H_c^{2(n-1)}(\mathcal M^{\mathrm{an}})$ as we already checked in \ref{HighestCohomologyGroup}. Thus, there is not much to say in this case. 

\paragraph{} We assume now that $\theta_{\mathrm{max}} = 1$, ie. $n = 2$ and $C$ is split, $n=3$ or $n=4$ and $C$ is non-split. In this case, the orbit type $2\theta$ of any vertex lattice is $0$ or $2$, so that $\theta = 0$ or $1$. We have $k_{s,1} = \delta_{1,s}$ by definition of $K_s^{(1)}$. Moreover, we have 
$$K_s^{(0)} := \{\gamma \subset N(\Lambda_0) \,|\, \#\gamma = s \text{ and } \Lambda(\gamma) = \Lambda_0\}.$$
If $s = 1$ then this set is empty, since $\Lambda(\gamma)$ would have orbit type $2$. However if $\gamma \subset N(\Lambda_0)$ has cardinality at least $2$, then the condition $\Lambda(\gamma) = \Lambda_0$ is automatically satisfied. Indeed, $\Lambda(\gamma)$ is a vertex lattice of orbit type strictly less than $2$, thus equal to $0$, and which contains $\Lambda_0$. Therefore, $K_s^{(0)}$ is simply the set of subsets of $N(\Lambda_0)$ of cardinality $s$. By \ref{CountingBTStrata}, we have 

$$
\#N(\Lambda_0) = 
\begin{cases}
2 & \text{if }n = 2 \text{ and } C \text{ is split,}\\
p+1 & \text{if } n = 3,\\
p^2+1 & \text{if } n = 4 \text{ and } C \text{ is non-split.}
\end{cases}
$$

\noindent Therefore $k_{s,0}$ is $0$ if $s = 1$, and for $s\geq 2$ we have 

$$
k_{s,0} = 
\begin{cases}
2\choose s & \text{if }n = 2 \text{ and } C \text{ is split,}\\
{p+1}\choose s & \text{if } n = 3,\\
{p^2+1}\choose s & \text{if } n = 4 \text{ and } C \text{ is non-split.}
\end{cases}
$$

\noindent In Figure 2 and 3, we draw the first page of the spectral sequence respectively when $n=2$ and $C$ is split, and when $n=3$. In brackets we have written the scalar by which $\tau$ acts on each term. The spectral sequence in the case $n=4$ with $C$ non-split is similar to Figure 3, except that two more 0 rows must be added at the bottom, and all eigenvalues of $\tau$ are multiplied by $p$. The morphisms on the top row are denoted $\varphi_i$ for $i=1$ if $n=2$ with $C$ split, for $1\leq i \leq p$ if $n=3$ and for $1\leq i \leq p^2$ if $n=4$ with $C$ non-split. Given the shape of these sequences and taking into account the Frobenius weights of each term, we observe that they degenerate on the second page.

\begin{figure}[h]
\centering
\begin{tikzcd}
\mathrm{c-Ind}^J_{J_{0}}\,\mathbf{1}[p] \arrow{r}{\varphi_1} & \mathrm{c-Ind}^J_{J_{1}}\mathbf{1}[p] \\
\, & 0 \\
\, & \mathrm{c-Ind}^J_{J_{1}}\,\mathbf{1}[1]
\end{tikzcd}
\caption{The first page $E_1$ when $n=2$ and $C$ is split.}
\end{figure}

\begin{figure}[h]
\hspace{-40pt}
\begin{tikzcd}
\ldots \arrow{r}{\varphi_4} & \left(\mathrm{c-Ind}^J_{J_{0}}\,\mathbf{1}\right)^{k_{4,0}}[p^2] \arrow{r}{\varphi_3} & \left(\mathrm{c-Ind}^J_{J_{0}}\,\mathbf{1}\right)^{k_{3,0}}[p^2] \arrow{r}{\varphi_2} & \left(\mathrm{c-Ind}^J_{J_{0}}\,\mathbf{1}\right)^{k_{2,0}}[p^2] \arrow{r}{\varphi_1} & \mathrm{c-Ind}^J_{J_{1}}\mathbf{1}[p^2] \\
\, & \, & \, & \, & 0 \\
\, & \, & \, & \, & \mathrm{c-Ind}^J_{J_{1}}\,\mathbf{1}[p]\\
\, & \, & \, & \, & 0 \\
\, & \, & \, & \, & 0 
\end{tikzcd}
\caption{The first page $E_1$ when $n=3$.}
\end{figure}

\section{The cohomology of the basic stratum of the Shimura variety for small values of $n$}

\subsection{The Hochschild-Serre spectral sequence induced by $p$-adic uniformization}

\paragraph{}\label{ClassificationAlgebraicRepresentations} Related to the $p$-adic uniformization theorem in \ref{Uniformization}, Fargues has built in \cite{fargues} a spectral sequence relating the cohomology of $\mathcal M^{\mathrm{an}}$ to that of $\widehat{\mathrm S}_{K^p}^{\mathrm{ss},\mathrm{an}}$. Even though the construction of loc. cit. is done in the context of unramified Rapoport-Zink spaces, it works in greater generality as mentioned in the last paragraph of 4.5.2.1.\\
Recall the notations of section 2.3. Let $\xi: \mathbb G \rightarrow W_{\xi}$ be a finite-dimensional irreducible algebraic representation of $\mathbb G$ over $\overline{\mathbb Q_{\ell}}$. In \cite{muller} 5.1.2, we recalled the classification of all such representations $\xi$ following \cite{harris} III.2. Let $\mathbb V_0$ denote the dual of $\mathbb V \otimes \overline{\mathbb Q_{\ell}}$ on which $\mathbb G$ acts. There exists uniquely defined integers $t(\xi), m(\xi) \geq 0$ and an idempotent $\epsilon(\xi) \in \mathrm{End}(\mathbb V_0^{\otimes m(\xi)})$ such that 
$$W_{\xi} \simeq c^{t(\xi)}\otimes\epsilon(\xi)(\mathbb V_0^{\otimes m(\xi)}),$$
where $c$ denotes the similitude factor. The weight of $\xi$ is defined by 
$$w(\xi) := m(\xi) - 2t(\xi).$$
One can associate to $\xi$ a local system $\mathcal L_{\xi}$ on the tower $(\mathrm{S}_{K^p})_{K^p}$ of Shimura varieties. Let $\mathcal A_{K^p}$ be the universal abelian scheme over $\mathrm{S}_{K^p}$. We write $\pi_{K^p}^m : \mathcal A_{K^p}^m \to \mathrm{S}_{K^p}$ for the structure morphism of the $m$-fold product of $\mathcal A_{K^p}$ with itself over $\mathrm{S}_{K^p}$. Then
$$\mathcal L_{\xi} \simeq \epsilon(\xi)\epsilon_{m(\xi)} \left( \mathrm{R}^{m(\xi)}({\pi_{K^p}^{m(\xi)}})_{*}\overline{\mathbb Q_{\ell}}(t(\xi))\right),$$
where $\epsilon_{m(\xi)}$ is some idempotent. We denote by $\overline{\mathcal L_{\xi}}$ its restriction to the special fiber $\overline{\mathrm S}_{K^p}$. 

\paragraph{}\label{FarguesSpectralSequence} Let $\mathcal A_{\xi}$ be the space of \textbf{automorphic forms of $I$ of type $\xi$ at infinity}. Explicitly, it is given by 
$$\mathcal A_{\xi} = \left\{f: I(\mathbb A_{f})\rightarrow W_{\xi} \,|\, f \text{ is } I(\mathbb A_{f}) \text{-smooth by right translations and } \forall \gamma \in I(\mathbb Q), f(\gamma\,\cdot) = \xi(\gamma)f(\cdot)\right\}.$$
We denote by $\mathcal L_{\xi}^{\mathrm{an}}$ the analytification of $\mathcal L_{\xi}$. 

\noindent \begin{notation}
We write $\mathrm{H}^{\bullet}(\widehat{\mathrm S}_{K^p}^{\mathrm{ss},\mathrm{an}},\mathcal L_{\xi}^{\mathrm{an}})$ for the cohomology of $\widehat{\mathrm S}_{K^p}^{\mathrm{ss},\mathrm{an}} \, \widehat{\otimes} \, \mathbb C_p$ with coefficients in $\mathcal L_{\xi}^{\mathrm{an}}$.
\end{notation}

\noindent \begin{theo}[\cite{fargues} 4.5.12]
There is a $W$-equivariant spectral sequence 
$$F_2^{a,b}(K^p) : \mathrm{Ext}_{J}^a \left (\mathrm H_c^{2(n-1)-b}(\mathcal M^{\mathrm{an}}, \overline{\mathbb Q_{\ell}})(1-n), \mathcal A^{K^p}_{\xi}\right) \implies \mathrm H^{a+b}(\widehat{\mathrm S}_{K^p}^{\mathrm{ss},\mathrm{an}}, \mathcal L_{\xi}^{\mathrm{an}}).$$
These spectral sequences are compatible as the open compact subgroup $K^p$ varies in $\mathbb G(\mathbb A_f^p)$. 
\end{theo}

\noindent We may take the limit $\varinjlim_{K^p}$ for all terms and obtain a $\mathbb G(\mathbb A_f^p) \times W$-equivariant spectral sequence. According to \cite{fargues} Lemme 4.4.12, we have $F_2^{a,b} = 0$ when $a > \theta_{\mathrm{max}}$ since $\theta_{\mathrm{max}}$ is also the semisimple rank of $J$. Since the Shimura variety $\mathrm S_{K^p}$ is smooth, the comparison theorem \cite{berk2} Corollary 3.7 of Berkovich gives an isomorphism
$$\mathrm{H}^{a+b}_c(\overline{S}_{K^p}^{\mathrm{ss}}, \overline{\mathcal L_{\xi}}) = \mathrm{H}^{a+b}(\overline{S}_{K^p}^{\mathrm{ss}}, \overline{\mathcal L_{\xi}}) \xrightarrow{\sim} \mathrm H^{a+b}(\widehat{\mathrm S}_{K^p}^{\mathrm{ss},\mathrm{an}}, \mathcal L_{\xi}^{\mathrm{an}}),$$
where first equality follows from the supersingular locus being a proper variety. Since $\dim \overline{S}_{K^p}^{\mathrm{ss}} = \theta_{\mathrm{max}}$ by \cite{RTW} Theorem 7.2, the cohomology $\mathrm H^{\bullet}(\widehat{\mathrm S}_{K^p}^{\mathrm{ss},\mathrm{an}}, \mathcal L_{\xi}^{\mathrm{an}})$ is concentrated in degrees $0$ to $2\theta_{\mathrm{max}}$.

\paragraph{}\label{FarguesSpectralSequenceBis} Let $\mathcal A(I)$ denote the set of all automorphic representations of $I$ counted with multiplicities, and let $\widecheck{\xi}$ be the contragredient of $\xi$. We also define 
$$\mathcal A_{\xi}(I) := \{\Pi \in \mathcal A(I) \,|\, \Pi_{\infty} = \widecheck{\xi}\}.$$ 
According to \cite{fargues} 4.6, we have an identification 
$$\mathcal A_{\xi}^{K_p} \simeq \bigoplus_{\Pi\in\mathcal A_{\xi}(I)} \Pi_p \otimes (\Pi^p)^{K_p}.$$
We deduce that
$$F_2^{a,b} := \varinjlim_{K^p} F_2^{a,b}(K^p) \simeq \bigoplus_{\Pi\in\mathcal A_{\xi}(I)} \mathrm{Ext}_{J}^a \left (\mathrm H_c^{2(n-1)-b}(\mathcal M^{\mathrm{an}}, \overline{\mathbb Q_{\ell}})(1-n), \Pi_p\right) \otimes \Pi^p.$$
The spectral sequence defined by the terms $F_2^{a,b}$ computes the cohomology of $\overline{S}^{\mathrm{ss}} := \varinjlim_{K^p} \overline{S}_{K^p}^{\mathrm{ss}}$.

\paragraph{}\label{FrobeniusActionOnHom} Let us focus on the Frobenius action on the $\mathrm{Ext}$ groups occuring in the spectral sequence. To this effect, we need the following lemma. For $\Pi \in \mathcal A_{\xi}(I)$, let $\omega_{\Pi}$ denote the central character. For any isomorphism $\iota:\overline{\mathbb Q_{\ell}} \simeq \mathbb C$, we define $|\cdot|_{\iota} := |\iota(\cdot)|$. 

\begin{lem}
We have $|\omega_{\Pi_p}(\pi^{-1}\cdot\mathrm{id})|_{\iota} = p^{w(\xi)/2}$.
\end{lem}

\begin{proof}
We have 
$$|\omega_{\Pi_p}(\pi^{-1}\cdot\mathrm{id})|_{\iota}^2 = |\omega_{\Pi_p}(p^{-1}\cdot\mathrm{id})|_{\iota}.$$
Indeed, $\pi^2$ is equal to $p$ up to a unit in $\mathbb Z_{p}^{\times}$. The value of the central character $\omega_{\Pi_p}$ at this unit has complex modulus $1$ under any isomorphism $\iota:\overline{\mathbb Q_{\ell}} \simeq \mathbb C$. Since $I$ is the group of unitary similitudes of some $\mathbb E/\mathbb Q$-hermitian space, its center is isomorphic to $\mathbb E^{\times}\cdot\mathrm{id}$. In particular, the element $p^{-1}\cdot\mathrm{id}$ in $\mathrm Z(J)$ can be seen as the image of $p^{-1}\cdot\mathrm{id}$ in $\mathrm{Z(I(\mathbb Q))}$. We have $\omega_{\Pi}(p^{-1}\cdot\mathrm{id}) = 1$, and at every finite place $q$ different from $p$ we have $|\omega_{\Pi_q}(p^{-1}\cdot\mathrm{id})|_{\iota} = 1$, since $p^{-1}\cdot\mathrm{id}$ lies in the maximal compact subgroup of $\mathrm Z(I(\mathbb Q_q))$. Eventually, the fact that $\Pi_{\infty} = \widecheck{\xi}$ implies that
$$|\omega_{\Pi_p}(p^{-1}\cdot\mathrm{id})|_{\iota} = |\omega_{\widecheck{\xi}}(p^{-1}\cdot\mathrm{id})|_{\iota}^{-1} = |\omega_{\xi}(p^{-1}\cdot\mathrm{id})|_{\iota} = p^{w(\xi)},$$
the last equality being a consequence of $W_{\xi} \simeq c^{t(\xi)}\otimes\epsilon(\xi)(\mathbb V_0^{\otimes m(\xi)})$ (see \ref{ClassificationAlgebraicRepresentations}).
\end{proof}

\noindent Let us fix a square root $\pi_{\ell}$ of $p$ in $\overline{\mathbb Q_{\ell}}$. We define 
$$\delta_{\Pi_p} := \omega_{\Pi_p}(\pi^{-1}\cdot\mathrm{id})\pi_{\ell}^{-w(\xi)}.$$
The lemma implies that $|\delta_{\Pi_p}|_{\iota} = 1$ for any isomorphism $\iota:\overline{\mathbb Q_{\ell}} \simeq \mathbb C$.\\
By convention, the action of $\mathrm{Frob}$ on a space $\mathrm{Ext}_{J\text{-sm}}^a (\mathrm H_c^{2(n-1)-b}(\mathcal M^{\mathrm{an}}, \overline{\mathbb Q_{\ell}})(1-n), \Pi_p)$ is given by functoriality of $\mathrm{Ext}$ applied to $\mathrm{Frob}^{-1}$ acting on the cohomology of $\mathcal M^{\mathrm{an}}$. Recall that the action of $\mathrm{Frob}$ on the cohomology is the composition of $\tau$ and of $\pi\cdot\mathrm{id} \in J$. Let $P_{\bullet}$ be a projective resolution of $\mathrm H_c^{2(n-1)-b}(\mathcal M^{\mathrm{an}}, \overline{\mathbb Q_{\ell}})(1-n)$ in the category of smooth representations of $J$. Let $\mathcal T:P_{\bullet} \rightarrow P_{\bullet}$ be a lift of $\tau^{-1}$ as a morphism of chain complexes. For $a \geq 0$, the action of $\mathrm{Frob}$ on an element of $\mathrm{Ext}_{J\text{-sm}}^a (\mathrm H_c^{2(n-1)-b}(\mathcal M^{\mathrm{an}}, \overline{\mathbb Q_{\ell}})(1-n), \Pi_p)$ represented by a function $f:P_a\to \Pi_p$, is given by 
$$\mathrm{Frob}^*f(v) = f((\pi^{-1}\cdot\mathrm{id})\mathcal T_av) = \omega_{\Pi_p}(\pi^{-1}\cdot\mathrm{id})f(\mathcal T_av) = \delta_{\Pi_p}\pi_{\ell}^{w(\xi)}f(\mathcal T_av).$$
In particular, if $\tau$ acts like $x\cdot\mathrm{id}$ on the cohomology of $\mathcal M^{\mathrm{an}}$ for some $x\in \overline{\mathbb Q_{\ell}}^{\times}$, then $\mathrm{Frob}$ acts by multiplication by $\delta_{\Pi_p}\pi_{\ell}^{w(\xi)}x^{-1}p^{n-1}$ on the corresponding $\mathrm{Ext}$ groups.

\paragraph{}\label{H0Shimura} If $x \in \overline{\mathbb Q_{\ell}}^{\times}$, let $\overline{\mathbb Q_{\ell}}[x]$ denote the $1$-dimensional representation of $W$ where the inertia acts trivially and $\mathrm{Frob}$ acts like multiplication by the scalar $x$. Let $X^{\mathrm{un}}(J)$ denote the set of unramified characters of $J$. Looking at the diagonal $a+b=0$ in the spectral sequence, we obtain the following result.

\noindent \begin{prop}
There is a $(G(\mathbb A_f^p)\times W)$-equivariant isomorphism 
$$\mathrm{H}^{0}_c(\overline{S}^{\mathrm{ss}}, \overline{\mathcal L_{\xi}}) \simeq \bigoplus_{\substack{\Pi\in\mathcal A_{\xi}(I) \\ \Pi_p \in X^{\mathrm{un}}(J)}} \Pi^p \otimes \overline{\mathbb Q_{\ell}}[\delta_{\Pi_p}\pi_{\ell}^{w(\xi)}].$$
\end{prop}

\noindent The proof is identic to \cite{muller} 5.1.6 Proposition. It follows from the following facts. First, $F_2^{0,0}$ is the only non-zero term on the diagonal $a+b=0$ and all differentials connected to $F_k^{0,0}$ are zero for $k\geq 2$. Then, by \ref{HighestCohomologyGroup} we have $\mathrm H_c^{2(n-1)}(\mathcal M^{\mathrm{an}}) \simeq \mathrm{c-Ind}_{J^{\circ}}^J \,\mathbf 1$ with $\tau$ acting by multiplication by $p^{n-1}$. Eventually, since $J^{\circ}$ is normal in $J$, $J/J^{\circ} \simeq \mathbb Z$ and $J^{\circ}$ is generated by all the compact open subgroups of $J$, any smooth irreducible representation of $J$ having some non-zero $J^{\circ}$-fixed vectors is actually an unramified character of $J$.

\subsection{The cohomology when $n=2$ with $C$ split, when $n=3$ and when $n=4$ with $C$ non-split}

\paragraph{}\label{EigenvaluesOnH2} From now on we assume that $\theta_{\mathrm{max}} = 1$ so that $n=2$ with $C$ split, $n=3$ or $n=4$ with $C$ non-split. We will use our knowledge of the spectral sequence given by $E_1^{a,b}$ as detailed in section 3.2 in order to compute all the terms $F_2^{a,b}$, and as a consequence we obtain a description of the cohomology of the supersingular locus of the Shimura variety. All the arguments used here are exactly the same as in \cite{muller} Section 5.2.\\
Since $\theta_{\mathrm{max}} = 1$, we have $F^{a,b}_2 = 0$ for all $a\geq 2$. As a consequence, all differentials in the second and deeper pages of the sequence are zero, so that it already degenerates on the second page. Moreover, the supersingular locus $\overline{S}^{\mathrm{ss}}$ has dimension one, thus $F_2^{0,b} = 0$ for $b\geq 3$ and $F_2^{1,b} = 0$ for $b\geq 2$. \\
In Figure 4, we draw the second page $F_2$ and we write between brackets the \textit{complex modulus} of the possible eigenvalues of $\mathrm{Frob}$ on each term. as computed in \ref{FrobeniusActionOnHom}.

\begin{figure}[h]
\centering
\begin{tikzcd}
F_2^{0,2}[p^{1+w(\xi)/2},p^{w(\xi)/2}] & 0 \\
F_2^{0,1}[p^{w(\xi)/2}] & F_2^{1,1}[p^{w(\xi)/2}]\\
F_2^{0,0}[p^{w(\xi)/2}] & F_2^{1,0}[p^{w(\xi)/2}]
\end{tikzcd}
\caption{The second page $F_2$ with the complex modulus of possible eigenvalues of $\mathrm{Frob}$ on each term.}
\end{figure}

\begin{prop}
We have $F_2^{1,1} = 0$ and the eigenspaces of $\mathrm{Frob}$ on $F_2^{0,2}$ attached to any eigenvalue of complex modulus $p^{w(\xi)/2}$ are zero.
\end{prop}

\begin{proof}
By the machinery of spectral sequences, $F_2^{1,1}$ is a $\mathbb G(\mathbb A_f^p)\times W$-subspace of $\mathrm H_c^2(\overline{S}^{\mathrm{ss}}_{K^p},\overline{\mathcal L_{\xi}})$, and the quotient by $F_2^{1,1}$ is isomorphic to $F_2^{0,2}$. We prove that no eigenvalue of $\mathrm{Frob}$ on this $\mathrm H_c^2$ cohomology group has complex modulus $p^{w(\xi)/2}$, and the result readily follows.\\
Let $K^p \subset \mathbb G(\mathbb A_f^p)$ be small enough. Recall from \ref{BTStrataOnShimura} the definition of the Bruhat-Tits strata $\overline{\mathrm S}_{K^p,\Lambda,k}$ inside the supersingular locus $\overline{S}^{\mathrm{ss}}_{K^p}$. Each stratum $\overline{\mathrm S}_{K^p,\Lambda,k}$ is isomorphic to $\mathcal M_{\Lambda}$. For $\Lambda \in \mathcal L_i$, define 
$$\mathcal M_{\Lambda}^{\circ} := \mathcal M_{\Lambda} \setminus \bigcup_{\Lambda' \subsetneq \Lambda} \mathcal M_{\Lambda'},$$
where $\Lambda'$ runs over all vertex lattices of $\mathcal L_i$ strictly contained in $\Lambda$. By \cite{RTW} Theorem 6.10, each $\mathcal M_{\Lambda}^{\circ}$ is open dense in $\mathcal M_{\Lambda}$. Via the isomorphism $\mathcal M_{\Lambda} \xrightarrow{\sim} S_{\theta}$ where $t(\Lambda) = 2\theta$, we have $\mathcal M_{\Lambda}^{\circ} \xrightarrow{\sim} X_{I_{\theta}}(w_{\theta})$ with the notations of \ref{Stratification}. In particular, $\mathcal M_{\Lambda}^{\circ}$ is isomorphic to the Coxeter variety for $\mathrm{Sp}(2\theta,\mathbb F_p)$. Let $\overline{\mathrm S}_{K^p,\Lambda,k}^{\circ} \subset \overline{\mathrm S}_{K^p,\Lambda,k}$ be the scheme theoric image of $\mathcal M_{\Lambda}^{\circ}$ in the $k$-th copy of $\mathcal M_{\mathrm{red}}$ via the $p$-adic uniformization $(\Theta_{K^p})_s$ of \ref{BTStrataOnShimura}. We have a stratification 
$$\overline{S}^{\mathrm{ss}}_{K^p} = \overline{S}^{\mathrm{ss}}_{K^p}[0] \sqcup \overline{S}^{\mathrm{ss}}_{K^p}[1],$$
where for $i= 0,1$ the stratum $\overline{S}^{\mathrm{ss}}_{K^p}[i]$ is the finite disjoint union of the $\overline{\mathrm S}_{K^p,\Lambda,k}^{\circ}$ for various $k$ and $\Lambda$ of orbit type $t(\Lambda) = 2i$. The stratum $\overline{S}^{\mathrm{ss}}_{K^p}[0]$ is closed of dimension $0$, and the stratum $\overline{S}^{\mathrm{ss}}_{K^p}[1]$ is open dense of dimension $1$. Therefore, we have an isomorphism between the highest degree cohomology groups 
$$\mathrm H_c^2(\overline{S}^{\mathrm{ss}}_{K^p}, \overline{\mathcal L_{\xi}}) \simeq \mathrm H_c^2(\overline{S}^{\mathrm{ss}}_{K^p}[1], \overline{\mathcal L_{\xi}}).$$
Since $\overline{S}^{\mathrm{ss}}_{K^p}[1] = \bigsqcup_{t(\Lambda) = 2,k} \overline{\mathrm S}_{K^p,\Lambda,k}^{\circ}$ and each $\overline{\mathrm S}_{K^p,\Lambda,k}^{\circ}$ is open and closed in $\overline{S}^{\mathrm{ss}}_{K^p}[1]$, we have 
$$\mathrm H_c^2(\overline{S}^{\mathrm{ss}}_{K^p}[1], \overline{\mathcal L_{\xi}}) \simeq \bigoplus_{t(\Lambda) = 2,k} \mathrm H_c^2(\overline{\mathrm S}_{K^p,\Lambda,k}^{\circ},\overline{\mathcal L_{\xi}}) \simeq \bigoplus_{t(\Lambda) = 2,k} \mathrm H_c^2(\overline{\mathrm S}_{K^p,\Lambda,k},\overline{\mathcal L_{\xi}}),$$
where the last isomorphism follows from the stratification 
$$\overline{\mathrm S}_{K^p,\Lambda,k} = \overline{\mathrm S}_{K^p,\Lambda,k}^{\circ} \sqcup \bigsqcup_{\Lambda'\subsetneq \Lambda} \overline{\mathrm S}_{K^p,\Lambda',k}^{\circ},$$
with the first term being open dense of dimension $1$ and the second term being closed of dimension $0$. Since $\mathcal L_{\xi} \simeq \epsilon(\xi)\epsilon_{m(\xi)} \left( \mathrm{R}^{m(\xi)}({\pi_{K^p}^{m(\xi)}})_{*}\overline{\mathbb Q_{\ell}}(t(\xi))\right)$, the local system $\overline{\mathcal L_{\xi}}$ is pure of weight $w(\xi)$. Moreover the variety $\overline{\mathrm S}_{K^p,\Lambda,k}$ is projective and smooth for $\theta = 1$ (it is actually isomorphic to $\mathbb P^1$ by \ref{IsomorphismDLVariety} Proposition). Hence, all eigenvalues of the Frobenius on $H_c^2(\overline{\mathrm S}_{K^p,\Lambda,k},\overline{\mathcal L_{\xi}})$, and therefore on $\mathrm H_c^2(\overline{S}^{\mathrm{ss}}_{K^p},\overline{\mathcal L_{\xi}})$ as well, must have complex modulus $p^{1+t(\xi)/2}$.
\end{proof}

\paragraph{}\label{VanishingExt1} We will need the following 

\begin{lem}
For $i=1$ when $n=2$ with $C$ split, for all $1\leq i \leq p$ when $n=3$ and for all $1\leq i \leq p^2$ when $n=4$ with $C$ non-split, we have $\mathrm{Ext}_{J}^1(\mathrm{Ker}(\varphi_i),\pi) = \mathrm{Ext}_{J}^1(\mathrm{Im}(\varphi_i),\pi) = 0$ for any smooth representation $\pi$ of $J$.
\end{lem}

\noindent The proof is strictly identic to \cite{muller} 5.2.5 Proposition. It relies on the observation of \ref{SpectralSequenceIsInduced} that the spectral sequence $E$, computing the cohomology of the Rapoport-Zink space $\mathcal M^{\mathrm{an}}$, is induced from $J^{\circ}$.

\paragraph{}\label{CohomologySupersingular}The remaining of this section is devoted to proving the following theorem. We use the same notations as \ref{H0Shimura}. 

\begin{theo}
There are $G(\mathbb A_f^p) \times W$-equivariant isomorphisms
\begin{align*}
\mathrm{H}^{0}_c(\overline{S}^{\mathrm{ss}}, \overline{\mathcal L_{\xi}}) & \simeq \bigoplus_{\substack{\Pi\in\mathcal A_{\xi}(I) \\ \Pi_p \in X^{\mathrm{un}}(J)}} \Pi^p \otimes \overline{\mathbb Q_{\ell}}[\delta_{\Pi_p}\pi_{\ell}^{w(\xi)}], \\
\mathrm{H}^{2}_c(\overline{S}^{\mathrm{ss}}, \overline{\mathcal L_{\xi}}) & \simeq \bigoplus_{\substack{\Pi\in\mathcal A_{\xi}(I) \\ \Pi_p^{J_1}\not = 0}} \Pi^p \otimes \overline{\mathbb Q_{\ell}}[\delta_{\Pi_p}\pi_{\ell}^{w(\xi)+2}].
\end{align*}
Moreover, there exists a $G(\mathbb A_f^p)\times W$-subspace $V \subset \mathrm{H}^{1}_c(\overline{S}^{\mathrm{ss}}, \overline{\mathcal L_{\xi}})$ such that 
$$V \simeq \bigoplus_{\Pi\in\mathcal A_{\xi}(I)} d(\Pi_p)\Pi^p \otimes \overline{\mathbb Q_{\ell}}[\delta_{\Pi_p}\pi_{\ell}^{w(\xi)}],$$
and with quotient space isomorphic to 
$$\mathrm{H}^{1}_c(\overline{S}^{\mathrm{ss}}, \overline{\mathcal L_{\xi}})/V \simeq \bigoplus_{\substack{\Pi\in\mathcal A_{\xi}(I) \\ \Pi_p^{J_1}\not = 0\\ \dim(\Pi_p) > 1}} (\nu-1-d(\Pi_p))\Pi^p \otimes \overline{\mathbb Q_{\ell}}[\delta_{\Pi_p}\pi_{\ell}^{w(\xi)}] \oplus \bigoplus_{\substack{\Pi\in\mathcal A_{\xi}(I) \\ \Pi_p \in X^{\mathrm{un}}(J)}} (\nu-d(\Pi_p))\Pi^p \otimes \overline{\mathbb Q_{\ell}}[\delta_{\Pi_p}\pi_{\ell}^{w(\xi)}],$$
where $\nu \in \mathbb Z_{\geq 0}$ is a multiplicity given by 
$$
\nu = \begin{cases}
1 & \text{if } n = 2 \text{ and } C \text{ is split},\\
p & \text{if } n = 3,\\
p^2 & \text{if } n = 4 \text{ and } C \text{ is non-split},
\end{cases}
$$ 
and $d(\Pi_p) := \dim\mathrm{Ext}_{J}^1(\mathrm{c-Ind}_{J^{\circ}}^J\,\mathbf 1,\Pi_p)$.
\end{theo}

\noindent We note that the statement regarding $\mathrm H^0_c$ has already been proved in \ref{H0Shimura}. The integers $d(\Pi_p) \geq 0$ are finite, and all the multiplicities $\nu-1-d(\Pi_p)$ and $\nu-d(\Pi_p)$ occuring in the formula are non negative. In particular, if $n=2$ with $C$ split then $\nu-1 = 0$ so that for all $\Pi$ with $\Pi_p^{J_1} \not = 0$ and $\dim(\Pi_p)>1$, we have $d(\Pi_p) = 0$ and the corresponding component $\Pi_p$ does not actually occur in $\mathrm H_c^1(\overline{S}^{\mathrm{ss}}, \overline{\mathcal L_{\xi}})/V$. \\
For all $\Pi \in \mathcal A_{\xi}(I)$ we have $d(\Pi_p) \not = 0 \implies \Pi_p^{J_0} \not = 0 \text{ and } \Pi_p^{J_1} \not = 0$. In particular, if $\Pi_p$ is supercuspidal or if the central character of $\Pi_p$ is not unramified, then $d(\Pi_p) = 0$. All of these facts are byproducts of the proof below, but may also be proved by direct computation. We are not aware of a formula computing $d(\Pi_p)$ in general.

\begin{proof}We begin with the statement regarding $\mathrm H_c^2$. By \ref{EigenvaluesOnH2} we have 
$$\mathrm{H}^{2}_c(\overline{S}^{\mathrm{ss}}, \overline{\mathcal L_{\xi}}) \simeq F_2^{0,2} \simeq \bigoplus_{\Pi\in\mathcal A_{\xi}(I)} \mathrm{Hom}_{J} \left (E_2^{0,b}(1-n), \Pi_p\right) \otimes \Pi^p,$$
where $b=0$ if $n=2$ with $C$ split, $b=2$ if $n=3$ and $b=4$ if $n=4$ with $C$ non-split. In any case we have $E_2^{0,b} \simeq \mathrm{c-Ind}_{J_1}^J\,\mathbf 1$. Using Frobenius reciprocity we have 
$$F_2^{0,2} \simeq \bigoplus_{\Pi\in\mathcal A_{\xi}(I)} \mathrm{Hom}_J(\mathrm{c-Ind}_{J_1}^J \, \mathbf 1(1-n),\Pi_p)\otimes \Pi^{p} \simeq \bigoplus_{\Pi\in\mathcal A_{\xi}(I)} \mathrm{Hom}_{J_1}(\mathbf 1(1-n),(\Pi_p)_{|J_1})\otimes \Pi^{p}.$$
Since $J_1$ is a special maximal compact subgroup of $J$, we have $\dim(\pi^{J_1}) = 1$ for all smooth irreducible representations of $J$ such that $\pi^{J_1} \not = 0$. Therefore, we have 
$$\mathrm H_c^2(\overline{S}^{\mathrm{ss}}, \overline{\mathcal L_{\xi}}) \simeq \bigoplus_{\substack{\Pi\in\mathcal A_{\xi}(I) \\ \Pi_p^{J_1} \not = 0}} \Pi^p \otimes \overline{\mathbb Q_{\ell}}[\delta_{\Pi_p}\pi_{\ell}^{w(\xi)+2}]$$
as claimed, using \ref{FrobeniusActionOnHom} for the action of $\mathrm{Frob}$. \\

\noindent We now investigate the group $\mathrm H^1_c$. First we deal with the case $n=2$ with $C$ split. By the spectral sequence, there exists a $\mathbb G(\mathbb A_f^p)\times W$-subspace $V$ of this cohomology group such that 
$$V \simeq F_2^{1,0} \quad \text{and} \quad \mathrm{H}^{1}_c(\overline{S}^{\mathrm{ss}}, \overline{\mathcal L_{\xi}})/V \simeq F_2^{0,1}.$$
We have 
\begin{align*}
F_2^{1,0} & \simeq \bigoplus_{\Pi\in\mathcal A_{\xi}(I)} \mathrm{Ext}^1_{J} \left (\mathrm H_c^{2}(\mathcal M^{\mathrm{an}},\overline{\mathbb Q_{\ell}})(-1), \Pi_p\right) \otimes \Pi^p \\
& \simeq \bigoplus_{\Pi\in\mathcal A_{\xi}(I)} \mathrm{Ext}^1_{J} \left (\mathrm{c-Ind}_{J^{\circ}}^J\,\mathbf 1(-1), \Pi_p\right) \otimes \Pi^p\\
&\simeq \bigoplus_{\Pi\in\mathcal A_{\xi}(I)} d(\Pi_p)\Pi^p \otimes \overline{\mathbb Q_{\ell}}[\delta_{\Pi_p}\pi_{\ell}^{w(\xi)}],
\end{align*}
where the eigenvalues of $\mathrm{Frob}$ are given by \ref{FrobeniusActionOnHom}.\\
Let us now compute $F_2^{0,1} = \bigoplus_{\Pi\in\mathcal A_{\xi}(I)} \mathrm{Hom}_{J} \left (\mathrm H_c^{1}(\mathcal M^{\mathrm{an}},\overline{\mathbb Q_{\ell}})(-1), \Pi_p\right) \otimes \Pi^p$. Recall the map $\varphi_1$ appearing in Figure 2, section 3.2. By \ref{HighestCohomologyGroup}, we have a short exact sequence 
$$0 \rightarrow \mathrm{Im}(\varphi_1) \rightarrow \mathrm{c-Ind}_{J_1}^J \, \mathbf 1 \rightarrow \mathrm{c-Ind}_{J^{\circ}}^J \,\mathbf 1 \rightarrow 0.$$
If $K$ is a compact open subgroup of $J$ and if $\rho$ is any smooth representation of $K$, the compactly induced representation $\mathrm{c-Ind}_K^J\,\rho$ is projective in the category of smooth representations of $J$. In particular, we have $\mathrm{Ext}^1_J(\mathrm{c-Ind}_{J_1}^J \,\mathbf 1,\Pi_p) = 0$. Therefore we obtain the following exact sequence 
$$\hspace{-30pt} 0 \rightarrow \mathrm{Hom}_{J}(\mathrm{c-Ind}_{J^{\circ}}^J \,\mathbf 1,\Pi_p) \rightarrow \mathrm{Hom}_{J}(\mathrm{c-Ind}_{J_1}^J \, \mathbf 1,\Pi_p) \rightarrow \mathrm{Hom}_J(\mathrm{Im}(\varphi_1),\Pi_p) \rightarrow \mathrm{Ext}^1_J(\mathrm{c-Ind}_{J^{\circ}}^J \,\mathbf 1,\Pi_p) \rightarrow 0.$$
By Frobenius reciprocity, the left term is $\overline{\mathbb Q_{\ell}}$ if $\Pi_p^{J^{\circ}} \not = 0$ and $0$ otherwise. As we have already observed in \ref{H0Shimura} we have $\Pi_p^{J^{\circ}} \not = 0 \iff \Pi_p \in X^{\mathrm{un}}(J)$. The second term from the left is $\overline{\mathbb Q_{\ell}}$ if $\Pi_p^{J_1} \not = 0$ and $0$ otherwise. It follows that 
$$\mathrm{Hom}_{J} \left (\mathrm{Im}(\varphi_1), \Pi_p\right) \simeq 
\begin{cases}
\overline{\mathbb Q_{\ell}}^{d(\Pi_p)} & \text{if } \Pi_p \in X^{\mathrm{un}}(J) \text{ or } \Pi_p^{J_1} = 0,\\
\overline{\mathbb Q_{\ell}}^{d(\Pi_p)+1} & \text{if } \Pi_p \not \in X^{\mathrm{un}}(J) \text{ and } \Pi_p^{J_1} \not = 0,
\end{cases}$$
Consider the following short exact sequence
$$0 \rightarrow \mathrm{Ker}(\varphi_1) \rightarrow \mathrm{c-Ind}_{J_0}^J\,\mathbf 1 \rightarrow \mathrm{Im}(\varphi_1) \rightarrow 0.$$
By \ref{VanishingExt1} we have $\mathrm{Ext}^1_J(\mathrm{Im}(\varphi_1),\Pi_p) = 0$. Thus we obtain the following exact sequence
$$0 \rightarrow \mathrm{Hom}_J(\mathrm{Im}(\varphi_1),\Pi_p) \rightarrow \mathrm{Hom}_{J}(\mathrm{c-Ind}_{J_0}^J\,\mathbf 1,\Pi_p) \rightarrow \mathrm{Hom}(\mathrm{Ker}(\varphi_1),\Pi_p)\rightarrow 0.$$
Since $J_0$ is also a special maximal compact subgroup of $J$, the middle term is $\overline{\mathbb Q_{\ell}}$ if $\Pi_p^{J_0} \not = 0$ and $0$ otherwise. The injectivity of the left arrow implies that $d(\Pi_p) > 0 \implies \Pi_p^{J_0} \not = 0$, and that $(\Pi_p \not\in X^{\mathrm{un}}(J) \text{ and } \Pi_p^{J_1} \not = 0) \implies \Pi_p^{J_0} \not = 0$. Since unramified character are obviously trivial on $J_0$, we actually have $\forall \Pi \in \mathcal A_{\xi}(I), \Pi_p^{J_1} \implies \Pi_p^{J_0} \not = 0$. In fact, this is an equivalence for the following reason (we use the same arguments as in the proof of \cite{muller} 5.2.6 Theorem).\\
Let $\Pi\in \mathcal A_{\xi}(I)$ such that $\Pi_p^{J_0} \not = 0$. Assume towards a contradiction that $\Pi_p^{J_1} = 0$. Let $B$ be a minimal parabolic subgroup of $J$ with Levi complement $T$. Since $J_0$ is a special maximal compact subgroup of $J$, there exists an unramified character $\chi$ of $T$ such that $\Pi_p$ is the unique $J_0$-spherical composition factor of the normalized parabolic induction $\iota_B^J \chi$. Since $J_1$ is special as well, $\iota_B^J \chi$ also contains a unique composition factor which is $J_1$-spherical. By hypothesis, it can not be $\Pi_p$. Therefore, $\iota_B^J \chi$ must have composition length $2$, and the $J_1$-spherical composition factor is given by the Aubert-Zelevinski dual $D(\Pi_p)$. Now, $D(\Pi_p)$ is also the $p$-component of some automorphic representation $\Pi' \in \mathcal A_{\xi}(I)$. Since $(\Pi'_p)^{J_1} \not = 0$ we must have $(\Pi'_p)^{J_0} \not = 0$ by the direct implication, which is a contradiction by unicity of the $J_0$-spherical composition factor of $\iota_B^T \chi$.\\
Therefore, we have proved that $\forall \Pi\in\mathcal A_{\xi}(I), \Pi_p^{J_1} \not = 0 \iff \Pi_p^{J_0} \not = 0.$ In particular, $d(\Pi_p) > 0 \implies \Pi_p^{J_1} \not = 0$. Using this fact and the short exact sequence above, we obtain 
$$\mathrm{Hom}_{J} \left (\mathrm{Ker}(\varphi_1), \Pi_p\right) \simeq 
\begin{cases}
0 & \Pi_p^{J_1} = 0,\\
\overline{\mathbb Q_{\ell}}^{1 - d(\Pi_p)}   & \text{if } \Pi_p \in X^{\mathrm{un}}(J),\\
0 & \text{if } \Pi_p \not \in X^{\mathrm{un}}(J) \text{ and } \Pi_p^{J_1} \not = 0.
\end{cases}$$
By the shape of the spectral sequence in Figure 2, we have $\mathrm{Ker}(\varphi_1) \simeq \mathrm H^1_c(\mathcal M^{\mathrm{an}})$ with $\mathrm{Frob}$ acting via multiplication by $p$. Thus, $\mathrm{Hom}_{J} \left (\mathrm{Ker}(\varphi_1), \Pi_p\right)$ computes $F_2^{0,1}$ and the result follows.\\

\noindent We now assume that $n=3$ or that $n=4$ with $C$ non-split. Both cases are similar except that one must be careful to shift the degrees and eigenvalues of the Frobenius suitably. As in the case $n=2$ with $C$ split, there is a $\mathbb G(\mathbb A_f^p)\times W$-subspace $V$ of this cohomology group such that 
$$V \simeq F_2^{1,0} \quad \text{and} \quad \mathrm{H}^{1}_c(\overline{S}^{\mathrm{ss}}, \overline{\mathcal L_{\xi}})/V \simeq F_2^{0,1},$$
and we have 
$$F_2^{1,0} \simeq \bigoplus_{\Pi\in\mathcal A_{\xi}(I)} d(\Pi_p)\Pi^p \otimes \overline{\mathbb Q_{\ell}}[\delta_{\Pi_p}\pi_{\ell}^{w(\xi)}].$$
Let us now compute $F_2^{0,1}$. Recall the maps $\varphi_i$ from Figure 3, where $1\leq i \leq p$ when $n=3$ and $1\leq i \leq p^2$ when $n=4$ with $C$ non-split. By the same arguments than in the case $n=2$ with $C$ split, we have 
$$\mathrm{Hom}_{J} \left (\mathrm{Im}(\varphi_1), \Pi_p\right) \simeq 
\begin{cases}
\overline{\mathbb Q_{\ell}}^{d(\Pi_p)} & \text{if } \Pi_p \in X^{\mathrm{un}}(J) \text{ or } \Pi_p^{J_1} = 0,\\
\overline{\mathbb Q_{\ell}}^{d(\Pi_p)+1} & \text{if } \Pi_p \not \in X^{\mathrm{un}}(J) \text{ and } \Pi_p^{J_1} \not = 0,
\end{cases}$$
Next, consider the short exact sequence 
$$0 \rightarrow \mathrm{Ker}(\varphi_1) \rightarrow (\mathrm{c-Ind}_{J_0}^J\,\mathbf 1)^{k_{2,0}} \rightarrow \mathrm{Im}(\varphi_1) \rightarrow 0.$$
By \ref{VanishingExt1} we have $\mathrm{Ext}^1_J(\mathrm{Im}(\varphi_1),\Pi_p) = 0$. Thus we obtain the following exact sequence
$$0 \rightarrow \mathrm{Hom}_J(\mathrm{Im}(\varphi_1),\Pi_p) \rightarrow \mathrm{Hom}_{J_0}(\mathbf 1,(\Pi_p)_{|J_0})^{k_2,0} \rightarrow \mathrm{Hom}(\mathrm{Ker}(\varphi_1),\Pi_p)\rightarrow 0.$$
Exactly as in the case $n=2$ with $C$ split, one may deduce that $\forall \Pi\in \mathcal A_{\xi}(I), (\Pi_p)^{J_1}\not = 0 \iff (\Pi_p)^{J_0}\not = 0$, and $d(\Pi_p) > 0 \implies \Pi_p^{J_1}\not = 0$. Therefore, we have 
$$
\mathrm{Hom}_J(\mathrm{Ker}(\varphi_1),\Pi_p) \simeq \begin{cases}
0 & \text{if } \Pi_p^{J_1} = 0,\\
\overline{\mathbb Q_{\ell}}^{k_{2,0}} & \text{if } \Pi_p \in X^{\mathrm{un}}(J),\\
\overline{\mathbb Q_{\ell}}^{k_{2,0}-1} & \text{if } \Pi_p^{J_1} \not = 0 \text{ and } \Pi_p \not \in X^{\mathrm{un}}(J).
\end{cases}
$$
Let us write $b=4$ if $n=3$ and $b=6$ if $n=4$ with $C$ non-split. Since we have 
$$\mathrm{Hom}_J(E_2^{-a,b},\Pi_p) = 0$$
for every $a\geq 2$, using \ref{VanishingExt1} repeatedly, one may prove by induction that 
$$\mathrm{Hom}_J(\mathrm{Im}(\varphi_2)) \simeq 
\begin{cases}
\overline{\mathbb Q_{\ell}}^s & \text{if } \Pi_p^{J_1} \not = 0,\\
0 & \text{otherwise,}
\end{cases}$$
where $s$ is the alternate sum of the multiplicities $k_{i,0}$ for $i$ running from $3$ to $p+1$ if $n=3$, and to $p^2+1$ if $n=4$. Eventually, let us consider the short exact sequence
$$0 \to \mathrm{Im}(\varphi_2) \to \mathrm{Ker}(\varphi_1) \to E_2^{-1,b} \to 0.$$
By \ref{EigenvaluesOnH2} we know that $F_2^{1,1}=0$. This imply that $\mathrm{Ext}_J^1(E_2^{-1,b},\Pi_p) = 0$. We deduce the following short exact sequence 
$$0 \rightarrow \mathrm{Hom}_J(E_2^{-1,b},\Pi_p) \rightarrow \mathrm{Hom}_J(\mathrm{Ker}(\varphi_1),\Pi_p) \rightarrow \mathrm{Hom}_J(\mathrm{Im}(\varphi_2),\Pi_p)\rightarrow 0.$$
From this, we deduce that 
$$\mathrm{Hom}_{J}(E_2^{-1,b}, \Pi_p) \simeq 
\begin{cases}
0 & \text{if } \Pi_p^{J_1} = 0,\\
\overline{\mathbb Q_{\ell}}^{k_{2,0} - s - d(\Pi_p)} & \text{if } \Pi_p \in X^{\mathrm{un}}(J),\\
\overline{\mathbb Q_{\ell}}^{k_{2,0} - s - d(\Pi_p) - 1} & \text{if }  \Pi_p^{J_1} \not = 0 \text{ and } \Pi_p \not \in X^{\mathrm{un}}(J).
\end{cases}$$
Since $E_2^{-1,b} \simeq \mathrm H_c^{b-1}(\mathcal M^{\mathrm{an}})$, the term $F_2^{0,1}$ is computed by the $\mathrm{Hom}$ above. It remains to compute  
$$k_{2,0} - s = \begin{cases}
\sum_{i=2}^{p+1} (-1)^i {{p+1}\choose i} = p & \text{if }n=3,\\
\sum_{i=2}^{p^2+1} (-1)^i {{p^2+1}\choose i} = p^2 & \text{if }n=4 \text{ with } C \text{ non-split}.
\end{cases}$$
Lastly, we observe that $\forall \Pi \in \mathcal A_{\xi}(I), (\Pi_p \not \in X^{\mathrm{un}}(J) \text{ and } \Pi_p^{J_1} \not = 0) \iff (\dim(\Pi_p) > 1 \text{ and } \Pi_p^{J_1} \not = 0)$. Indeed, the reverse implication is obvious. The direct implication follows given that if $\Pi_p^{J_1} \not = 0$ then $\Pi_p^{J_0} \not = 0$. If furthermore $\Pi_p$ is a character, it is then trivial on the subgroup generated by all the conjugates of $J_0$ and of $J_1$. By general theory, this group is no other than $J^{\circ}$. Therefore $\Pi_p \in X^{\mathrm{un}}(J)$. Thus, we have proved that 
$$F_2^{0,1} \simeq \bigoplus_{\substack{\Pi\in\mathcal A_{\xi}(I) \\ \Pi_p^{J_1}\not = 0\\ \dim(\Pi_p) > 1}} (\nu-1-d(\Pi_p))\Pi^p \otimes \overline{\mathbb Q_{\ell}}[\delta_{\Pi_p}\pi_{\ell}^{w(\xi)}] \oplus \bigoplus_{\substack{\Pi\in\mathcal A_{\xi}(I) \\ \Pi_p \in X^{\mathrm{un}}(J)}} (\nu-d(\Pi_p))\Pi^p \otimes \overline{\mathbb Q_{\ell}}[\delta_{\Pi_p}\pi_{\ell}^{w(\xi)}],$$
with $\nu = p$ when $n=3$ and $\nu = p^2$ when $n=4$ with $C$ non-split. This concludes the proof.
\end{proof}

\phantomsection
\printbibliography[heading=bibintoc, title={Bibliography}]
\markboth{Bibliography}{Bibliography}

\end{document}